\documentclass[a4paper, 12p]{article}
\usepackage{amsmath}
\usepackage{amssymb}
\usepackage{amsthm, mathtools}
\usepackage{enumerate}
\usepackage[shortlabels]{enumitem}
\usepackage{comment}
\usepackage{bbm}
\usepackage{graphicx} 
\usepackage[usenames,dvipsnames,svgnames,table]{xcolor}
\usepackage[colorlinks=true,
            linkcolor=RoyalBlue,
            urlcolor=Mulberry,
            citecolor=ForestGreen, hypertexnames=false]{hyperref}
\numberwithin{equation}{section}

\usepackage[nameinlink, capitalize]{cleveref}
\usepackage{hyperref}
\usepackage[backend=biber,style=alphabetic,sorting=anyt,maxbibnames=99, backref=true, backrefstyle=none]{biblatex}
\usepackage[margin=1in]{geometry}
\usepackage{microtype}

\newcommand{\Exp}{\mathbb{E}}
\newcommand{\FF}{F}

\newcommand{\Omproj}{\overline{\Omega}}

\newcommand{\piproj}{\overline{\pi}}
\newcommand{\Pproj}{\overline{P}}
\newcommand{\Qproj}{\overline{Q}}

\newcommand{\Omorproj}{\hat{\Omega}}
\newcommand{\Porproj}{\hat{P}}

\newcommand{\piorproj}{\hat{\pi}}

\newcommand{\flo}[2]{F_{#1 \rightarrow #2}}
\newcommand{\var}{\operatorname{Var}}
\newcommand{\dir}{\mathcal{E}}

\newcommand{\ent}{\operatorname{Ent}}

\newcommand{\rar}{\rightarrow}

\newcommand{\Om}{\Omega}

\def\RRp{\mathbb{R}_{\ge 0}}

\newtheorem{theorem}{Theorem}[section]
\newtheorem{lemma}[theorem]{Lemma}
\newtheorem{corollary}[theorem]{Corollary}
\newtheorem{remark}[theorem]{Remark}

\newtheorem{definition}[theorem]{Definition}
\newtheorem{fact}[theorem]{Fact}

\usepackage{stmaryrd}
\DeclarePairedDelimiter\set{\lbrace}{\rbrace}
\DeclarePairedDelimiter\parens{\lparen}{\rparen}
\DeclarePairedDelimiter\sqbr{[ }{]}

\DeclarePairedDelimiter\norm{\|}{\|}

\DeclarePairedDelimiter\inpr{\langle}{\rangle}
\def\Tmix{\tau_{\mathtt{mix}}}
\usepackage{boxedminipage}

\DeclareMathOperator{\Gap}{\mathtt{gap}}

\DeclareMathOperator{\LS}{\mathtt{ls}}
\def\TV{\mathtt{TV}}

\title{Faster Mixing for Triangulations via Transport Flows}
\author{Vedat Levi Alev\thanks{University of Haifa --- \href{mailto:vedatalev@math.haifa.ac.il}{vedatalev@math.haifa.ac.il}; Supported by the ISF Grant No. 721/2024 of Uriya A. First.}, Daniel Frishberg\thanks{California Polytechnic State University
---  \href{mailto:dfrishbe@calpoly.edu}{dfrishbe@calpoly.edu}}, Michail Sarantis, and Prasad Tetali\thanks{Carnegie Mellon University --- \href{mailto:ptetali@cmu.edu}{ptetali@cmu.edu}; Supported in part by the NSF DMS-2151283 grant
and Alexander M. Knaster Professorship.}}
\date{\today}

\begin{document}

\maketitle

\begin{abstract}
We prove an~$\widetilde O(n^2)$ bound for the \emph{relaxation time} and the \emph{log-Sobolev time} (inverse log-Sobolev constant) of the classical triangulation flip chain on a convex $(n+2)$-gon, implying a mixing time of $\widetilde O(n^2)$. The previous state of the art for the mixing time of this chain due to \citeauthor{eppstein2022improved} \cite{eppstein2022improved} was $\widetilde O(n^3)$, while the best known lower bound on the mixing time due to \citeauthor{mrs}~\cite{mrs} is~$\Omega(n^{3/2})$. Our relaxation time bound makes significant progress towards Aldous' \cite{aldousconj} conjectured bound of $\Theta(n^{3/2})$ for the relaxation time.

We improve upon the analysis of~\cite{eppstein2022improved} by further developing the framework of \emph{transport flows} introduced in~the work \cite{weimingjs} of \citeauthor{weimingjs}. In this light, our results can be seen as a more efficient way of using combinatorial decompositions to obtain functional inequalities for Markov chains. We hope our ideas will find other applications in the future. 
\end{abstract}

\section{Introduction}

\paragraph{Motivation.}

Local-to-global methods for bounding the mixing times of Markov chains \cite{kaufman2020high, alevlau, guo2020local, chen2021optimal, anari2022entropic} have enjoyed immense success in recent years in conjunction with the spectral independence framework of \cite{anarihardcore, chen2021rapidcol, feng2022rapid, anari2022entropic}. The basic idea intrinsic to these techniques is imposing on the state space a \emph{simplicial complex} structure and analyzing the \emph{global} chain by way of analyzing smaller \emph{local} chains, which are easier for analysis. While the idea of \emph{decomposing} the state space hierarchically and using intermediate chains to establish a log-Sobolev or Poincar\'e inequality is a classical idea in the study of Markov chains \cite{kaibelexp, eppstein2022improved, eppstein2021rapid, heinrich2020glauber, zongchentw,jstv,jstvimproved,federmihail}, the ideas involved in the classical works are typically of a combinatorial nature. In contrast, the ideas employed by the local-to-global framework are often more analytical or algebraic.

In this work, we adapt the local-to-global machinery to a well-studied chain on triangulations of a convex polygon. This walk is one of the most well-known random walks on the so-called \emph{Catalan structures} (see~\cite{aldousclad, aldousconj, cohenthesis, ardilamatroid, wilsonlozenge, mrs, mct, eppstein2022improved}). Catalan structures are a class of naturally defined combinatorial objects. These include but are not limited to \emph{Dyck paths} (paths in a lattice that stay above a particular line); \emph{balanced} sequences of ones and zeroes; \emph{non-crossing chord diagrams}; \emph{integer partitions}; \emph{binary plane trees}; and \emph{triangulations} of a convex polgyon. Chains on generalizations of Catalan structures have also been studied~\cite{eppstein2021rapid, caraceni2020, ncstmixing}.

A few of the Catalan structures (and their generalizations) are susceptible to recent methods using the theory of simplicial complexes (e.g.~the Catalan matroid studied by~\cite{ardilamatroid}). Others, such as the case of triangulations (and the isomorphic chain on binary search trees), have proven resistant to these methods.

In contrast, a few results over the past several years have achieved some success applying the classical method of \emph{canonical paths}, namely in the cases of triangulations~\cite{eppstein2022improved} and \emph{non-crossing spanning trees}~\cite{ncstmixing}.
The canonical paths technique, along with its generalization to multicommodity flows, is useful in bounding conductance or, when the paths are not too long, in directly establishing a Poincar\'e inequality~\cite{diacstroock,sinclairimproved}.

\paragraph{The Aldous conjecture.} A famous conjecture of \cite{aldous1994triangulating} states that the relaxation time of the \emph{triangulation flip walk} should be $\Theta(n^{3/2})$. \cite{mrs} proved that the mixing time is $\Omega(n^{3/2})$, but a matching upper bound still has not been found. \cite{mct} proved an $\widetilde O(n^5)$ upper bound\footnote{where we recall that $\widetilde\Omega(\bullet)$ and $\widetilde O(\bullet)$ hide polylogarithmic factors in $n$} via a comparison argument, using a tight result by Wilson \cite{wilsonlozenge} for the chain on \emph{Dyck paths}. \cite{eppstein2022improved} proved an $\widetilde O(n^3)$ upper bound, the best known result until this paper.

\paragraph{Transport flows.}
Recently~\cite{weimingjs} proved a refinement of the mixing result of \cite{jerrum1989approximating} for the natural random walk on perfect matchings using a technique they called \emph{transport flows}. This technique morally combines the ``best of both worlds'' from simplicial local-to-global methods and canonical paths. That is, in the case of both the recent local-to-global methods and also older decomposition techniques, one generally must consider the spectral gap of each intermediate localized walk in the hierarchy, and one incurs a multiplicative loss in the spectral gap at each level of the decomposition. This may lead to an undesirable blow-up in the relaxation time (inverse spectral gap) unless the structures considered are all strong expanders. In such a setting, it would be more desirable to trade a multiplicative loss for an additive one.   

While the result of \cite{weimingjs} trades the multiplicative loss for an additive one, using ideas similar to classical flow-based techniques, it is limited in the sense that the techniques only apply when the state space has a specific structure (a \emph{partite simplical complex}). One of our main contributions will be extending these ideas to a more general decomposition framework, more in line with classical decomposition techniques, e.g.~\cite{jstv} etc. 

\paragraph{Multi-way single-commodity flows.}
Inspired by the results of \cite{weimingjs}, we turn to a previous result proven in
\cite{eppstein2022improved}. This result used a multicommodity flow (a well-studied generalization of canonical paths) with bounded congestion in the state space of the triangulation walk, replacing the multiplicative loss with an additive loss using a construction analogous to transport flows. However, this combinatorial bound proved a bound on the conductance and consequently suffered a quadratic loss in bounding the \emph{spectral gap}, due to the celebrated Cheeger inequality, \cite{cheeger1, cheeger2}. 

A natural way to bound the spectral gap, without the aforementioned quadratic loss, is to appeal to the local-to-global methods. However, as the structures considered do not appear to meet the very strong expansion required of these methods, this is unlikely to work. 

\paragraph{Our contribution.}
We generalize the transport flow technique of \cite{weimingjs} from the setting of partite complexes to a more general decomposition framework. Na\"ively, a limitation of the flow-based techniques is the length of the paths used \cite{diacstroock}, and indeed this limitation also manifests in the techniques of \cite{weimingjs}. Inspired by \cite{montenegrovc}, we will avoid the quadratic loss and get a sharper spectral gap bound for the triangulation flip walk by studying the average congestion. These ideas will culminate in an $\widetilde O(n^2)$ relaxation time bound for the triangulation flip chain, which improves upon the relaxation time bound of $\widetilde O(n^3)$ proven in \cite{eppstein2022improved}.

To obtain an $\widetilde O(n^2)$ mixing time bound, we will prove a log-Sobolev inequality for the triangulation flip chain.  It is well known that a bound on the relaxation time can be converted to a bound on the mixing time, while suffering a logarithmic loss in the size of the state space \cite{diaconis1996logarithmic}. Usually, and indeed in the case of triangulations, this factor is unfortunately linear in the problem size $n$. Since our results are decomposition based, by appealing to this comparison when the compared chains are small enough, we only suffer a doubly logarithmic loss in the size of the state space, and consequently only a polylogarithmic loss in $n$.

This $\widetilde O(n^{2})$ mixing time bound on the triangulation flip chain improves the state of the art mixing time of $\widetilde O(n^3)$ by \cite{eppstein2022improved}, while also getting an improved result for the spectral gap and a novel log-Sobolev inequality. Our technical contribution can be thought of as a more efficient way of leveraging transport flow constructions, which we hope will inspire further research in the future.

Our main results are as follows.
\begin{theorem}[Functional Inequalities for the Flip Chain]
    \label{thm:lsi}
    The triangulation walk satisfies a log-Sobolev inequality with constant $\widetilde{\Omega}(n^{-2})$, and has spectral gap $\widetilde\Omega( n^{-2})$.
\end{theorem}

\begin{corollary}[Mixing Time for the Flip Walk]
    \label{cor:mix_flip}
    The mixing time of the triangulation walk is $\widetilde{O}(n^{2})$.
\end{corollary}
As mentioned before, our results avoid the Cheeger loss, by bounding the average congestion of the flow we analyze. This will follow by an analysis of the heights of binary trees and insights concerning the isomorphism between the triangulation walk and the natural \emph{rotation walk} on binary trees (see e.g.~\cite{sleator1988rotation,hiltonpedersen,cohenthesis}). 

Montenegro~\cite{montenegrovc} noted that the average vertex congestion is essentially equivalent to the expected path length in a multicommodity flow; so our result can be thought of as a dual version of \cite{weimingjs} utilizing the average congestion in place of the expected path length.

\section{Preliminaries}\label{sec:prelim}
\subsection{Random Walks and Mixing Times}
\label{sec:mcmc}
A random walk matrix $P \in \mathbb R^{\Omega \times \Omega}$ is a matrix with non-negative entries, all of whose rows sum to 1. Formally,
\[ \forall~x, y \in \Omega:~ P(x, y) \ge 0 \qquad \textrm{and}\qquad \forall~ x \in \Omega:~ \sum_{y \in \Omega} P(x, y) = 1.\]
A distribution $\pi: \Omega \to [0, 1]$ is called stationary for $P$ if, $\pi P = \pi$. The walk described by $P$ is reversible, if the following \emph{detailed balance conditions} hold:
\[ \pi(x) P(x, y) = \pi(y) P(y, x)\,,\quad\quad\forall~x, y \in \Omega.\]
We will also write $\alpha(P) = \min_{x \in \Omega} P(x, x)$ for the \emph{holding probability} of the random walk $P$ and $\pi^*~=~\min_{x \in \Omega} \pi(x)$ for the \emph{minimum measure} of $\pi$.

The $\epsilon$-mixing time of a random walk is defined to be the least time point such that after $t$ steps of random walk according to $P$, the distribution of the random walk is $\epsilon$-close to the stationary distribution regardless of the initial distribution, i.e.~
\[ \Tmix(P, \epsilon) = \min\set*{  t \in \mathbb N  ~\left|~\norm*{ \pi - \mu P^t}_{\TV} \le \epsilon\,,~\textrm{for all probability distributions}~\mu: \Omega \to [0,1]\right.}, \]
where $\norm*{\mu - \nu}_{\TV} = 1/2 \cdot \sum_{x \in \Omega} |\mu(x) - \nu(x)|$.

\subsection{Projection-Restriction and Product Chains}

Let $(\Om,P,\pi)$ be an ergodic Markov chain and $\Om=\bigcup_{t\in T}\Om_t$ a decomposition of its state space. When the stationary distribution $\pi$ is clear from context, we will simply write $(\Omega, P)$ in place of $(\Omega, P, \pi)$. We write $\bar\pi(t) = \pi(\Omega_t)$ for all $t \in T$ and define the \ref{eq:proj_chain_def} by the triple $(\Omproj,\Pproj, \bar\pi)$ where
\begin{equation}\tag{projection chain}\label{eq:proj_chain_def}
    \Pproj(t,t') = \frac{\sum_{x \in \Omega_t, y \in \Omega_t'} \pi(x) P(x, y)}{\sum_{z \in \Omega_t}\pi(z)}\,,
\end{equation}
when $t \neq t'$ and with self loops for the remaining probabilities.
That is, the probability of transitioning from class $t$ to $t'$ in the \ref{eq:proj_chain_def} is the probability we transition from any element of $\Om_t$ to any element of $\Om_{t'}$ in the original chain conditioned on being in $\Om_t$. Naturally, by defining
$\pi_t(x)=\frac{\pi(x)}{\pi(\Omega_t)}$
we can also define the \ref{eq:restr_chain_def} as the triple $(\Omega_t, P_t, \pi_t)$ where
\begin{equation}
\label{eq:restr_chain_def}\tag{restriction chain}
P_t(x, y) = \frac{P(x, y)}{\sum_{z \in \Omega_t} P(x, z)}\,,
\end{equation}
for all $x, y \in \Omega_t$.

Let $k$ be a positive integer, $(\Omega_, P_i, \pi_i),\ 1\leq i\leq k$ be Markov chains and $w$ a distribution on $k$. Consider the \ref{eqn:prod_transition} $(\Om,P,\pi)$ with $\Om=\prod_{i=1}^n \Om_i$ and
\begin{equation}\label{eqn:prod_transition}\tag{product chain}
P(x,y)=\sum_{i=1}^n w_i P_i(x_i,y_i)\prod_{j\ne i}\mathbbm{1}\{x_j=y_j\}.    
\end{equation}
Equivalently, the transition from every state consists of picking a coordinate in $[k]$ according to $w$ and then change this coordinate according to $(\Om_i,P_i)$. It is straightforward to verify that the stationary distribution of the product chain is $\pi=\bigotimes_{i=1}^n\pi_i$. We note that the \emph{graph} which underlies the \ref{eqn:prod_transition} corresponds to a weighted Cartesian product of the graphs which underlie the individual chains~$(\Omega, P_i, \pi_i)$.

\subsection{Variance, Entropy and Functional Inequalities}
\label{sec:funcineq}
Let $f:\Om\to\mathbb{R}$ be any function. Given a probability measure $\pi:\Omega \to [0,1]$ the
\ref{eq:var-def} functional $\var_\pi(f)$  is:
\begin{equation} \var_\pi(f) = \Exp_\pi[f^2] - \parens*{ \Exp_{\pi}
	f}^2 = \inpr*{ f, (I - J_\pi)
f}_\pi.\tag{variance}\label{eq:var-def}
\end{equation}
If $\bigcup_{t\in T}\Om_t$ is a decomposition of $\Om$, then the following equation is known as the \ref{eqn:total_var}
\begin{equation}\label{eqn:total_var}
    \var_{\pi}(f) = \var_{\piproj} F + \Exp_{t \sim \piproj} \var_{\pi_t} f,
    \tag{law of total variance}
\end{equation}
where $F:T\to\mathbb{R}$ is defined as $F(t)=\Exp_{\pi_t}f$ and $\piproj$ is the stationary distribution of the projection chain.

We will make use of the following consequence of Jensen's inequality,
\begin{lemma}\label{lem:convexity_dpi}
    Let $\pi$ be a probability distribution supported on $\Omega = \bigcup_{t \in T} \Omega_t$ (for disjoint $\Omega_t$) and $g : \Omega \to \RRp$ be a non-negative valued function.

    We define the probability measure $\bar\pi$ on $T$ by setting $\bar\pi(t) = \pi(\Omega_t)$ for each $t \in T$ and set $G : T \to \RRp$ 
    \[ G(t) = \Exp_{x \sim \pi}\sqbr*{g(x)~\mid~x \in \Omega_t}\]
    for each $t \in T$.
    Then, for any function $\beta: \RRp \to \RRp$ such that $\beta^2$ is convex, we have
    \[ \var_{\bar \pi} \beta(G) \le \var_\pi \beta(g).\]
\end{lemma}
\begin{proof}
    We have,
    \begin{align*}
        \var_{\bar\pi} \beta(G)
        &~=~\Exp_{t \sim \bar\pi} \beta^2(G(t)) - \parens*{\Exp_{t \sim \bar\pi} \beta(G(t))}^2\\
        &~=~\Exp_{t \sim \bar\pi} \sqbr*{\beta^2\parens*{\Exp_{x \sim \pi}\sqbr*{ g(x)~\mid~x \in \Omega_t}}} - \parens*{\Exp_{t \sim \bar \pi} \Exp_{x \sim \pi}\sqbr*{ \beta(g(x))~\mid~x \in \Omega_t} }^2\\
        &~\le~\Exp_{t \sim \bar\pi} \Exp_{x \sim \pi}\sqbr*{ \beta^2(g(x))~\mid~x \in \Omega_t} -  \parens*{\Exp_{t \sim \bar \pi} \Exp_{x \sim \pi}\sqbr*{ \beta(g(x))~\mid~x \in \Omega_t} }^2\\
        &~=~\var_{\pi} \beta(g)\,,
    \end{align*}
    where for the inequality we have used Jensen's inequality, and the final equality is due to the law of total expectation.
\end{proof}

The second functional we will need is that of entropy. If $f: \Omega \to \mathbb R_{\ge 0}$ is a non-negative function, the \ref{eq:ent_def} $\ent_\pi(f)$ of $f$ is
\begin{equation}\label{eq:ent_def}\tag{entropy}
    \ent_\pi(f) = \Exp_{\pi}\sqbr*{ f \log f} - \parens*{ \Exp_\pi f} \log \parens*{\Exp_\pi f}\,.
\end{equation}
Similarly to variance, we have the analogous \ref{eqn:total_ent} for a decomposition $\cup_{t\in T}\Om_t$ of the state space:
\begin{equation}\label{eqn:total_ent}
    \ent_{\pi}( f )= \ent_{\piproj} (F) + \Exp_{t \sim \piproj} \ent_{\pi_t}(f).
    \tag{law of total entropy}
\end{equation}

The following comparison between \ref{eq:var-def} and \ref{eq:ent_def} is well-known:
\begin{fact}[Theorem A.1, \cite{diaconis1996logarithmic}]\label{fac:varent_comp}
    Let $\pi$ be a distribution supported on $\Omega$. Then, writing $f^2: \Omega \to \mathbb R$ for the function obtained by $f^2(x) = (f(x))^2$, for all $f: \Omega \to \mathbb R$,
    \[ \frac{1 - 2\pi^*}{\log(1/\pi^* - 1)} \cdot \ent_\pi(f^2) \le \var_\pi(f).\]
\end{fact}

For a reversible $P \in \mathbb R^{\Omega \times \Omega}$ with
stationary measure of $\pi$, we define the \ref{eq:dirichlet_form} $\dir_P(f)$ of $f$ as follows:
\begin{equation}\label{eq:dirichlet_form}\tag{Dirichlet form} \dir_P(f) = \frac12 \cdot\sum_{x, y \in \Omega} \pi(x) P(x, y)( f(x) - f(y))^2\,.
\end{equation}

The Poincar\'e constant or the \ref{eq:gap-def} of $P$ and
is denoted by $\Gap(P)$ and is the solution to the following variational formula,
\begin{equation}\label{eq:gap-def}
	\Gap(P) = \min\set*{ \left. \frac{ \dir_P(f)
	}{\var_\pi(f) } ~\right|~ \var_\pi(f) \ne 0} = 1 -
	\lambda_2(P).\tag{spectral gap}
\end{equation}
The \ref{eq:logsob_def} $\LS(P)$ of $P$ is defined as the solution to the following variational formula,
\begin{equation}\label{eq:logsob_def}\tag{log-Sobolev constant}
    \LS(P) = \inf\set*{\left. \frac{\dir_P(f)}{\ent_\pi(f^2)}~\right|~\ent_\pi(f^2) \ne 0}.
\end{equation}
The following result is well-known, see e.g.~\cite{diacstroock}:
\begin{theorem}\label{thm:spec-mix}
    Let $P \in \mathbb R^{\Omega \times \Omega}$ be a self-adjoint row-stochastic matrix with stationary distribution $\pi$. Then,
    \[ \Tmix(P, \epsilon) \le \frac{1}{\Gap(P)} \cdot \log\parens*{\frac{1}{\sqrt{\epsilon \cdot \pi^\star}}}.\]
\end{theorem}
The following result shows that the \ref{eq:logsob_def} controls the mixing time in a very precise manner, 
\begin{theorem}[Corollary 2.4, \cite{montenegro2006mathematical}]\footnote{See the discussion following Corollary 2.4 for our precise statement.}\label{thm:lsi_mix}
    Let $P \in \mathbb R^{\Omega \times \Omega}$ holding probability $\alpha$ and with stationary distribution $\pi$. Then, there exists some absolute constant $C > 0$ such that for any $\epsilon > 0$:
    \[ \Tmix(P, \epsilon) \le \frac{C \alpha^{-1}}{\LS(P)} \cdot \parens*{\log \log \parens*{  \frac{1}{\pi^*}} + \log \frac{1}{\epsilon}}\,,\]
    where $C> 0$ does not depend on $\Omega, P, \pi$ or $\epsilon$ and $\pi^*$ is the minimum measure of $\pi$.
\end{theorem}
In our argument, it will be important to write the Poincar\'e and LSI constants of a \ref{eqn:prod_transition} in terms of the corresponding constants of its components. The relation between them is well-known and given by the following lemma: 
\begin{lemma}[Lemma 2.2.11, \cite{SC97}]\label{lem:product_ineq}
    Let $k$ be a positive integer, $w$ a distribution on $[k]$ and $(\Omega, P, \pi)$ the cartesian product of the chains $\{(\Omega_i, P_i) \mid i \in [k]\}$. Suppose further that each chain satisfies a Poincar\'e and a log-Sobolev inequality 
    \[
        \lambda_i \var_{\pi_i}( f ) \leq \dir_{\pi_i} (f) \qquad \text{and} \qquad \beta_i \ent_{\pi_i} (f^2) \leq \dir_{\pi_i}(f)\,,
    \]
    for all $f:\Om_i\to \mathbb{R}$ where $\lambda_i, \beta_i \ge 0$. Then the chain $(\Omega, P)$ satisfies the inequalities:
    \[
        \parens*{\min_{1\leq i\leq k} w_i\lambda_i} \var_{\pi}( f )\leq \dir_{\pi}(f) \qquad \text{and} \qquad \parens*{\min_{1\leq i\leq k} w_i\beta_i} \ent_{\pi} (f^2) \leq \dir_{\pi}(f)\,,
    \]
    for all $f:\Om\to\mathbb{R}$.
\end{lemma}

\subsection{Catalan Structures: Triangulations and Trees}
A \emph{triangulation} of a point set (in, say, the Euclidean plane) is a maximal collection of pairwise non-crossing edges connecting pairs of points. In the special case that the point set is convex, every triangulation of the point set includes the convex hull, and thus we will assume the point set is a convex polygon. Since the edges belonging to the convex hull are in every triangulation, we will identify a triangulation $x$ by its set of non-hull edges, and we will call these edges \emph{diagonals}.

Given a diagonal~$d$, it will be convenient for us to define the \emph{length} of~$d$ to be the number of edges in a shortest path consisting of polygon edges that connects the two endpoints of~$d$. 

One can view a triangulation naturally as a planar graph. The dual graph of a convex polygon triangulation is a binary tree. One can orient the polygon so that the dual tree is rooted, and therefore the number of triangulations of an $n+2$-gon is equal to the number of binary plane trees with $n$ nodes. This is known to be equal to the Catalan number~\cite{hiltonpedersen}

\[
    C_n = \frac{1}{n+1}\binom{2n}{n}\,.
\] 

Notice that the Catalan numbers satisfy the recurrence relation,
\[ C_n = \sum_{j = 1}^{n} C_{j-1} C_{n-j}\,, \]
where $C_0 = 1$ (see for example, \cite[Chapters 5 and 7]{concretemathematics}). With this, it is easy to observe that $C_n$ counts the number of rooted subtrees of the infinite binary tree on $n$ vertices, or equivalently, unlabelled rooted plane trees on $n$-vertices. Henceforth, we will refer to such trees as Catalan trees.

Using Stirling's formula, $C_n$ is seen to grow asymptotically as $\Theta\left(\frac{1}{n^{3/2}}\cdot 4^n\right)$.

The \emph{triangulation flip walk} (which we will also call the \emph{triangulation walk}) is the following random walk, defined with respect to the regular $n+2$-gon: start with an arbitrary triangulation $x$. Then repeatedly \emph{flip} a uniformly random diagonal~$d$ of the current triangulation: that is, remove~$d$ and replace it with the unique diagonal $d' \neq d$ that can be added to the triangulation~$x$ without introducing a crossing. (The removal of $d$ induces a quadrilateral formed by the two triangles incident to $d$. The existence and uniqueness of~$d'$ can be seen to follow from the convexity of the $n+2$-gon.) We impose a holding probability of 1/2 (see \cref{sec:mcmc}).

The flip walk is invariant to perturbations of the $n+2$-gon, so long as it remains convex. 

The flip walk is known to be isomorphic to the \emph{rotation walk} (see e.g.~\cite{sleator1988rotation}) on binary trees. A \emph{rotation} in a binary tree is an operation of one of two forms. The first form is as follows: take a parent node $z$ and a left child $w$ of $z$. Replace $z$ by $w$: that is, make $w$ a child of the parent of $z$ (a left child if $z$ is a left child, a right child if $z$ is a right child)\textemdash or, if $z$ is the root, make $w$ the root. Let $w$ retain its left child; let $z$ retain its right child. Then, let $u$ be the right child of $w$; and make $u$ the new left child of $z$.

The second form is the mirror image of the first: let $w$ be a right child of $z$, and proceed as in the first case but with left and right reversed.

The rotation walk is as follows: start with an arbitrary binary tree on $n$ nodes. Then repeatedly choose a uniformly random edge in the current tree, and perform a rotation at the parent node of that edge. (Impose a $1/2$ holding probability as in the triangulation walk.)

The isomorphism between the two walks follows from observing the correspondence between a triangulation flip and a tree rotation.

\subsection{Multicommodity Flows and Canonical paths}
\label{sec:mcflow}
A standard technique in bounding the mixing times of Markov chains uses \emph{canonical paths}. The idea is to show that the state space of the chain is in some sense free of ``bottlenecks,'' by finding a path between each pair of states such that no edge belongs to too many paths. A generalization of canonical paths is to construct a \emph{multicommodity flow}: a collection of flow functions in which for every pair of states $s, t \in \Omega$, $s$ sends a unit of flow to $t$ through paths in the state space~$\Omega$. The \emph{congestion} of a multicommodity flow is the maximum, over all edges, of the amount of flow sent across the edge, summed over all $s,t$ pairs that use the edge.

Two classical theorems relate canonical paths (more generally, multicommodity flows) to mixing. The first uses the \emph{Cheeger} inequalities to obtain an upper bound on the relaxation time from an upper bound on the congestion in a multicommodity flow. However, one suffers a quadratic cost in this process, as well as an additional cost in passing from the relaxation time to the mixing time, \cite[Theorem 2.1]{Trevisan13}. The latter problem can be solved in some cases using a result in \cite{avgcond}. \cite{eppstein2022improved} used this theorem to obtain an~$\widetilde O(n^3)$ mixing time for the triangulation walk.

Another theorem~\cite{diacstroock} allows one to pass from flows to relaxation time without the quadratic cost; however, one trades this cost for a cost incurred in analyzing the worst-case length of a path in the construction. The worst-case path length is sufficiently long in the construction in~\cite{eppstein2022improved} that it is not clear how to apply the theorem in this case.

In this paper, rather than using multicommodity flows, it will be useful to consider a flow function from one set of states $S \subseteq \Omega$ to a set $T \subseteq \Omega$, where we only use a single commodity. \cite{eppstein2022improved} gave a notion of \emph{multi-way single commodity flows}. We adapt the formulation as follows.

Given $S, T \subseteq \Omega$ where $S \cap T = \emptyset$, let an \emph{$S$-$T$ flow} be a function $\phi: \{(x, y) \in \Omega^2 \mid P(x, y) > 0\} \rightarrow \mathbb{R}$ such that $\phi(x, y) = -\phi(y, x)$ for all $x, y$, and such that,  defining the \emph{net flow} out of a state $x \in \Omega$ to be $\sum_{y \in \Omega: P(x, y) > 0} \phi(x, y)$:
\begin{enumerate}[(i)]
\item the net flow out of each $x \in S$ is equal to $\pi(T)$\,,
\item the net flow into each $y \in T$ is equal to $\pi(S)$\,, and 
\item the net flow into (and the net flow out of) each $x \in \Omega \setminus \left(S \cup T\right)$ is zero.
\end{enumerate}
(Here, if the net flow out of~$x$ is $\phi$ then we let the net flow into $x$ be $-\phi$.) 

\section{Transport Flows}\label{sec:transport_flow}
We recall the definition of transport flows from \cite{weimingjs}. Let $(\Omega, P, \pi)$ be an ergodic Markov chain and   $\mu, \nu$  distributions supported on $\Omega$. A \emph{transport flow} from $\mu$ to $\nu$ is a distribution $\Gamma$ of paths such that when $\gamma$ is drawn from $\Gamma$, the starting state of the path is distributed according to $\mu$ and the ending state according to $\nu$. We will denote the starting and ending states of the path $\gamma$ by $s(\gamma)$ and $t(\gamma)$ respectively.

We will assume for convenience that for every $x, y \in \Omega$ with $P(x, y) > 0$ the transitions $(x,y)$ and $(y, x)$ are not both used in the construction: i.e.~for every $\gamma$ such that $\Gamma(\gamma) > 0$, if $(x, y) \in \gamma$ then for all $\gamma'$ such that $\Gamma(\gamma') > 0$ we have $(y, x) \notin \gamma'$. 

It is easy to modify any transport flow to satisfy this condition (\cref{lem:onedir}).

Let $S, T \subseteq \Omega$. A transport flow from $S$ to $T$ is a transport flow from $\mu = \pi_S$ to $\nu = \pi_T$, where $\pi_S$ is supported on $S$, $\pi_T$ is supported on $T$, and $\mu(x) = \pi_S(x) = \frac{\pi(x)}{\pi(S)}$ for all $x \in S$, and similarly $\nu(x) = \pi_T(x)$ for all $x \in T$.

Let $f: \Omega \rightarrow \mathbb{R}$ be an arbitrary function. Let $\FF(S) = \Exp_{x \sim \pi_S} f(x)$ be the expectation of $f$ over $\pi_S$. 

The main tool we will use for establishing our functional inequalities will be the following result,
\begin{theorem}\label{thm:ourtransport}
Let $S, T \subseteq \Omega$. Suppose a transport flow $\Gamma$ exists from $S$ to $T$ where the maximum congestion is $\rho$, i.e.~for all $x,y \in \Omega$, 

\[
   \phi_{xy} := \frac{\pi(S)\pi(T)}{\pi(x)P(x, y)} \cdot \Exp_{\gamma \sim \Gamma} \left[ 1[(x, y) \in \gamma]\right] \leq \rho\,.
\]

The {average congestion} is

\[
    \bar\rho =  \sum_{x,y\in\Omega} \phi_{xy}\pi(x)P(x, y) = \pi(S)\pi(T)\sum_{x, y \in \Omega} \Exp_{\gamma \sim \Gamma}  1[(x, y) \in \gamma]\,.
\]

Then 
\begin{equation}
    \pi(S)\pi(T)(\FF(S) - \FF(T))^2 \leq \frac{\bar\rho\rho}{\pi(S)\pi(T)}\sum_{x, y \in \Omega}\pi(x)P(x, y)(f(x) - f(y))^2\,.
    \label{eqn:transport}
\end{equation}

\end{theorem}
A precursor of \cref{thm:ourtransport} using \emph{average path length} instead of \emph{average congestion} already appeared in \cite{weimingjs}. For concreteness, we present an equivalent formulation of their result below, 
\begin{theorem}[Theorem 10, \cite{weimingjs}]
\label{thm:transport}
    Let $S, T \subseteq \Omega$. Suppose a transport flow $\Gamma$ from $S$ to $T$ exists, where for every $x, y \in \Omega$ satisfying $y \sim x$, we have
    \[
        \frac{\pi(S)\pi(T)}{\pi(x)P(x, y)} \cdot \Exp_{\gamma \sim \Gamma}\left[\frac{1[(x, y) \in \gamma]}{|\gamma|}\right] \leq \kappa
    \]
    and $\Exp_{\gamma \sim \Gamma}\left[|\gamma|^2\right] \leq L$. 
    
    Then for any function $f: \Omega \rightarrow \mathbb{R}$
    \[
        \pi(S)\pi(T)(\FF(S) - \FF(T))^2 \leq \kappa L \sum_{x, y \in \Omega} \pi(x)P(x, y)(f(x) - f(y))^2\,.
    \]
\end{theorem}
We now present the proof of \cref{thm:ourtransport}.
\begin{proof}[Proof of \cref{thm:ourtransport}]
    Let $\phi_{xy} = \sum_{\gamma} \Gamma(\gamma) \frac{\pi(S)\pi(T)}{\pi(x)P(x, y)}$. We have
    \begin{align}
        \pi(S)\pi(T)(\FF(S) - \FF(T)) &= \pi(S)\pi(T)\Exp_{\gamma \sim \Gamma} \left[\sum_{(x, y) \in \gamma} (f(x) - f(y))\right] \nonumber \\
        &\leq \pi(S)\pi(T)\Exp_{\gamma \sim \Gamma} \sum_{(x, y) \in \gamma}|f(x) - f(y)| \nonumber \\
        &= \sum_{(x, y)} \sum_{\gamma: (x, y) \in \gamma}\Gamma(\gamma)\pi(S)\pi(T)|f(x) - f(y)| \nonumber \\
        &= \sum_{(x, y)} \phi_{xy}\pi(x)P(x, y)|f(x) - f(y)|\,.
    \end{align}

    Therefore
    \begin{align}
        \left(\pi(S)\pi(T)(\FF(S) - \FF(T))\right)^2 & \leq \Bigl( \sum_{(x, y) \in \Omega} \phi_{xy}\pi(x)P(x, y)|f(x) - f(y)| \Bigr)^2 \nonumber \\
        &\leq \Bigl(\sum_{(x, y) \in \Omega} \phi_{xy}\pi(x)P(x, y)\Bigr) \label{eqn:avgrho} \\
        &\quad \cdot \Bigl( \sum_{(x, y) \in \Omega} \phi_{xy}\pi(x)P(x, y)(f(x) - f(y))^2\Bigr)\,,  \label{eqn:rho}
    \end{align}
    by the Cauchy-Schwarz inequality. \eqref{eqn:avgrho} is $\bar\rho$; the claim follows.
\end{proof}
In \cite{montenegrovc} it was observed that \emph{average path length} and \emph{average congestion} are highly related parameters; for our proofs it will be more convenient to work with the latter parameter. Indeed, our proof of \cref{thm:ourtransport} is inspired by the proof of \cref{thm:transport} in \cite{weimingjs}.

The utility of \cref{thm:ourtransport} is that it provides a way to bypass the Cheeger inequality when one has paths that can be long in the worst case but where the average congestion is small. (This is true for the construction of \cite{eppstein2022improved}.) It is important to find sufficiently large sets $S, T$ between which to send the flow, however, as the denominator of the right-hand side of \eqref{eqn:transport} indicates.

It will be useful for our purposes to refine \cref{thm:ourtransport} as follows:
\begin{corollary}\label{cor:bettertransport}
Let $S, T \subseteq \Omega$. Suppose a transport flow $\Gamma$ exists from $S$ to $T$ where the maximum congestion is $\rho$ and the paths in the flow only use edges between vertices in $S \cup T$.

Let
\[
    \bar\rho_S =  \pi(T)\sum_{x \in S, y \in S \cup T} \Exp_{\gamma \sim \Gamma}  1[(x, y) \in \gamma]
\]
be the average congestion across edges having an endpoint in $S$, and define $\bar\rho_T$ symmetrically.

Then
\begin{align}
    \pi(S)\pi(T)(\FF(S) - \FF(T))^2 &\leq \frac{\rho \bar\rho_S}{\pi(T)}\sum_{x \in S, y \in S \cup T}\pi(x)P(x, y)(f(x) - f(y))^2 \nonumber \\
    &\quad\quad+ \frac{\rho \bar\rho_T}{\pi(S)}\sum_{x \in T, y \in S \cup T}\pi(x)P(x, y)(f(x) - f(y))^2\,.\label{eqn:bettertransport}
\end{align}
\end{corollary}
\begin{proof}
    The theorem follows from the proof of \cref{thm:ourtransport} and the observation that, by the definition of $\bar\rho$, we have
    $\bar\rho = \pi(S)\bar\rho_S + \pi(T)\bar\rho_T$.
\end{proof}

The following will allow us to pass from a combinatorial flow construction to a transport flow:
\begin{lemma}
    Given $(\Omega, P, \pi)$, let $S, T \subseteq \Omega$. Let $\phi$ be an $S$-$T$ flow. Then there exists a transport flow $\Gamma^{S\rightarrow T}$ satisfying, for each edge $(x, y)$ such that $\phi(x, y) \geq 0$:

    \[
        \sum_{\gamma \ni (x, y)}\Gamma(\gamma) = \phi(x, y)\,.
    \]
    \label{lem:flowtotransportflow}
\end{lemma}
The proof of \cref{lem:flowtotransportflow} uses a straightforward iterative process of increasing the flow across a path until an edge becomes ``tight''; the argument is straightforward but we defer the details to \cref{sec:msf}.

\section{$\widetilde{O}(n^{2})$ Mixing for Triangulations}
\label{sec:trimixing}

\subsection{Decomposing the Triangulation Walk and Outline of the Proof}
\label{sec:partition}
We begin by presenting the two key partitions of the triangulation state space given by \cite{mrs} and \cite{eppstein2022improved} and outlining some of their main properties. The \textit{central triangle} of a triangulation is the triangle which contains the center of the polygon; in the case where the center lies on one of the diagonals, we slightly perturb the center so that it lies in a unique triangle. 

Let $T$ be the set of central triangles. Define the \emph{central triangle partition} as the projection chain $\left(\Omproj = T, \Pproj, \piproj\right)$ corresponding to the decomposition $\Omega_t = \{x \in \Omega \mid x \textnormal{ includes the central triangle } t\}$. Recall that the transitions are given by
\begin{equation*}
    \Pproj(t,t') = \frac{\sum_{x \in \Omega_t, y \in \Omega_{t'}} \pi(x) P(x, y)}{\sum_{z \in \Omega_t}\pi(z)}\,,
\end{equation*}
when $t \neq t'$ and with self loops for the remaining probabilities, while the stationary distribution $\piproj$ satisfies $\piproj(t) = \pi(\Omega_t) = \sum_{x \in \Omega_t} \pi(x)$. According to the \ref{eqn:total_var}, we may decompose
\begin{equation*}
    \var_{\pi} f = \var_{\piproj} F + \Exp_{t \sim \piproj} \var_{\pi_t} f.
\end{equation*}

Since our goal is to obtain a Poincar\'e inequality, we want to bound both the terms by the Dirichlet form of the overall chain. Note that the second term is the average restricted variance in $\Om_t$, i.e. over triangulations with a given central triangle $t$. As the average of the Dirichlet forms of the restrictions is bounded by Dirichlet form of the overall chain (this is straightforward to check, but we will prove it explicitly when we use it), we need a meaningful Poincar\'e inequality for each restriction chain. A crucial observation is that every restriction chain is in fact a product chain over smaller triangulation walks (see \Cref{fig:cartesian}). This, combined with \Cref{lem:product_ineq}, will enable us to obtain a recursive bound on the Poincar\'e constant.

\begin{figure}[h]
    \centering
    \includegraphics[width=0.38\linewidth]{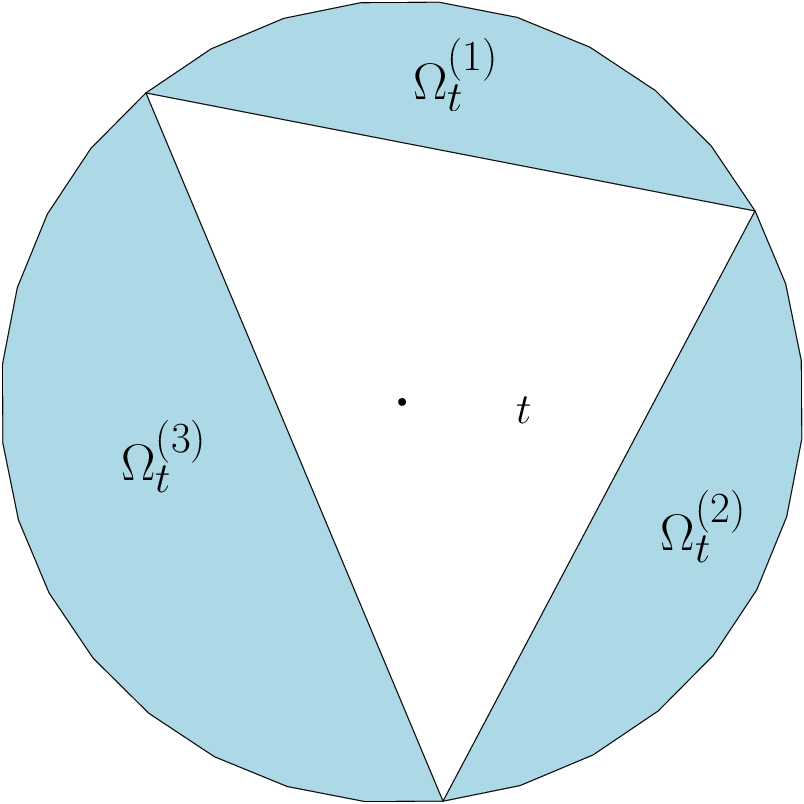}
    \caption{The original polygon is decomposed into three polygons by $t$. The restriction chain $(\Omega_t,P_t, \pi_t)$ is the product chain on the triangulations of each polygon. The filled region represents an arbitrary triangulation of the rest of the polygon.}
    \label{fig:cartesian}
\end{figure}

\begin{lemma}[\cite{eppstein2022improved}]
    \label{lem:cartprod}
    For each state $t \in \Omproj$, the restriction chain $(\Omega_t, P_t, \pi_t)$ is the Cartesian product of three chains $(\Omega_t^{(1)}, P_t^{(1)}, \pi_t^{(1)}), (\Omega_t^{(2)}, P_t^{(2)}, \pi_t^{(2)}), (\Omega_t^{(3)}, P_t^{(3)}, \pi_t^{(3)})$ each of which is isomorphic to the triangulation walk on a smaller polygon (possibly empty) on at most $n/2+1$ vertices.
\end{lemma}
\begin{remark}
    The last part of the lemma is crucial for our recursive argument. As each iteration reduces the problem to polygons of at most half the size of the original, we perform only a logarithmic number of iterations.
\end{remark}

It remains to bound $\var_{\piproj} F$, the variance on the projection chain, by the Dirichlet form on the overall chain. This will be done using the transport flow machinery from \Cref{sec:transport_flow}. The flow construction naturally points to studying boundaries between $\Om_t,\Om_{t'}$ for different central triangles, in which a new projection chain will arise.

Define the \emph{oriented partition chain} as the projection chain $\left(\Omorproj = [n], \Porproj, \piorproj\right)$ corresponding to the decomposition $\Omega_i = \{x \in \Omega \mid x \textnormal{ includes the triangle } (0, i, n+1)\}$, with $\Porproj, \piorproj$ defined in an analogous fashion to $\Pproj, \piproj$. Given $S \subseteq \Omorproj$, we will write $\Omega[S] := \bigcup_{i \in S}\Omega_i$.

Given $t, t' \in \Omproj$ with $\Pproj(t, t') > 0$, define
\[
    \Omega_{t t'} = \left\{x \in \Omega_t \mid \exists y \in \Omega_{t'}, P(x, y) > 0\right\}\,.
\]
The set $\Omega_{t t'}$ is the set of states in $\Omega_t$ having a neighboring state in $\Omega_{t'}$. 

Similarly given $i, j \in \Omorproj$ (recall that for all $i,j$, $\Porproj(i, j) > 0$), define
\[
    \Omega_{i j} = \left\{x \in \Omega_i \mid \exists y \in \Omega_{j}, P(i, j) > 0\right\}\,.
\]

See \cref{fig:pinningij}.

\begin{figure}[h]
    \centering
    \includegraphics[width=0.27\linewidth]{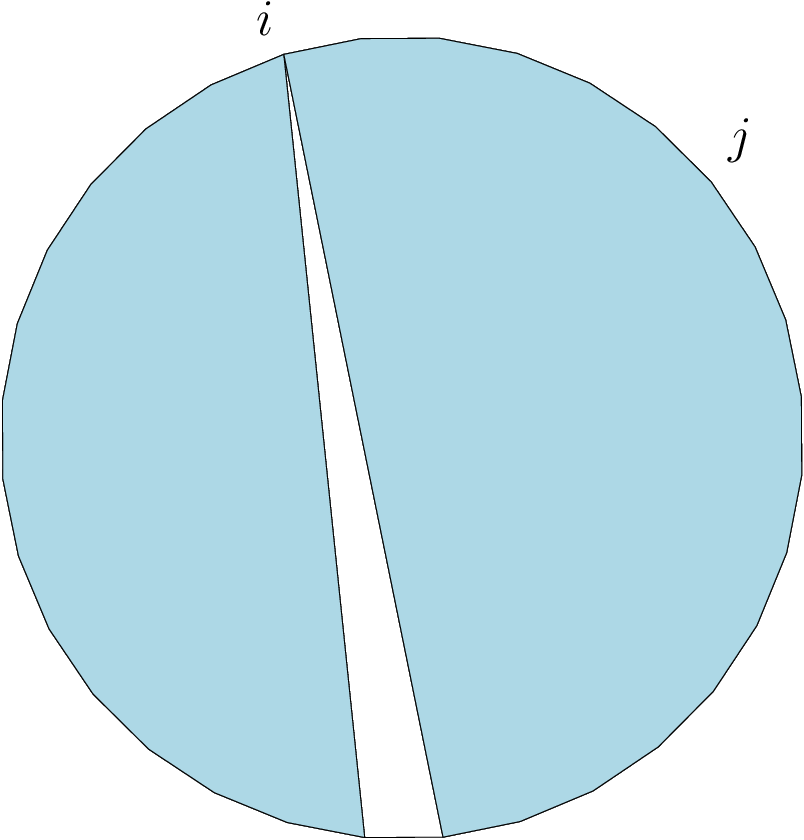}
    \hspace{0.2cm}
    \includegraphics[width=0.27\linewidth]{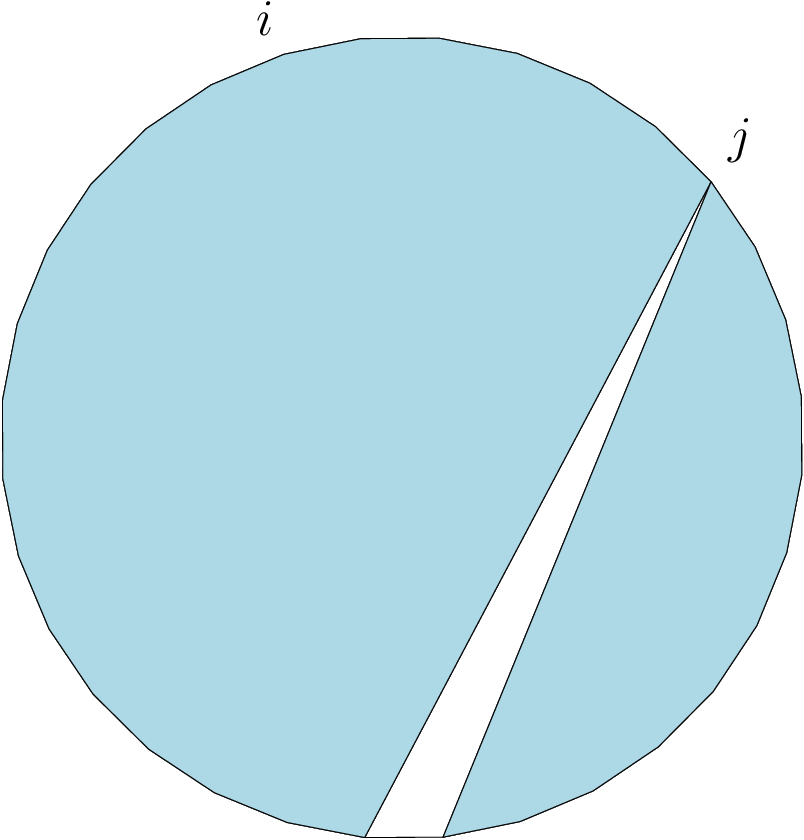}
    \hspace{0.2cm}
    \includegraphics[width=0.27\linewidth]{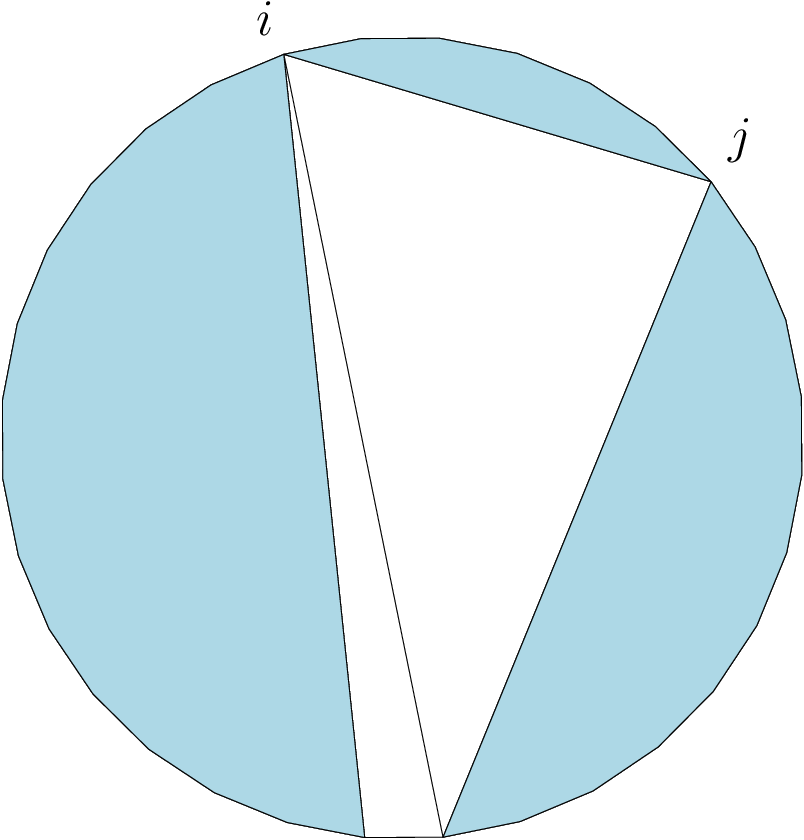}
    \caption{Left: the set $\Omega_i \subseteq \Omega$, represented by the state $i \in \Omorproj$, is the set of all triangulations that contain the triangle shown in the figure.
    Center: the set $\Omega_j$ is defined similarly.
    Right: The set $\Omega_{ij} \subseteq \Omega_i$ (see the definition of pinnings) is the set of triangulations that contain the two triangles shown in the figure.}
    \label{fig:pinningij}
\end{figure}

The central triangle~$t$ partitions the polygon into three smaller polygons (one of which may be the empty polgyon). Label the sub-polygon containing the triangle $u$ as polygon (1); label the other two sub-polgyons as (2) and (3). Consider any partial triangulation in which sub-polygons (2) and (3) are fully triangulated, but sub-polygon (1) is not triangulated. Denote such a sub-triangulation as $\eta \in \Omega_t^{(2)} \times \Omega_t^{(3)}$. Given $\eta$, let $\Omega^{(t, \eta)}$ denote the set of states in $\Omega_t$ having sub-polygons (2) and (3) triangulated according to~$\eta$. 

We extend this notation and use~$\left(\Omega^{(t,\eta)}, P^{(t,\eta)}, \pi^{(t,\eta)}\right)$ to refer respectively to the copy of the chain $\left(\Omega_t^{(1)}, P_t^{(1)}, \pi_t^{(1)}\right)$ induced by fixing~$\eta$. Given a function~$f:\Omega \rightarrow \mathbb{R}$ and given $t \in \Omproj$ and given $\eta \in \Omega_t^{(2)} \times \Omega_t^{(3)}$, define the function $f^{(t,\eta)}: \Omega^{(t,\eta)}\rightarrow \mathbb{R}$ so that, if $z \in \Omega^{(t,\eta)}$ we let $f^{(t,\eta)}(z) = f(z \cup \eta)$.

We will also denote by~$\Delta$ the degree of the (regular) graph induced by the chain  $(\Omega, P, \pi)$, and note that for all $x,y\in\Omega$ such that $P(x, y) > 0$ we have $P(x, y) = \frac{1}{\Delta}$. (We have $\Delta = \Theta(n)$, but it will be useful for clarity to distinguish it as a variable.) We will denote by~$\Delta_t^{(1)}$ the degree of the graph induced by $\left(\Omega_t^{(1)}, P_t^{(1)}, \pi_t^{(1)}\right)$ and define $\Delta_t^{(2)}, \Delta_t^{(3)}$ similarly.

\cite{eppstein2022improved} observed that given $t,t' \in \Omproj$, the boundary set $\Omega_{t t'}$ is precisely the set of triangulations that contain both the triangle $t$, and a particular triangle $u$ formed by two vertices of $t$ and an additional vertex of $t'$.  Letting $i$ be that additional vertex of $t'$, $u$ is the unique triangle having vertex $i$ and sharing its other two vertices with $t$. See \Cref{fig:boundary_decomp}.

We will denote by $\Omorproj^{(t,\eta)}$ the set of all vertices~$i$ induced by some triangle~$u$ as described above, and denote by $\left(\Omorproj^{(t,\eta)}, \Porproj^{(t,\eta)}, \piorproj^{(t,\eta)}\right)$ the oriented projection chain induced by fixing the ``special edge'' to be the edge of~$u$ that bounds polygon (1) (i.e., the edge $(j, b)$ in \Cref{fig:boundary_decomp}).  We then define $\FF^{(t,\eta)}:\Omorproj^{(t,\eta)} \rightarrow \mathbb{R}$ so that 
\[\FF^{(t,\eta)}(i) := \sum_{z \in \Omega_i^{(t,\eta)}}\frac{\pi^{(t,\eta)}(z)}{\piorproj^{(t,\eta)}(i)}\,.\]

We have:

\begin{lemma}[\cite{eppstein2022improved}]
For all $t, t' \in \Omproj$ where $P(t, t') > 0$, for all $\eta \in \Omega_t^{(2)} \times \Omega_t^{(3)}$: \begin{equation}
    \Omega_{t t'} = \bigcup_{\eta \in \Omega_t^{(2)} \times \Omega_t^{(3)}} \Omega_{i}^{(t,\eta)} \ \,,\label{eqn:relateproj}
\end{equation}
\label{lem:relateproj}
\end{lemma}
where by $\Omega_{i}^{(t,\eta)} \subseteq \Omega^{(t, \eta)}$ we denote the set of states in $\Omega^{(t,\eta)}$ that contain the triangle $u$ (having vertex $i$) described above.

In other words, \cref{lem:relateproj} states that a triangulation $x \in \Omega_t$ is in $\Omega_{tt'}$ if and only if $x$ contains the triangle $u$ (which has vertex $i$). These triangulations can then grouped according to how they triangulate $\Omega_t^{(2)}$ and $\Omega_t^{(3)}$, and by definition this union is disjoint. See \Cref{fig:boundary_decomp}.

\cref{lem:relateproj} is a key observation that relates the central-triangle projection chain and the oriented projection chain. It allowed \cite{eppstein2022improved} to reduce the problem of sending flow between central-triangle projection states $t$ and $t'$ to two problems: (i) sending flow from the boundary set $\Omega_{t't}$ to $\Omega_{tt'}$, and (ii) sending flow from the boundary set $\Omega_{tt'}$ to the rest of $\Omega_t$. Problem (i) is easy to solve as there is a perfect matching (see \cref{rmk:perfmat}) between $\Omega_{tt'}$ and $\Omega_{t't}$. Problem (ii) was solved by reducing to the problem of sending flow from $\Omega_i^{(t,\eta)}$ to the rest of $\Omega^{(t, \eta)}$\textemdash that is, the problem of sending flow from a state in the oriented projection chain $\Omorproj^{(t, \eta)}$ to the other states in that projection chain. This problem in turn, \cite{eppstein2022improved} showed (we will retrace the proof rigorously), can be recursively decomposed with no blowup in congestion.

\begin{figure}[h]
    \centering
    \includegraphics[width=0.4\linewidth]{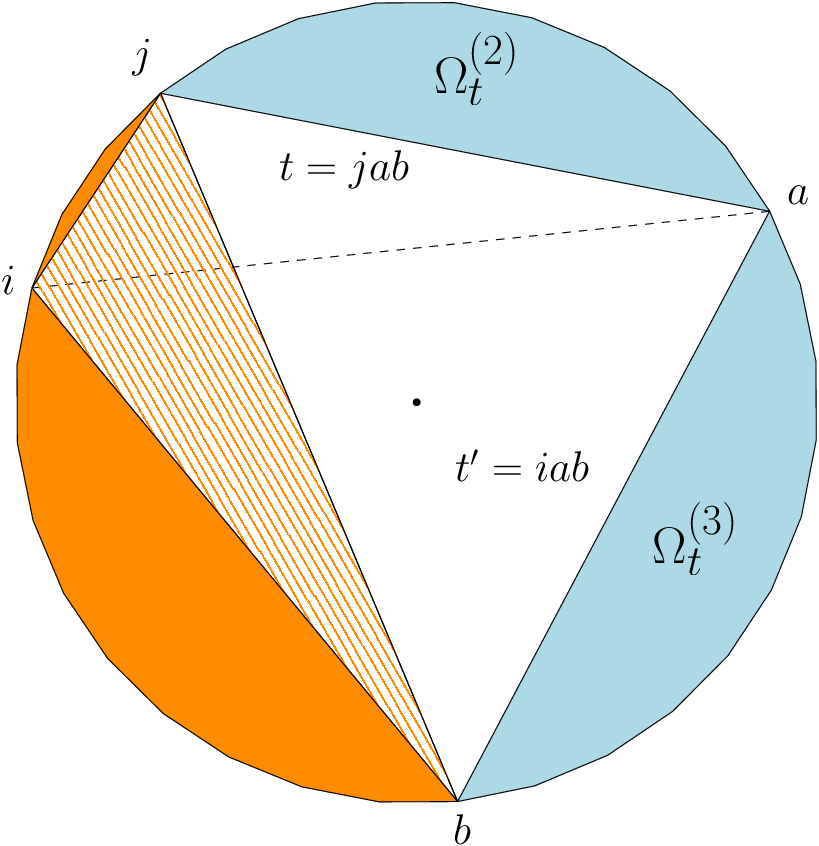}
    \caption{A triangulation in $\Om_{tt'}$ must necessarily contain the triangles $t$ and $u:=ijb$. Given a triangulation in $\Omega_t^{(2)}$ and $\Omega_t^{(3)}$ (light blue), we are left with triangulations in $\Omega_t^{(1)}$ (orange) subject to containing $u$ (shaded orange). We can now think as simply discarding polygons $(2)$, $(3)$. Then, we set the side of $t\cap (1)$ as the special edge of polygon $(1)$, and such triangulations of $(1)$ correspond to the elements of the oriented chain.}
    \label{fig:boundary_decomp}
\end{figure}

\subsection{Bounding the Variance of the Function over the Projection Chain}\label{sec:boundprojvar}
The aim of this section is to show the following lemma:

\begin{lemma}\label{lem:boundprojvar}
    The variance of the function $\FF$ over the projection chain satisfies the inequality
    $$\var_{\piproj} \FF \leq C(\log^{c}n)n^{2} \dir_{\pi}(f)\,,$$
    for all $f:\Om\to\mathbb{R}$, for some universal constants $C,c>0$.
\end{lemma}

\cref{lem:boundprojvar}, combined with a hierarchical decomposition of the (variance of the) overall chain using the projection chain $(\Omproj, \Pproj, \piproj)$, will allow us to establish the desired lower bound on the spectral gap.

\cite{eppstein2022improved} constructed a multicommodity flow in the projection chain $(\Omproj, \Pproj, \piproj)$ and analyzed its congestion. Their analysis can be used to show the following (we show this in \cref{sec:msf}):

\begin{lemma}[\cite{eppstein2022improved}]
Let~$f:\Omproj \rightarrow \mathbb{R}$ be a function over the projection chain $\left(\Omproj, \Pproj, \piproj)\right)$ defined in \cref{sec:partition}; let $\FF$ be defined with respect to~$f$ as in \cref{sec:funcineq}. Then:
{\small
\begin{align}
    &\var_{\piproj} F\nonumber\\
    &~\leq C_1\log^{c_1} n  \cdot n^{3/2} \sum_{t, t'} \piproj(t)\Pproj(t, t')(f(\Omega_{t t'}) - f(\Omega_{t' t}))^2 \label{eqn:bdryterm} \\
    &~+ C_1  \sqrt n\log^{c_1} n  \sum_{t} \Exp_{\eta \sim \nu_{2,3}}\left[ \sum_{k = 1}^{C_1 \log^{c_1} n} \piproj(t)\piorproj^{(t,\eta)}\left(\hat S_k^{(t,\eta)}\right)\left(\FF^{(t,\eta)}(S_k^{(t,\eta)}) - \Exp_{\pi^{(t,\eta)}} f^{(t,\eta)}\right)^2\right] \nonumber \\
    &~+ C_1\sqrt{n}\log^{c_1} n\sum_{t}\Exp_{\eta \sim \nu_{1,3}} \left[ \sum_{k = 1}^{C_1 \log^{c_1} n} \piproj(t)\piorproj^{(t,\eta)}\left(\hat S_k^{(t,\eta)}\right)\left(\FF^{(t,\eta)}(S_k^{(t,\eta)}) - \Exp_{\pi^{(t,\eta)}} f^{(t,\eta)}\right)^2\right] \nonumber \\
    &~+ C_1\sqrt n \log^{c_1} n  \sum_{t} \Exp_{\eta \sim \nu_{1,2}}\left[ \sum_{k = 1}^{C_1 \log^{c_1} n} \piproj(t)\piorproj^{(t,\eta)}\left(\hat S_k^{(t,\eta)}\right)\left(\FF^{(t,\eta)}(S_k^{(t,\eta)}) - \Exp_{\pi^{(t,\eta)}} f^{(t,\eta)}\right)^2\right]\,, \label{eqn:Ssetterm}
\end{align}}
for constants $C_1, c_1 > 0$,
where $\nu_{i,j}$ is the uniform distribution on $\Omega_t^{(i)} \times \Omega_t^{(j)}$, and
for all $t, i, k$, $\hat S_k^{(t,\eta)} \subseteq \Omorproj^{(t,\eta)}$, and $S_k^{(t,\eta)} = \bigcup_{i\in \hat S_k^{(t,\eta)}}\Omega_i^{(t,\eta)}$, and where $\piorproj(\hat S_k^{(t,\eta)}) \leq 3/4$.
\label{lem:flowreduction}
\end{lemma}

The term in the RHS of \eqref{eqn:bdryterm} describes the problem of sending flow between the boundary sets $\Omega_{tt'}$ and $\Omega_{t't}$; similarly the RHS of \eqref{eqn:Ssetterm} describes the problem of sending flow from a set $S_k^{(t, \eta)} \subseteq \Omega^{(t,\eta)}$, i.e.~a union of subsets of the chain $\Omega^{(t,\eta)}$, to the rest of $\Omega^{(t,\eta)}$. The expectation runs over all partial triangulations in $\Omega_t$ in which two of the three sub-polgyons are triangulated.

We derive \cref{lem:flowreduction} from \cite[Lemma 32]{eppstein2022improvedfull}, which is the congestion analysis of a flow construction. In the construction, each $\Omega_t$ begins with  uniformly concentrated flow that $\Omega_t$ needs to route to the rest of $\Omega$ (through edges in~$\Omega$). (This flow problem corresponds to the terms of the form $\piproj(t)(F(t) - \Exp_{\pi} f)^2$ in the variance of the projection function~$\FF$.) The authors then reduce the problem to a collection of flow subproblems in which (i) a pair of adjacent projection chain states send flow across the boundary between them, and (ii) a state $t$ receives flow from other states that it must distribute from its boundary throughout $\Omega_t$. 

Subproblems of the form (i) correspond to \eqref{eqn:bdryterm}, and subproblems of the form (ii) correspond to \eqref{eqn:Ssetterm}. For reasons specific to the flow construction, each subproblem of form (ii) involves distributing flow from multiple boundary sets (i.e.~multiple states in the oriented projection chain $\Omorproj^{(t,\eta)}$), of the form $\hat S_k^{(t, \eta)} \subseteq \Omorproj^{(t, \eta)}$.

The following insight from~\cite{eppstein2022improved} will allow us to bound \eqref{eqn:bdryterm}:
\begin{lemma}[Lemma 8, \cite{eppstein2022improved}]
    For all $t, t' \in \Omproj$ such that $P(t, t') > 0$, 
    the set of edges between states in $\Omega_{tt'}$ and $\Omega_{t't}$ is a perfect matching.
    \label{rmk:perfmat}
\end{lemma}

Using \cref{rmk:perfmat} and the fact that~$\pi$ is uniform, we prove (see e.g.~\cite{federmihail, jstv} for a similar technique):
\begin{lemma}
For all $t, t' \in \Omproj$ such that $\Pproj(t, t') > 0$:
\begin{align}
    \piproj(t)\Pproj(t, t')(f(\Omega_{tt'}) - f(\Omega_{t't}))^2 &\leq \sum_{x \in \Omega_{tt'}, y \in \Omega_{t't}} \pi(x)P(x, y)(f(x) - f(y))^2\,.\label{eqn:perfmat}
\end{align}

\label{lem:perfmat}
\end{lemma}
\begin{proof}   
    By the perfect matching between $\Omega_{tt'}$ and $\Omega_{t't}$ (\cref{rmk:perfmat}) we have
    \begin{align}
          \piproj(t)\Pproj(t, t') &= \sum_{x \in \Omega_{tt'}, y \in \Omega_{t't}}\pi(x)P(x, y) = \sum_{x \in \Omega_{tt'}} \frac{\pi(x)}{\Delta} \,.\label{eqn:sumxDelta}
    \end{align}

    Therefore:
    \begin{align}
        \piproj(t)\Pproj(t, t')(f(\Omega_{tt'}) - f(\Omega_{t't}))^2 &= \piproj(t)\Pproj(t, t')\left(\sum_{x \in \Omega_{tt'}}\frac{\pi(x)}{\Delta\piproj(t)\Pproj(t, t')}f(x) - \sum_{y \in \Omega_{t't}} \frac{\pi(y)}{\Delta\piproj(t')\Pproj(t', t)}f(y)\right)^2 \label{eqn:splitxy} \\
        &= \piproj(t)\Pproj(t, t')\left(\sum_{x \in \Omega_{tt'}, y \in \Omega_{t't}: P(x, y) > 0} \frac{\pi(x)P(x, y)}{\piproj(t)\Pproj(t, t')}(f(x) - f(y))\right)^2 \label{eqn:matchxy} \\
        &\leq\sum_{x \in \Omega_{tt'}, y \in \Omega_{t't}: P(x, y)  > 0} \pi(x)P(x, y)\left(f(x) - f(y)\right)^2 \,,\label{eqn:jensenxy}
    \end{align}
    where \eqref{eqn:splitxy} is by combining \eqref{eqn:sumxDelta} with the definition of $f(\Omega_{tt'})$; \eqref{eqn:matchxy} is again by \cref{rmk:perfmat} and by reversibility of the projection chain; 
    and \eqref{eqn:jensenxy} is by Jensen's inequality.
\end{proof}
The more challenging task is to bound \eqref{eqn:Ssetterm}:

\begin{lemma}\label{lem:proj_flow}
     There exist constants $C_2, C_3, c_2, c_3$ such that for sufficiently large $n$, the following holds:
     
     Given $i \in \Omorproj$ and letting $k, \hat S_k^{(t,\eta)}, S_k^{(t,\eta)}$ be as in \cref{lem:flowreduction}, let $T = \Omega^{(t,\eta)} \setminus S$. There exists a transport flow from $S$ to $T$ (in $\Omega^{(t, \eta)}$) with maximum congestion~$\rho \leq \Delta$, and with average congestion 
     \[
     \bar\rho_S \leq C_2 \sqrt n \log^{c_2} n \cdot \pi^{(t,\eta)}(T)\,,\]
     
     over all $(x, y)$ pairs with $x, y \in S$, and with average congestion
     \[
        \bar\rho_T \leq C_3 \sqrt n \log^{c_3} n \cdot \pi^{(t,\eta)}(S)\,,
    \]
    over all $(x, y)$ pairs with $x, y \in T$.
\end{lemma}
We defer the proof of \cref{lem:proj_flow} to \Cref{sec:avgcond}.
We will combine these bounds with the following straightforward corollary of \cref{cor:bettertransport}.
\begin{corollary}\label{cor:bettertransport2}
In the notation of \Cref{cor:bettertransport},
    \begin{align}
    \frac{\pi(S)}{\pi(S^c)}(\FF(S) - \Exp_\pi f)^2 &\leq \frac{\rho \bar\rho_S}{\pi(S^c)}\sum_{x \in S, y \in \Omega}\pi(x)P(x, y)(f(x) - f(y))^2 \nonumber \\
    &\quad\quad+ \frac{\rho \bar\rho_{S^c}}{\pi(S)}\sum_{x \in T, y \in \Omega}\pi(x)P(x, y)(f(x) - f(y))^2\,.\label{eqn:bettertransport2}
\end{align}
\end{corollary}
\begin{proof}
    This follows immediately by \Cref{cor:bettertransport} for $T:=S^c$ and by observing that
    \begin{align*}
        F(S)-F(S^c)&=\frac{1}{\pi(S)}\sum_{x\in S}\pi(x)f(x)-\frac{1}{1-\pi(S)}\sum_{y\in S^c}\pi(y)f(y)\\
        &=\frac{1}{\pi(S)}\sum_{x\in S}\pi(x)f(x)-\frac{1}{1-\pi(S)}\left[\sum_{y\in \Omega}\pi(y)f(y)-\sum_{x\in S}\pi(x)f(x)\right]\\
        &=\frac{1}{1-\pi(S)}\left[\frac{1}{\pi(S)}\sum_{x\in S}\pi(x)f(x)-\sum_{y\in \Omega}\pi(y)f(y)\right]\\
        &=\frac{1}{\pi(S^c)}(F(S)-\Exp_\pi f).
    \end{align*}
\end{proof}

We are now ready to prove \Cref{lem:boundprojvar}.

\begin{proof}[Proof of \Cref{lem:boundprojvar}]
Given $\eta \in \Omega_t^{(2)} \times \Omega_t^{(3)}$, applying \Cref{cor:bettertransport2} for $S := S_k^{(t, \eta)}$, the bounds on the congestion from \Cref{lem:proj_flow}, and using the fact that $\hat{\pi}_t^{(\eta)}(\hat S_k^{(t,\eta)})\leq 3/4$, we get that for 
\begin{align}
&\piproj(t)\piorproj_t\left(\hat S_k^{(t,\eta)}\right)\left(\FF^{(t,\eta)}(S_k^{(t,\eta)}) - \Exp_{\pi^{(t,\eta)}} f\right)^2 \notag 
    \\
    &\qquad\leq C_2C_3 \Delta_t^{(1)}\cdot \sqrt n \log^{c_2+c_3} n\piproj(t)\sum_{x, y \in \Omega^{(t,\eta)}} \pi^{(t,\eta)}(x)P^{(t,\eta)}(x, y)(f^{(t,\eta)}(x) - f^{(t,\eta)}(y))^2\,,\label{eqn:trixport}
\end{align}
for some constants $C_2, C_3, c_2, c_3$.

Furthermore:
\begin{align}
    \sum_{x,y \in \Omega_t}\pi(x)P(x,y)(f(x)-f(y))^2 &= \Exp_{\eta \sim \nu_{2,3}}\sum_{x,y\in\Omega^{(t,\eta)}}\pi(x)P(x,y)(f(x)-f(y))^2 \nonumber \\
    &\qquad + \Exp_{\eta \sim \nu_{1,3}}\sum_{x,y\in\Omega^{(t,\eta)}}\pi(x)P(x,y)(f(x)-f(y))^2 \nonumber \\
    &\qquad + \Exp_{\eta \sim \nu_{1,2}}\sum_{x,y\in\Omega^{(t,\eta)}}\pi(x)P(x,y)(f(x)-f(y))^2\,.
    \label{eqn:etatofull}
\end{align}

We set $S = S_k^{(t,\eta)}, T = \Omega^{(t,\eta)} \setminus S_k^{(t,\eta)}$. We have used that by \cref{lem:proj_flow}, $\rho \leq \Delta$ and $\bar \rho_S \leq \pi^{(t,\eta)}(T)\Delta$ and $\bar \rho_T \leq \pi^{(t,\eta)}(S)\Delta$.

Plugging \eqref{eqn:perfmat} into \eqref{eqn:bdryterm} and  \eqref{eqn:trixport} into \eqref{eqn:Ssetterm} yields
\begin{align}
    & \var_{\piproj} F\nonumber\\
    &~\leq C \Delta \sqrt n \log^c n \sum_{t \neq t'}\sum_{x \in \Omega_t, y \in \Omega_{t'}} \pi(x)P(x, y)(f(x) - f(y))^2 \nonumber  \\
    &~~~ + C \Delta_t^{(1)}\cdot n\log^c n \sum_{t}\piproj(t)\Exp_{\eta \sim \nu_{2,3}}\sum_{x, y \in \Omega_t^{(\eta)}}\left[ \pi^{(t,\eta)}(x)P^{(t,\eta)}(x, y)(f^{(t,\eta)}(x) - f^{(t,\eta)}(y))^2\right] \nonumber  \\ 
    &~~~ + C \Delta_t^{(2)}\cdot n \log^c n \sum_{t}\piproj(t)\Exp_{\eta \sim \nu_{1,3}}\sum_{x, y \in \Omega_t^{(\eta)}}\left[ \pi^{(t,\eta)}(x)P^{(t,\eta)}(x, y)(f^{(t,\eta)}(x) - f^{(t,\eta)}(y))^2\right] \nonumber  \\ 
    &~~~ + C \Delta_t^{(3)}\cdot n \log^c n \sum_{t}\piproj(t)\Exp_{\eta \sim \nu_{1,2}}\sum_{x, y \in \Omega_t^{(\eta)}}\left[ \pi^{(t,\eta)}(x)P^{(t,\eta)}(x, y)(f^{(t,\eta)}(x) - f^{(t,\eta)}(y))^2\right] \nonumber  \\ 
    &~\leq C \Delta \sqrt n \log^c n \sum_{t \neq t'}\sum_{x \in \Omega_t, y \in \Omega_{t'}} \pi(x)P(x, y)(f(x) - f(y))^2  \nonumber \\
    &~~~ + C \Delta\cdot n\log^c n \sum_t \sum_{x, y \in \Omega_t} \pi(x)P(x, y)(f(x) - f(y))^2 \label{eqn:vardirt} \\
    &~\leq C(\log^cn)n^{2}\cdot \dir_\pi(f)\,,\nonumber
\end{align}
where $C, c$ are constants determined by $C_1, C_2, C_3, c_1, c_2, c_3$. For \eqref{eqn:vardirt} we have used \eqref{eqn:etatofull} and also the observation that $\Delta_t^{(j)}P^{(t,\eta)} = \Delta P(x, y) = 1$ for $j \in \{1, 2, 3\}$ and for all $t, \eta, x, y$. 
\end{proof}

\subsection{Proof of the Main Theorem}
The following is an immediate consequence of \cref{lem:convexity_dpi},
\begin{lemma}
Let $g:= f^2$ (pointwise) and let $G(t) := \Exp_{x \sim \pi_t} g(x)$. Then

    \[
        \var_{\bar \pi} \sqrt{G} \leq \var_{\bar \pi} F + \Exp_{t \sim \bar \pi} \var_{\pi_t} f = \var_\pi f\,.
    \]
\label{lem:sqrtG}
\end{lemma}

\begin{proof}[Proof of \cref{thm:lsi}]
We first prove the Poincar\'e inequality. We will use the law of total variance and the decomposition properties of the chain to get a recursive bound on the relaxation time. Let $t_n$ be the relaxation time for the flip walk on an $n+2$-gon, that is, $\var_\pi f\leq t_n\dir_\pi(f)$ for all real-valued functions $f$, and $t_n$ is the smallest such constant. 

Recall that, by the total law of variance, we may write
\begin{equation}\label{eqn:vardecomp_proof}
    \var_{\pi} f = \var_{\piproj} \FF + \Exp_{t \sim \piproj} \var_{\pi_t} f\,.
\end{equation}
By \Cref{lem:boundprojvar},
\begin{equation}\label{eqn:varterm}
    \var_{\piproj} \FF \leq C'(\log^{c'}n)n^{2} \dir_{\pi}(f, f)=:h(n)\dir_\pi(f) \,,
\end{equation}
for some constants $C',c'>0$.

To bound $\Exp_{t \sim \piproj} \var_{\pi_t} f$ recall that $(\Omega_t, P_t)$ is a product chain of at most three chains, each of which is the triangulation flip walk on a smaller polygon each of which has $i+2$ sides for some $i\leq n/2$ (\Cref{lem:cartprod})
Hence, by \Cref{lem:product_ineq},
\begin{equation}\label{eqn:poi_prod}
\var_{\pi_t}f\leq \frac{n-4}{i_1-1}t_{i_1}\dir_{\pi_t}(f)\,,
\end{equation}
for some $i_1\leq n/2$.
Plugging in \eqref{eqn:varterm} and \eqref{eqn:poi_prod} into \eqref{eqn:vardecomp_proof} we get
\begin{equation}\label{eqn:var_recurs}
    \var_\pi f \leq h(n)\dir_\pi(f)+\frac{n-4}{i_1-1}t_{i_1}\Exp_{t\sim\pi_t}\dir_{\pi_t}(f)\,.
\end{equation}
Recall that by the definition of the restriction chain $(\Omega_t, P_t)$, $\pi_t(x)=\frac{\pi(x)}{\piproj(t)
}$ and
\[P_t(x, y) = \frac{P(x, y)}{\sum_{z \in \Omega_t}P(x, z)}=\frac{n-1}{n-4}P(x,y)\,,\]
for all $x,y\in\Om_t$. Thus,
\begin{align}
    \Exp_{t\sim\pi_t}\dir_{\pi_t}(f)&=\sum_{t\in\Omproj}\piproj(t)\sum_{x,y\in\Om_t}\pi_t(x)P_t(x,y)(f(x)-f(y))^2 \nonumber \\
    &=\frac{n-1}{n-4}\sum_{t\in\Omproj}\sum_{x,y\in\Om_t}\pi(x)P(x,y)(f(x)-f(y))^2 \nonumber \\
    &\leq \frac{n-1}{n-4}\dir_\pi(f)\,.\label{eqn:exp_dir} 
\end{align}
Combining \eqref{eqn:var_recurs} and \eqref{eqn:exp_dir} we get that for all functions $f$
\begin{equation}\label{eqn:var_recurs2}
    \var_\pi f\leq \left[h(n)+ \frac{n-1}{i_1-1}t_{i_1}\right]\dir_\pi(f),
\end{equation}
i.e.
\begin{equation}\label{eqn:rel_recurs}
    t_n\leq h(n)+\frac{n-1}{i_1-1}t_{i_1}\,,
\end{equation}
for some $i_1\leq n/2$. Iterating this $M\leq\log n$ times to reach the relaxation time of a constant-size walk, we get
\begin{align*}
    t_n&\leq h(n)+\frac{n-1}{i_1-1}\left(h(i_1)+\frac{i_1-1}{i_2-1}t_{i_2}\right)\\
    &=h(n)+\frac{n-1}{i_1-1}h(i_1)+\frac{n-1}{i_2-1}t_{i_2}\\
    &\leq h(n)+\frac{n-1}{i_1-1}h(i_1)+\dots+\frac{n-1}{i_M-1}h(i_M)+\frac{n-1}{i_M-1}.
\end{align*}
Note that for any $i\leq n$ we have $\frac{n-1}{i-1}h(i)=\frac{n-1}{i-1}i^{2}\log^ci\leq h(n)$, so
$$t_n\leq (\log n+1)h(n)+n=\widetilde O(n^{2}).$$
This concludes the proof for the spectral gap.

For the log-Sobolev inequality, the proof is nearly identical. Let $g:=f^2$ and $\tilde t_n$ be the inverse log-Sobolev constant for the flip walk on an $n+2$-gon, that is, $\ent_\pi g\leq\tilde t_n\dir_\pi(f)$ for all non-constant functions $f$, and $\tilde t_n$ is the smallest such constant. By the \ref{eqn:total_ent} we have
\begin{equation}\label{eqn:entdecomp_proof}
    \ent_{\pi} g = \ent_{\piproj} G + \Exp_{t \sim \piproj} \ent_{\pi_t} g,
\end{equation}
where $G(t):=\Exp_{x\sim\pi_t}g(x)$.

By standard comparisons between variance and entropy (\cref{fac:varent_comp}) and by applying \cref{lem:sqrtG}, we get
\[\ent_{\bar\pi}(G) \le O(\log n) \cdot \var_{\bar\pi}(\sqrt{G}) \le O(\log n) \var_{\pi} f,\]
observing that the minimum measure of a state in the projection chain~$\piproj$ is $\Omega(n^{-3})$, thus passing from variance to entropy incurs at most a logarithmic cost. Using the Poincar\'e inequality we proved, this yields
\begin{equation}\label{eqn:entterm}
    \ent_{\piproj} G\leq C''(\log^{c''}n)n^{2}\dir_\pi(f):=\tilde{h}(n)\dir_\pi(f).
\end{equation}
To bound $\Exp_{t\sim\piproj}\ent_{\pi_t} g$ we simply repeat the argument for the bound of $\Exp_{t\sim\piproj}\var_{\pi_t} f$. We rewrite the main steps for completeness.

Recall that $(\Omega_t, P_t)$ is a product chain of at most three chains, each of which is the triangulation flip walk on a smaller polygon each of which has $i+2$ sides for some $i\leq n/2$ (\Cref{lem:cartprod})
Hence, by the entropy part of \Cref{lem:product_ineq} we get
\begin{equation}\label{eqn:lsi_prod}
\ent_{\pi_t}g\leq \frac{n-4}{i_1-1}\tilde t_{i_1}\dir_{\pi_t}(f)\,,
\end{equation}
for some $i_1\leq n/2$.
Plugging in \eqref{eqn:entterm} and \eqref{eqn:lsi_prod} into \eqref{eqn:entdecomp_proof} we get
\begin{equation}\label{eqn:var_recurslsi}
    \ent_\pi g \leq \tilde h(n)\dir_\pi(f)+\frac{n-4}{i_1-1}\tilde t_{i_1}\Exp_{t\sim\pi_t}\dir_{\pi_t}(f).
\end{equation}
Observing that \eqref{eqn:var_recurslsi} is analogous to \eqref{eqn:var_recurs}, the rest of the proof follows the identical calculations for the Dirichlet form and the recursion, which lead to
$$\tilde t_n\leq (\log n+1)\tilde{h}(n)+n=\widetilde O(n^{2}),$$
thus proving the desired log-Sobolev inequality.
\end{proof}
\cref{cor:mix_flip} follows immediately from \cref{thm:lsi} and the observation that the holding probability of the flip walk is $1/2$.

\section{A Transport Flow for Triangulations}
\subsection{Transport Flow Idea}
\label{sec:trixport}
In this section we describe and analyze the transport flow that will enable us to establish \Cref{lem:proj_flow}, 
the remaining ingredient in proving \cref{thm:lsi}.

\cite{eppstein2022improved} gave a flow construction in combinatorial terms --- a \emph{multi-way single-commodity flow (MSF)} in their language. They gave a bound on the maximum congestion, which we will also use. For our purposes, we will also need a bound on the average congestion or the average path length in this construction (by \cite{montenegrovc} these are equivalent), which is not immediate from their analysis. In this section we retrace their construction rigorously and establish the bounds we need to prove \cref{thm:lsi}. To this end, we characterize the flow as a functional \emph{equality} (not inequality)\textemdash in the spirit of~\cite{diacstroock}. In intuitive terms,  we describe the problem of sending flow between a pair of states $i,j\in\Omorproj$ as the difference $\FF(i) - \FF(j)$ in the value of the function $\FF$ at these two states (\cref{lem:l1dirweak}). We then describe the flow construction as an expectation of telescoping sums of differences over the paths in the flow. This formal description of the flow will allow us to analyze the average congestion in \cref{sec:avgcond}.

\begin{lemma}
    \label{lem:l1dirweak}
    Let $f:\Omega\rightarrow \mathbb{R}$ be a function and consider the function~$\FF$. For all $i, j \in \Omorproj,$ there exists a flow function $\phi_{ij}: \left\{(x, y) \in (\Omega_i \cup \Omega_j)^2\mid P(x, y) > 0\right\} \rightarrow \mathbb{R}$ satisfying
    
    \begin{align}
        \FF(i) - \FF(j) &= \frac{\Delta}{2\piorproj(i)\piorproj(j)}\sum_{x, y \in \Omega_i}\phi_{ij,xy}\pi(x)P(x, y)(f(x) - f(y)) \nonumber \\
        & \qquad + \frac{\Delta}{2\piorproj(i)\piorproj(j)}\sum_{x, y \in \Omega_j}\phi_{ij,xy}\pi(x)P(x, y)(f(x) - f(y)) \label{eqn:theflow} \\
        & \qquad + \frac{\Delta}{\piorproj(i)\piorproj(j)}\sum_{x \in \Omega_i, y \in \Omega_j}\phi_{ij,xy}\pi(x)P(x, y)(f(x) - f(y))\,.
    \end{align}
    where the function~$\phi_{ij}$ describes the congestion across the edge $(x, y)$, and where $\phi_{ij,xy}$ is bounded in absolute value by 1.

    (From our definition of $\phi_{ij}$ it will follow that $\phi_{ij,xy} = -\phi_{ij,yx}$ for all $x,y$, so $\phi_{ij,xy}(f(x) - f(y)) = \phi_{ij,yx}(f(y) - f(x))$.)

    The function~$\phi_{ij}$ induces a
    transport flow $\Gamma^{\Omega_i\rightarrow \Omega_j}$ in $\Omega$ that produces maximum congestion at most~$\Delta$.
\end{lemma}

In combinatorial terms, \cref{lem:l1dirweak} describes a flow construction in which we route flow from $\Omega_i$ to $\Omega_j$, through the state space of the overall chain. The congestion incurred by this flow is at most $\Delta$. 

In this section, we will prove \cref{lem:l1dirweak}. In \cref{sec:avgcond}, combining \cref{lem:l1dirweak} with an additional average congestion analysis relying on the asymptotic behavior of Catalan structures, we will then prove the following:

\begin{lemma}
 \label{lem:ldirstrong}
 For all $S, T \subseteq \Omorproj,$ there exists a transport flow $\Gamma^{\Omega[S]\rightarrow \Omega[T]}$ in $\Omega$ that produces maximum congestion~$\rho \leq \Delta$ and that satisfies $\bar\rho_{\Omega[S]} \leq C\log^c n \cdot \sqrt n \piorproj(T)$ and $\bar\rho_{\Omega[T]} \leq C \log^c n \cdot \sqrt n \piorproj(S)$ for constants $C > 0, c > 0$.\footnote{In \cref{lem:proj_flow} we considered sets $S,T \subseteq \Omega$ that were subsets of the overall state space. Here we have redefined $S,T \subseteq \Omorproj$, as this will be convenient in the proofs in \cref{sec:avgcond}. Thus we are using the notation $\Omega[S], \Omega[T]$ (which we defined in \cref{sec:prelim}) to denote the induced subsets of the overall state space~$\Omega$.}
\end{lemma}

\cref{lem:ldirstrong} strengthens \cref{lem:l1dirweak} in two ways: adding an average congestion analysis, and allowing for the sets $S, T$ in the transport flow to include  multiple states in the projection chain. Once proven, \cref{lem:ldirstrong} will imply \cref{lem:proj_flow}, which completes the proof of \cref{thm:lsi} given in \cref{sec:trimixing}.

\subsection{Pinnings}
\label{sec:pinnings}
Our aim is now to prove \cref{lem:l1dirweak}. We need to adapt the idea of a \emph{pinning} from existing local-to-global frameworks. However, our case involves a recursive partitioning of the state space, rather than a simplicial complex. We will define a pinning to be a sequence $\eta = (i_1, i_2, \dots, i_s)$ of polygon vertices, where no $j < j' < j''$ triple exists such that $i_{j'} < i_{j} < i_{j''}$ or $i_{j''} < i_j < i_{j'}$. The pinning induces a collection of triangles which we construct recursively: we begin with the triangle $(0, n+1, i_1)$, then for every $2 \leq j \leq s$, we add the unique triangle that includes vertex $i_j$ as well as some edge already added to the pinning. See \Cref{fig:pinning_detailed}.

\begin{figure}[h]
    \centering
    \includegraphics[width=0.3\linewidth]{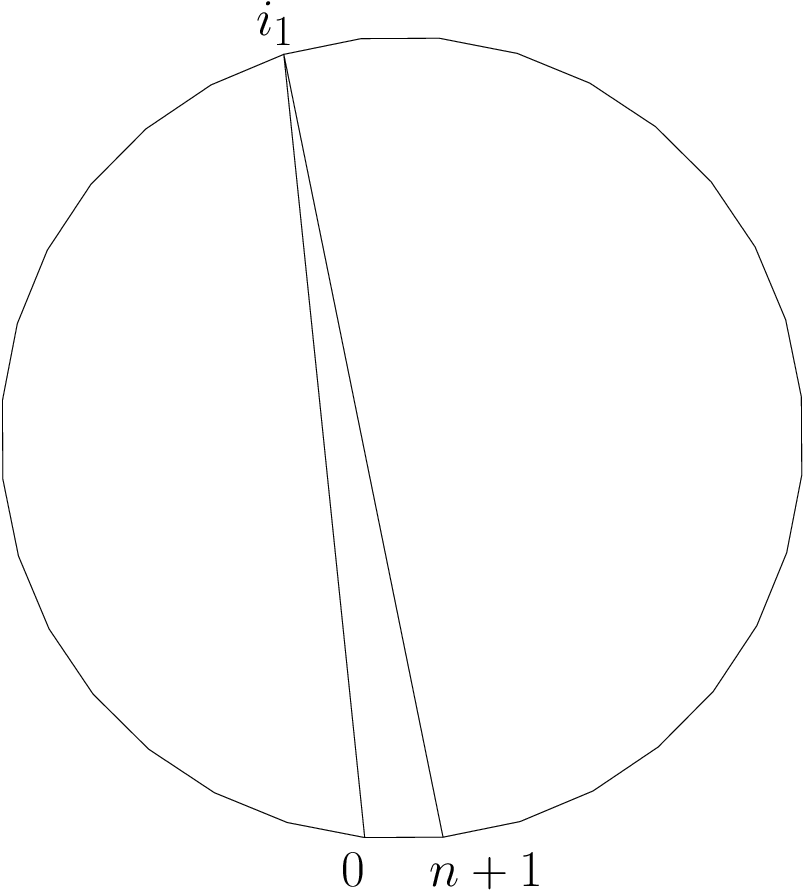}
    \hspace{1cm}
    \includegraphics[width=0.3\linewidth]{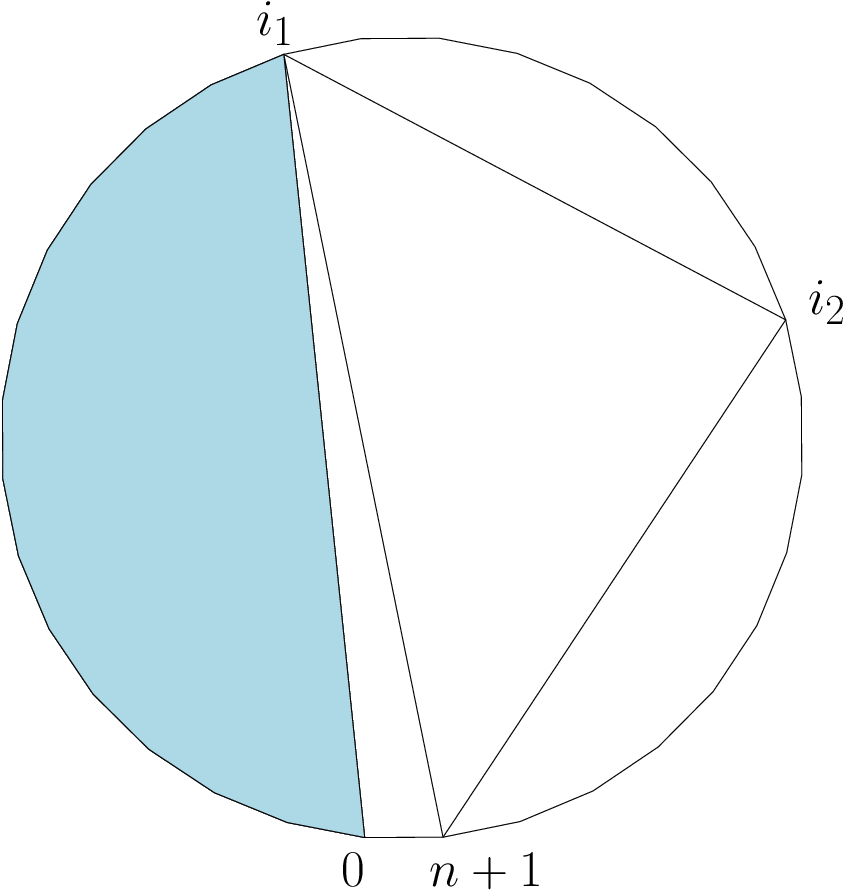}
    \\
    \vspace{0.5cm}
    \includegraphics[width=0.3\linewidth]{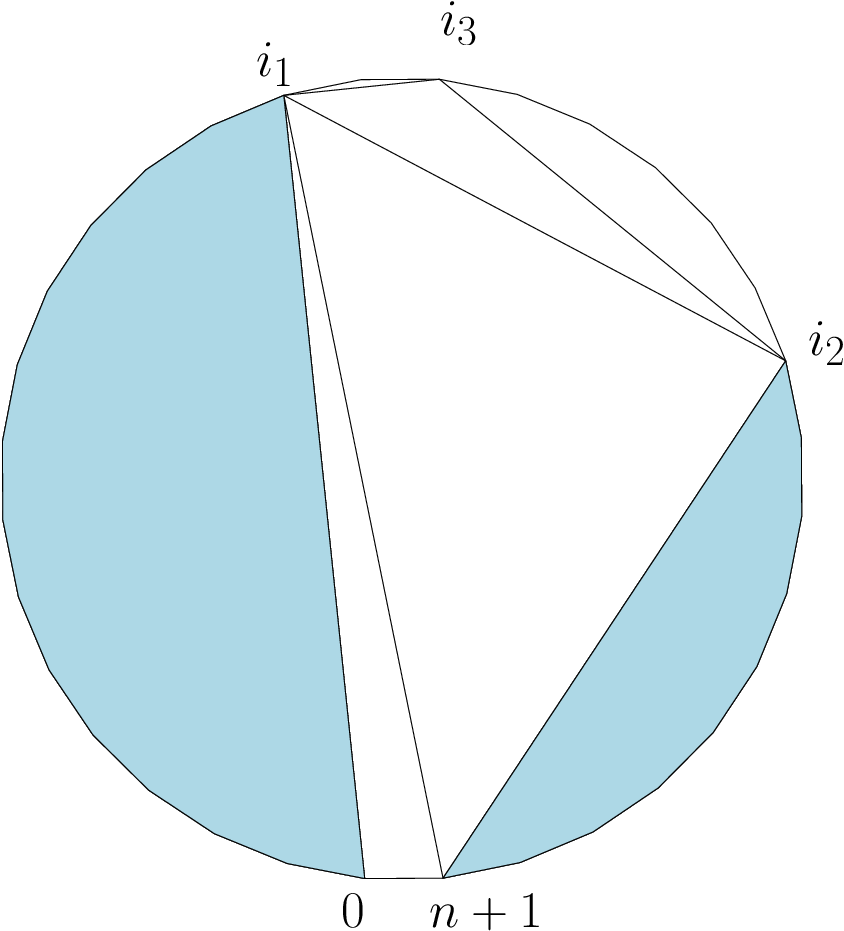}
    \hspace{1cm}
    \includegraphics[width=0.3\linewidth]{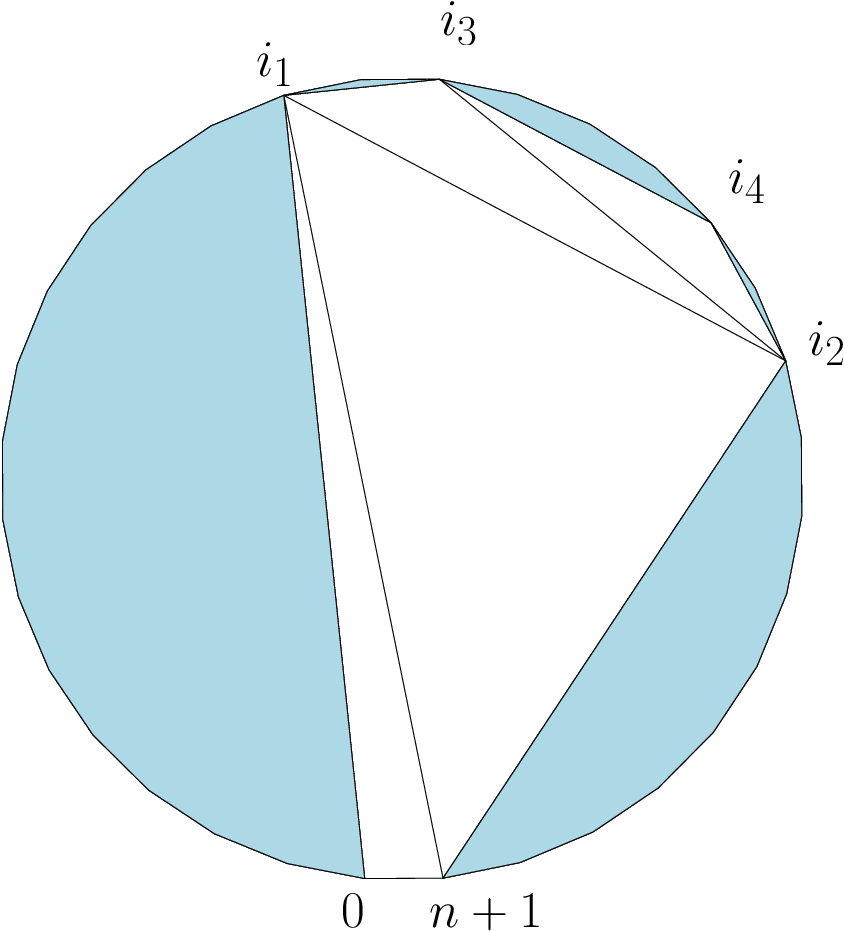}
    \caption{A pinning $\eta = (i_1, i_2, i_3, i_4)$ and the corresponding sequence of triangles.}
    \label{fig:pinning_detailed}
\end{figure}

Let $\Omega_{\eta}$ be the set of triangulations $x \in \Omega$ such that $x$ contains the triangles induced by the pinning $\eta$. The ``innermost'' triangle $T$ of $\eta$ (the one having the last vertex of $\eta$ as a vertex) induces a split of part of the $n$-gon into left and right sides, see \Cref{fig:pinning}. Each state $x \in \Omega_{\eta}$ has one or two triangles sharing an edge with $T$. Define $\Omorproj_{\eta,l}$ and respectively $\Omorproj_{\eta, r}$ to be the projection chains induced by identifying each state $x \in \Omega_{\eta}$ with its triangles to the left and right of $T$. (In \cref{fig:pinning}, $\Omorproj_{\eta,l} = \{i_3+1, i_3+2, \dots, i_4-1\}$ and $\Omorproj_{\eta,r} = \{i_4+1, i_4+2, \dots, i_2-1\}$. In \cref{fig:pinning_l}, we zoom in to the equivalence classes that form the projection chain $\Omorproj_{\eta,l}$.)

\begin{figure}[h]
    \centering
    \includegraphics[width=0.5\linewidth]{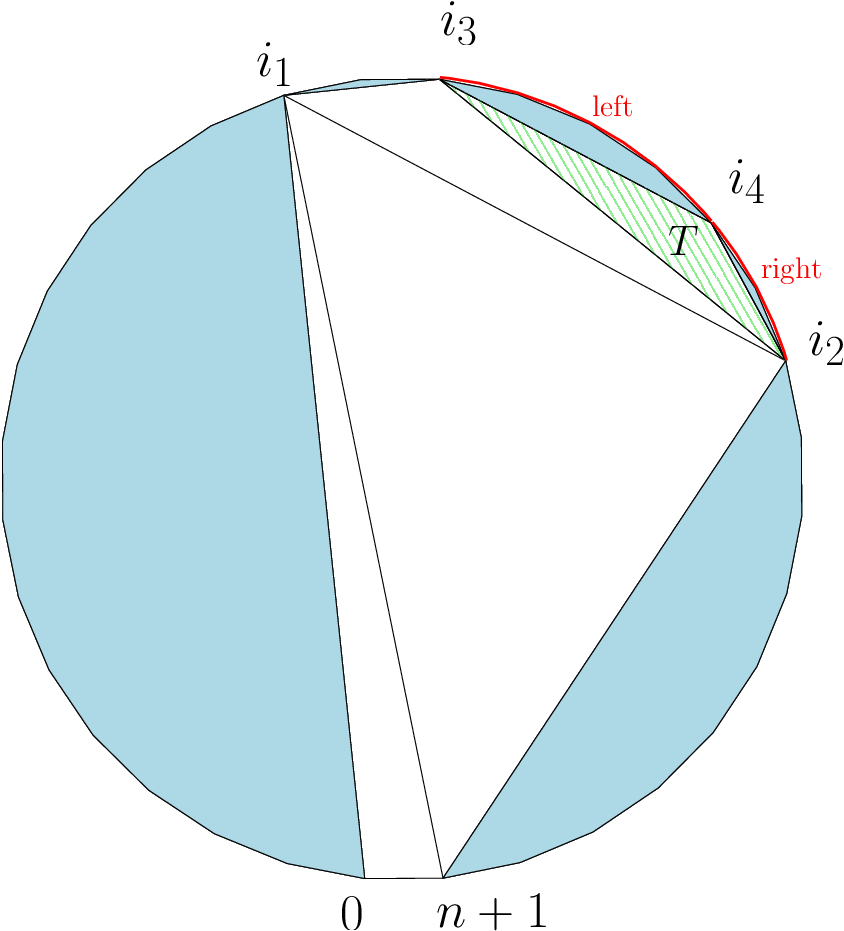}
    \caption{A pinning $\eta = (i_1, i_2, i_3, i_4)$. The untriangulated region is shown with a fill.}
    \label{fig:pinning}
\end{figure}

\begin{figure}[h]
    \centering
    \includegraphics[width=0.6\linewidth]{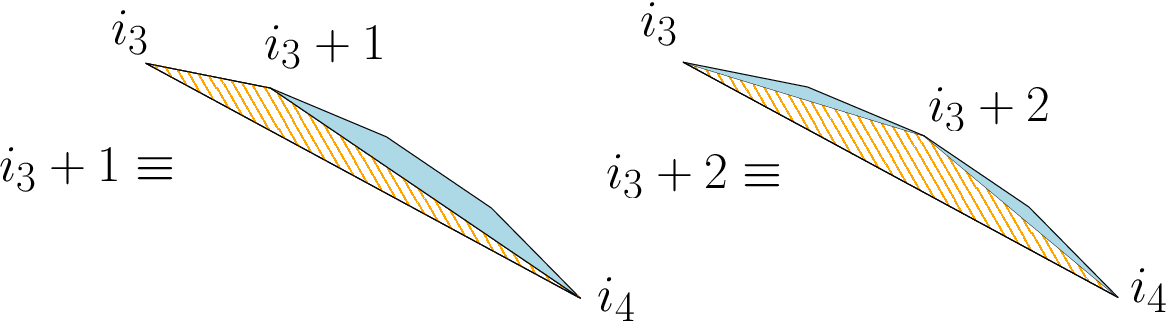}
    \caption{For the pinning of \Cref{fig:pinning}, the elements $i_3+1,i_3+2$ are the equivalence classes of triangulations containing the pinning $\eta$ and the respective shaded orange triangle. They are naturally identified with the triangulations of the polygon label as ``left'' in \Cref{fig:pinning} which contain the appropriate triangle.}
    \label{fig:pinning_l}
\end{figure}

Given a pinning $\eta = (i_1, i_2, \dots, i_k)$, let $\piorproj(\eta) = \pi(\Omega_{\eta})$. Given a vertex $j \in \Omorproj_{\eta, l}\cup{\Omorproj}_{\eta,r}$, let $\eta j$ be the pinning $(i_1, i_2, \dots, i_k, j)$. In general, writing a string of pinnings and/or vertices will denote the pinning derived by concatenating the elements of the string. E.g.~$\Omega_{ijk}$ where $i,j,k\in\Omorproj$ is equivalent to $\Omega_{\eta}$ where $\eta = (i, j, k)$. Let $\piorproj_{\eta}(j) = \frac{\pi(\Omega_{\eta j})}{\pi(\Omega_{\eta})}$.

When considering a function $f: \Omega \rightarrow \mathbb{R}$ we will write $\FF(\eta) = \Exp_{\pi_{\eta}} [f]$.

We begin by restating two properties of Catalan numbers that are central to~\cite{eppstein2022improved}:

\begin{lemma}[\cite{eppstein2022improved}]
If $0 \leq k < l$ then
\[
    \frac{C_k}{C_{k+1}} > \frac{C_l}{C_{l+1}} \,.
\]
\label{lem:catalanmono}
\end{lemma}

The following is a corollary of \cref{lem:catalanmono}:

\begin{lemma}[\cite{eppstein2022improved},               Lemma 10]
For all $j, i_1,\ldots, i_s$, writing $\eta = (i_1, i_2, \dots, i_s)$ and $\eta' = \eta i_{s+1} = (i_1, i_2, \dots, i_s, i_{s+1})$, we have $\piorproj_{\eta'}(j) > \piorproj_{\eta}(j)$.
\label{lem:pigood}
\end{lemma}

Intuitively, \Cref{lem:pigood} says the probability having a triangulation with a specific vertex forming a triangle with a given side increases as the polygon becomes smaller. For example, this becomes clear in the extremal case outlined in \Cref{fig:lem_incr}.

\begin{figure}[h]
    \centering
    \includegraphics[width=0.4\linewidth]{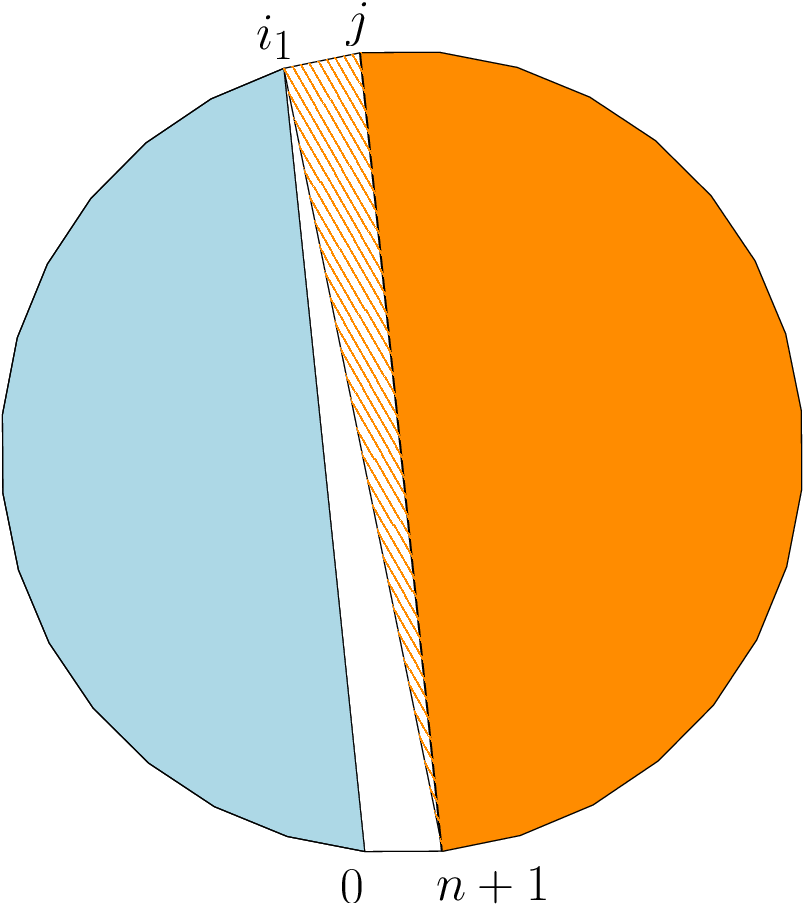} \hspace{1cm}
    \includegraphics[width=0.4\linewidth]{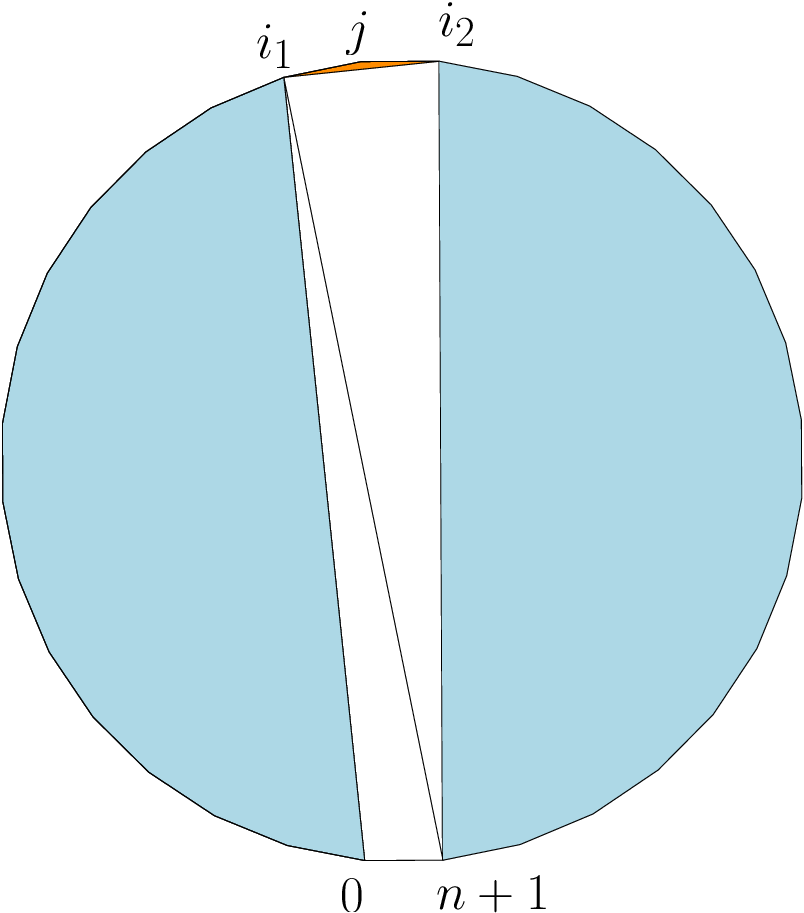}

    \caption{When $\eta:=(i_1)$ and $j:=i_1+1$, $\piorproj_\eta(j)$ equals the proportion of triangulations of the orange region that contain the shaded triangle. This is the ratio of two consecutive Catalan numbers. When $\eta'=(i_1,i_2)$ with $i_2:=i_1+2$, the triangle $i_1ji_2$ is always contained, so $\piorproj_{\eta'}(j)=1$. }
    \label{fig:lem_incr}
\end{figure}

Building upon the proof of \cref{lem:pigood} one can obtain the following additional monotonicity property. Recall that the states in $\Omorproj$ are indexed $1, 2, \dots, n$ and thus admit a natural total ordering.

The following lemma states that if $t$ is farther away from $k_d$ than $s$ is from $k_d$, then the increase in conditional probability of the triangle with vertex $t$  occurring as one extends $\eta$ to $\eta'$ is greater than the corresponding increase for $s$.

\begin{lemma}
    Given a pinning $\eta = (k_1, k_2, \dots, k_{d-1}, k_d)$, and another pinning $\eta' = (k_1, \dots, k_{d+1})$, if $s, t \in \Omorproj_{\eta', r}$ and $s < t$  we have
    \[
        \frac{\piorproj_{\eta}(s)}{\piorproj_{\eta'}(s)} \leq \frac{\piorproj_{\eta}(t)}{\piorproj_{\eta'}(t)}\,.
    \]
    and the same holds if $s, t \in \Omorproj_{\eta, l}$ and $s > t$.
    \label{lem:pimono}
\end{lemma}
\begin{proof}
    Consider the convex polygon $P$ induced by~$\eta$ with respect to which $\Omorproj_{\eta, r}$ (or $\Omorproj_{\eta, l}$) is defined, in which the innermost triangle of~$\eta'$ (the one with vertex~$k_{d+1}$) lies.

    Let $P'$ be the polygon with respect to which $\Omorproj_{\eta', r}$ is defined.
    
    Without loss of generality, let $s,t \in \Omorproj_{\eta', r}$ and let $s < t$. Suppose further that $t = s + 1$, as it suffices to prove the claim in this case.

    \begin{align*}
        \piorproj_{\eta}(s) &= \frac{C_{a}C_{k-a-2}}{C_{k-1}}\,, \mbox{ and }\\
        \piorproj_{\eta}(t) &= \frac{C_{a-1}C_{k-a-1}}{C_{k-1}}\,,
    \end{align*}
    where $k + 1$ is the number of sides of $P'$, and $1 \leq a \leq k - 1$.

    Similarly
    \begin{align*}
        \piorproj_{\eta'}(s) &= \frac{C_{a}C_{l-a-2}}{C_{l-1}}\,, \mbox{ and}\\
        \piorproj_{\eta'}(t) &= \frac{C_{a-1}C_{l-a-1}}{C_{l-1}}\,,
    \end{align*}
    where $l + 1 > k + 1$ is the number of sides of $P$.

    Therefore
    \begin{equation*}
        \frac{\piorproj_{\eta}(s)\piorproj_{\eta'}(t)}{\piorproj_{\eta}(t)\piorproj_{\eta'}(s)} = \frac{C_{k-a-2}C_{l-a-1}}{C_{k-a-1}C_{l-a-2}} \leq 1 \,,
    \end{equation*}
    by \cref{lem:catalanmono}.
\end{proof}

The ratios being compared in \cref{lem:pimono} will determine the congestion incurred across certain sets of edges in our transport flow. \cref{lem:pimono} states that as one moves around the vertices of the polygon, these factors increase monotonically. This will be important in obtaining our average congestion bound.

The following is immediate from \cref{rmk:perfmat}:
\begin{lemma}
    For every pinning $\eta$ and for all $i,j \in \Omorproj_{\eta, l}$ (similarly for all $i,j \in \Omorproj_{\eta, r}$), there is a perfect matching between $\Omega_{\eta ij}$ and $\Omega_{\eta ji}$.
    \label{lem:matching}
\end{lemma}

We now strengthen \cref{lem:perfmat}; this will be a key ingredient in passing from the Dirichlet form of~$F$ (defined over the projection chain) to the Dirichlet form of the overall function~$f$:

\begin{lemma}
\label{lem:ijji}
Let $\eta$ be a pinning. Suppose $i,j \in \Omorproj_{\eta, l}$, or $i,j \in \Omorproj_{\eta, r}$. Then
\begin{align}
    \piorproj(\eta i)\piorproj_{\eta i}(j)(\FF(\eta ij) - \FF(\eta j i)) &= \sum_{x \in \Omega_{\eta ij}, y \in \Omega_{\eta ji}} \Delta \pi(x)P(x, y)(f(x) - f(y))\,.
\end{align}
\end{lemma}
\begin{proof}
    Proceeding as in the proof of \cref{lem:perfmat} we get
    \begin{align*}
        \piorproj(\eta i)\piorproj_{\eta i}(j)(\FF(\eta ij) - \FF(\eta j i)) &= \piorproj(i)\Porproj(j)\left(\sum_{x \in \Omega_{\eta i j}}\frac{\pi(x)}{\Delta \piorproj_{\eta}(i)\Porproj_{\eta}(i, j)} f(x) - \sum_{y \in \Omega_{\eta j i}}\frac{\pi(y)}{\Delta \piorproj(j)\Porproj(j, i)}f(y)\right) \\
        &= \piorproj(i)\Porproj(i, j)\sum_{x \in \Omega_{\eta ij}, y \in \Omega_{\eta j i}: P(x, y) > 0} \frac{\pi(x)P(x, y)}{\piorproj(i)\Porproj(i, j)}(f(x) - f(y)) \\
        &= \sum_{x \in \Omega_{\eta i j}, y \in \Omega_{\eta j i}}\pi(x)P(x, y)(f(x) - f(y))\,.
    \end{align*}
\end{proof}

\begin{lemma}
    \label{lem:iij}
    Let $\eta$ be a pinning. Suppose $i,j \in \Omorproj_{\eta, l}$ or $i,j \in \Omorproj_{\eta, r}$. Suppose $j > i$ without loss of generality. Then
\begin{align}
    \FF(\eta i) - \FF(\eta ij) &= \sum_{k \in \Omorproj_{\eta i,r}} \piorproj_{\eta i}(k)(\FF(\eta ik) - \FF(\eta i j))\,. \label{eqn:etaik}
\end{align}
\end{lemma}

\begin{proof}
Using the definition of $\FF(\cdot)$ and $\piorproj_{(\cdot)}(\cdot)$ we obtain the identities
$\FF(\eta i)~=~\sum_{k \in \Omorproj_{\eta i, r}}~\piorproj_{\eta i}(k)\FF(\eta i k)$
and $\sum_{k \in \Omorproj_{\eta i, r}} \piorproj_{\eta i}(k) = 1$. \eqref{eqn:etaik} follows.
\end{proof}

In the proof of \cref{lem:l1dirweak} we will construct a flow in which states in $\Omega_i$ send flow to states in $\Omega_j$. In this flow, a commodity is initially uniformly distributed throughout $\Omega_i$ and must be sent along edges (transitions) in the state space $\Omega$ so that it is uniformly distributed throughout $\Omega_j$. The expression
\[
\piorproj(i)\piorproj(j)(\FF(i) - \FF(j))
\]
captures the ``problem'' of routing this flow from $\Omega_i$ to $\Omega_j$. Constructing a flow is equivalent to decomposing this expression into a summation over expressions 
\[\phi_{ij,xy}\pi(x)P(x, y)(f(x) - f(y))\quad\textrm{or}\quad\phi_{ji,xy}\pi(x)P(x, y)(f(x) - f(y))\,,\]

with $\phi_{ij,xy}$ (or $\phi_{ji,xy}$) giving the flow across each $(x,y)$ edge (and thus $\Delta\phi_{ij,xy}$ or $\Delta\phi_{ji,xy}$ giving the congestion across each edge). In fact, the $\phi_{ij,xy}$, $\phi_{ji,xy}$ values will obey the axioms of an $\Omega_i$-$\Omega_j$ flow as described in \cref{sec:mcflow}.

I.e.~the net flow out of each $x \in \Omega_i$ will be $\pi(x)\piorproj(j)$, the net flow into each $y \in \Omega_j$ will be $\piorproj(i)\pi(y)$, i.e.
\[
    \sum_{y \sim x}\phi_{ij,xy} = \pi(x)\piorproj(j)\,,
\]
for all $x \in \Omega_i$
and
\[
    \sum_{x \sim y}\phi_{ji,xy} = \piorproj(i)\pi(y)\,,
\]
for all $y \in \Omega_j$.

Also, we will have $\phi_{ij,xy} = -\phi_{ij,yx}$.

The way we will accomplish this decomposition is to recursively decompose $\Omega$ by first partitioning using the projection chain $\Omorproj$, then recursively decomposing $\Omega_i$ for each $i \in \Omorproj$. 

\begin{definition}
    Given $\eta = (k_1 = i, k_2, \dots, k_d)$ define a function $\phi_{ij,\eta s t}$ recursively so that:
\begin{itemize}
    \item $\phi_{ij,ij} = \frac{\piorproj(j)}{\piorproj_i(j)}$\,,
    \item  $\phi_{ji,ji} = -\phi_{ij,ij}$\,,
    \item  $\phi_{ij,ij'} = 0$ for all $j' \neq j$, and
    \item $\phi_{ji,ji'} = 0$ for all $i' \neq i$\,.
\end{itemize}

For all $s, t \in \Omorproj_{\eta, l}$ (similarly for all $s, t \in \Omorproj_{\eta, r}$) let
\begin{align}
   \phi_{ij, \eta s t} = \frac{\piorproj_{\eta}(t)}{\piorproj_{\eta s}(t)}(\phi_{ij, \eta t} - \phi_{ij, \eta s}) = \frac{\piorproj_{\eta}(s)}{\piorproj_{\eta t}(s)}(\phi_{ij, \eta t} - \phi_{ij, \eta s}) \,,\label{eqn:rhoetast}
\end{align}
and whenever $\eta' = \eta k_{d+1}$ let $\phi_{ij,\eta' s} = \phi_{ij,\eta k_{d+1} s}$.

For all $x \in \Omega_{\eta st}, y \in \Omega_{\eta t s}$ let    $\phi_{ij,xy} = \phi_{ij,\eta s t}$.

Finally, for $x \in \Omega_i, y \in \Omega_j, P(x, y) > 0$, let $\phi_{ij,xy} = \phi_{ij,ij}$.
\end{definition}

This function will describe the flow across each edge in our construction. In particular, $\phi_{\eta st}$ describes the flow across edges $(x, y)$ where $x \in \Omega_{\eta s}, y \in \Omega_{\eta t}$.  

Conceptually, at each step of the decomposition, we are within some $\Omega_{\eta}$. We imagine that $\phi_{ij,\eta s}$ flow is concentrated within states in $\{\Omega_{\eta s} \mid s \in \Omorproj_{\eta, l}\}$ (and similarly for $\Omorproj_{\eta, r}$), which must be distributed to the rest of $\Omega_{\eta}$. This reduces to the problem of distributing the flow to each $\Omega_{\eta t}, t \in \Omorproj_{\eta, l}$. The resulting flow across the matching $\Omega_{\eta st}, \Omega_{\eta t s}$ is precisely 
\begin{align}
    \phi_{ij, \eta s}  \cdot \frac{\piorproj_{\eta}(s)\piorproj_{\eta}(t)}{\piorproj_{\eta}(s t)}  &= \phi_{ij, \eta s}\cdot \frac{\piorproj_{\eta}(t)}{\piorproj_{\eta s}(t)} = \phi_{ij, \eta s}\cdot \frac{\piorproj_{\eta}(s)}{\piorproj_{\eta t}(s)}\,. \label{eqn:rhodistmat}
\end{align}
This is because the total flow concentrated in $\Omega_{\eta s}$ is $\phi_{ij, \eta s}\piorproj(\eta)\piorproj_{\eta}(s)$, and a $\piorproj_{\eta}(t)$ portion of this flow must be sent to $\Omega_{\eta t}$. Distributing this flow across the edges in the matching $(\Omega_{\eta s t}, \Omega_{\eta t s})$, each side of which has measure $\piorproj(\eta)\piorproj_{\eta}(st)$, justifies the left-hand side of \eqref{eqn:rhodistmat}. The right-hand side follows from the definition of $\piorproj_{(\cdot)}(\cdot)$.

We obtain \eqref{eqn:rhoetast} from considering the symmetric problem of distributing flow from $\Omega_{\eta t}$ to $\Omega_{\eta s}$.

Before presenting the formal construction, we prove the following lemmas.
\begin{lemma}
    For all $i,j\in\Omorproj$, for all pinnings $\eta = (i = k_1, k_2, \dots, k_d)$ such that $j \sim \eta$,
    \[
      \phi_{ij,\eta j} = \frac{\piorproj(j)}{\piorproj_{\eta}(j)}\,.
    \]
    \label{lem:rhoj1}
\end{lemma}

The following will ensure that $\phi_{ij,\eta t} - \phi_{ij,\eta s}$ is always bounded in absolute value by $1$:
\begin{lemma}
For all $i,j\in\Omorproj$, for all pinnings $\eta = (i = k_1, k_2, \dots, k_d)$ such that $j = k_u$  for some $u \in [d]$, and for all $s, t \in \Omorproj_{\eta,r}$ where $s < t$ (similarly for all $s, t \in \Omorproj_{\eta, l}$ where $t < s$):
either
\begin{equation}
    0 \leq \phi_{ij, \eta s} \leq \phi_{ij, \eta t} \leq 1 \label{eqn:mono+}
\end{equation}
or
\begin{equation}
    0 \geq \phi_{ij, \eta s} \geq \phi_{ij, \eta t} \geq -1\,. \label{eqn:mono-}
\end{equation}
\label{lem:rholeq1}
\end{lemma}
\begin{proof}
We proceed by induction on $d$, the number of vertices in $\eta$. Suppose $j = k_d$, Let $\tau = (k_1, \dots, k_{d-2}), \tau' = (k_1, \dots, k_{d-1}).$ Then for all $s, t \in \Omorproj_{\eta, r}, s < t$ we have $\phi_{ij,\tau' s} = \phi_{ij, \tau' t} = 0$ and $\phi_{ij, \tau' j} = \frac{\piorproj(j)}{\piorproj_{\tau'}(j)}$ (by \cref{lem:rhoj1}) and therefore
\begin{align}
  \phi_{ij,\eta s} &= \phi_{ij, \tau' j s} \\
  &= \frac{\piorproj_{\tau'}(s)}{\piorproj_{\tau' j}(s)}(\phi_{ij, \tau' s} - \phi_{ij, \tau' j}) \\
  &= -\frac{\piorproj(j)}{\piorproj_{\tau'}(j)}\cdot \frac{\piorproj_{\tau'}(s)}{\piorproj_{\tau' j}(s)}\,.
\end{align}

Similarly $\phi_{ij,\eta t} = -\frac{\piorproj(j)}{\piorproj_{\tau'}(j)}\cdot \frac{\piorproj_{\tau'}(s)}{\piorproj_{\tau' j}(s)}\,.$

The claim then follows from \cref{lem:pimono}.

For the inductive step, let $\eta = (i = k_1, k_2, \dots, k_d)$ be a pinning such that $j \in \eta$, and let $\eta' = \eta k_{d+1}$ be a pinning. Suppose $s, t \in \Omorproj_{\eta', r}$ and $s < t$ (the proof is symmetric in the other case).

Let $\tau = \eta \setminus \{k_d\}$.

By \cref{lem:pimono} we have 
\begin{equation}
    \frac{\piorproj_{\tau}(k_{d+1})}{\piorproj_{\eta}(k_{d+1})}\leq \frac{\piorproj_{\tau}(s)}{\piorproj_{\eta}(s)} \leq \frac{\piorproj_{\tau}(t)}{\piorproj_{\eta}(t)}\,.
\end{equation}
By the inductive hypothesis either \eqref{eqn:mono+} or \eqref{eqn:mono-} holds for the pairs $k_{d+1}, s$ and $s,t$.

Then, from the definition:
\begin{align*}
    \phi_{ij,\eta's} &= \frac{\piorproj_{\eta}(s)}{\piorproj_{\eta'}(s)}(\phi_{ij,\eta s} - \phi_{ij,\eta'})\,, \nonumber \\
\end{align*}
and similarly 
\begin{align*}
   \phi_{ij,\eta' t}  &= \frac{\piorproj_{\eta}(t)}{\piorproj_{\eta'}(t)}(\phi_{ij,\eta t} - \phi_{ij,\eta'})\,.
\end{align*}

Furthermore, by the inductive hypothesis we have
\begin{align}
    0 \leq \phi_{ij,\eta s} - \phi_{ij,\eta'} \leq \phi_{ij,\eta t} - \phi_{ij,\eta'} \leq 1
\end{align}
or
\begin{align}
    0 \geq \phi_{ij,\eta s} - \phi_{ij,\eta'} \geq \phi_{ij,\eta t} - \phi_{ij,\eta'} \geq -1\,,
\end{align}
and by \cref{lem:pimono} we have
\begin{align}
    \frac{\piorproj_{\eta}(s)}{\piorproj_{\eta'}(s)} < \frac{\piorproj_{\eta}(t)}{\piorproj_{\eta'}(t)}\,.
\end{align}
One obtains one of $0 \leq \phi_{\eta' s} \leq \phi_{\eta' t} \leq 1$ or $0 \geq \phi_{\eta' s} \geq \phi_{\eta' t} \geq -1$ by observing that for any real $\rho_1, \rho_2$ such that $0 \leq \rho_1 \leq \rho_2 \leq 1$ and for any $0 < \pi_1 \leq \pi_2 \leq 1$  we have

\[
  0 \leq \pi_1 \rho_1 \leq \pi_2 \rho_2 \leq 1\,.
\]

\end{proof}

To prove \cref{lem:l1dirweak}, we will introduce the following notation for convenience:

\begin{definition}
    Given a pinning $\eta$ and given $s, t \in \Omorproj_{\eta, l}$ (similarly given $s, t \in \Omorproj_{\eta, r}$) define the following:

    \begin{itemize}
        \item     Let $\flo{\eta s}{\eta t} = \piorproj(\eta)\piorproj_{\eta}(s)\piorproj_{\eta}(t)(\FF(\eta s) - \FF(\eta t))$.
        \item Let $\flo{\eta s t}{\eta t s} = \piorproj(\eta)\piorproj_{\eta}(s)\piorproj_{\eta s}(t)(\FF(\eta s t) - \FF(\eta t s))$.
        \item Define $\flo{\eta}{\eta s} = \piorproj(\eta)\piorproj_{\eta}(s)(\FF(\eta) - \FF(\eta s))$.
    \end{itemize}
\end{definition}

Intuitively, $\flo{\eta s}{\eta t}$ is a term that describes the weighted difference of $\FF$ between $\eta s$ and $\eta t$.

\begin{proof}[Proof of \cref{lem:l1dirweak}]
    Recall that our goal is to send the flow $\flo{i}{j}$ from $\Omega_i$ to $\Omega_j$ and bound its congestion. The main idea is that, by an appropriate rerouting and decomposition of the flow, the original problem is reduced to a set of equivalent problems of smaller size. A recursive application of this idea will result in the proof of the lemma.

    We begin by applying the flow rerouting $\Om_i\rar\Om_{ij}\rar\Om_{ji}\rar\Om_{j}$ to get
    \begin{align}
        \flo{i}{j} &= 
                \phi_{ij,ij}\flo{i}{ij} \nonumber \\
        &\qquad + \phi_{ij,ij}\flo{ij}{ji} \nonumber \\
        &\qquad + \phi_{ij,ij}\flo{ji}{j}  \nonumber \\
       &= \phi_{ij,ij}\flo{i}{ij} \label{eqn:iij} \\
        &\qquad + \phi_{ij,ij}\flo{ij}{ji} \label{eqn:ijji}  \\
        &\qquad + \phi_{ji,ji}\flo{j}{ji} \label{eqn:jij} 
    \end{align}

    (and note $\phi_{ij,ij} = \frac{\piorproj(j)}{\piorproj_i(j)} = \frac{\piorproj(i)}{\piorproj_j(i)} = -\phi_{ji,ji}$, so \eqref{eqn:ijji} is in fact symmetrically consistent).

    First we consider \eqref{eqn:ijji}. By \Cref{lem:ijji}, we immediately get
    \begin{equation}\label{eqn:ijjidir}
        \frac{\piorproj(j)}{\piorproj_i(j)}\flo{ij}{ji}\ = \sum_{x\in\Om_i, y \in \Om_j}\Delta\phi_{ij,xy}\pi(x)P(x, y)\left(f(x) - f(y)\right)\,,
    \end{equation}
    where $0 \leq \phi_{ij,xy} = \phi_{ij,ij} = \frac{\piorproj(j)}{\piorproj_i(j)} \leq 1.$    

    Next we observe that \eqref{eqn:iij} and \eqref{eqn:jij} are symmetric. Thus it suffices to consider \eqref{eqn:iij}. Without loss of generality assume $j > i$. Every triangulation in $\Omega_i$ has to triangulate $\Omega_{i,r}$ and thus we can apply the flow decomposition $\cup_{k\in\Omorproj_{i,r}}\Om_{ik}\rar \Om_{ij}$ to get
    \begin{equation} \label{eqn:ikij}
        \frac{\piorproj(j)}{\piorproj_i(j)}\flo{i}{ij}=\frac{\piorproj(j)}{\piorproj_i(j)}\sum_{k\in\Omorproj_{i,r}}\flo{ik}{ij} = \sum_{k \in \Omorproj_{i,r}} \phi_{ij,ij}\flo{ik}{ij}\,.
    \end{equation}
    
    Note that the problem of sending flow from $\Om_{ik}$ to $\Om_{ij}$ can be viewed as the problem of sending flow from ``$\Om_k$'' to ``$\Om_j$'' in the smaller polygon that has $i0$ as a side and contains $k,j$. We thus repeat the equivalent rerouting as the one in the beginning of the proof, that is $\Om_{ik}\rar\Om_{ikj}\rar\Om_{ijk}\rar\Om_{ij}$ to get
    
    \begin{align}\label{eqn:ikij_decomp}
        \flo{ik}{ij}&=\frac{\piorproj_i(j)}{\piorproj_{ik}(j)}\flo{ik}{ikj}+\frac{\piorproj_i(j)}{\piorproj_{ik}(j)}\flo{ikj}{ijk}+\frac{\piorproj_i(k)}{\piorproj_{ij}(k)}\flo{ijk}{ij}\,.
    \end{align}
    Substituting \eqref{eqn:ikij_decomp} into \eqref{eqn:ikij} we get
    \begin{align}
        \frac{\piorproj(j)}{\piorproj_i(j)}\flo{i}{ij}&=\frac{\piorproj(j)}{\piorproj_i(j)}\sum_{k\in\Omorproj_{i,r}} \left[ \frac{\piorproj_i(j)}{\piorproj_{ik}(j)}\flo{ik}{ikj}+\frac{\piorproj_i(j)}{\piorproj_{ik}(j)}\flo{ikj}{ijk}+\frac{\piorproj_i(j)}{\piorproj_{ik}(j)}\flo{ijk}{ij}\right] \nonumber\\
        &=\sum_{k\in\Omorproj_{i,r}}\phi_{ij,ikj}\flo{ik}{ikj} \label{eqn:ikikj}\\
        &\qquad + \sum_{k\in\Omorproj_{i,r}}\phi_{ij,ikj}\flo{ikj}{ijk} \label{eqn:ikjijk}\\
        &\qquad +\sum_{k\in\Omorproj_{i,r}}\phi_{ij,ijk}\flo{ij}{ijk} \,.\label{eqn:ijkij} 
    \end{align}

Observe that, again, there is a clear analogy between the pairs of equations \eqref{eqn:ikikj}-\eqref{eqn:iij}, \eqref{eqn:ikjijk}-\eqref{eqn:ijji}, \eqref{eqn:ijkij}-\eqref{eqn:jij}.

Iterating the previous steps, one obtains a decomposition of the original term $\flo{i}{ij}$ into expressions of the form \eqref{eqn:ikjijk}-\eqref{eqn:ijji}. To prove the lemma, it suffices to bound the coefficients $\phi_{ij,xy}$ on the right-hand side.

Furthermore, this decomposition preserves the invariant that the coefficient of the term $\flo{i}{ij}$ in \eqref{eqn:ijji}\textemdash and similarly $\flo{ikj}{ijk}$ in \eqref{eqn:ikjijk} and all other terms of the same form\textemdash is bounded in absolute value by one.

In general, the coefficients are of the form
\begin{equation}
    \label{eqn:rhoform}
    \phi_{ij,\eta st} = \frac{\piorproj_{\eta}(t)}{\piorproj_{\eta s}(t)}(\phi_{ij,\eta t} - \phi_{ij, \eta s})
\end{equation}
(as in the definition)\,,
which is bounded in absolute value by one, by \cref{lem:rholeq1}.

\eqref{eqn:rhoform} can be justified as follows: first, note the coefficient $\phi_{ij,xy} = \phi_{ij,ij} = \frac{\piorproj(j)}{\piorproj_i(j)} $ in \eqref{eqn:ijji}. 

For a given $\eta$ and $\eta' = \eta u$, and for a given $s, t \in \Omorproj_{\eta', r}$, the term $\flo{\eta' s}{\eta' t}$ arises from equalities of the following form

\begin{align}
    \sum_t \phi_{ij,\eta t}\flo{\eta}{\eta t} &= \sum_t\sum_s \phi_{ij,\eta t}\flo{\eta s}{\eta t} \nonumber \\
    &= \sum_{s,t: s < t}(\phi_{ij,\eta t} - \phi_{ij, \eta s})\flo{\eta s}{\eta t} \nonumber \\
    &= \sum_{s,t:s<t}(\phi_{ij,\eta t} - \phi_{ij, \eta s})\left[\frac{\piorproj_{\eta}(t)}{\piorproj_{\eta s}(t)}\flo{\eta s}{\eta s t} + \frac{\piorproj_{\eta}(t)}{\piorproj_{\eta s}(t)}\flo{\eta s t}{\eta t s} + \frac{\piorproj_{\eta}(s)}{\piorproj_{\eta t}(s)}\flo{\eta t s}{\eta t}\right] \nonumber \\
    &= \sum_{s,t: s < t}\left[\phi_{ij,\eta s t}\flo{\eta s}{\eta st} + \phi_{ij,\eta s t}\flo{\eta st}{\eta ts} + \phi_{ij, \eta t s}\flo{\eta t}{\eta t s}\right]\,.
\end{align}

The function~$\phi_{ij}$ describes the desired $\Omega_i$-$\Omega_j$ flow. Furthermore, the maximum congestion is at most~$\Delta$ since $\phi_{ij}$ takes on values bounded in absolute value by one, and since $P(x, y) = \frac{1}{\Delta}$ whenever $P(x, y) > 0$.. Finally, the existence of the desired transport flow follows from applying \cref{lem:flowtotransportflow} to~$\phi_{ij}$.
\end{proof}

\section{Bounding the Maximum and Average Congestion}
\label{sec:avgcond}

First we give a bound on the maximum congestion, which we will prove in \cref{sec:deferredproofs}:

\begin{lemma}
    In the flow construction in \cref{lem:l1dirweak}, let $x,y \in \Omega$. Then
    
    \[
        \Bigl|\sum_{i \in S, j \in T}\phi_{ij,xy}\Bigr| \leq 1\,.
    \]
    \label{lem:maxcong}
\end{lemma}

The rest of this section is devoted to proving a bound on the average congestion.

Given a triangulation pair~$x, y$ satisfying $P(x, y) > 0$, there exists a unique maximal pinning $\eta$ such that $x,y \in \Omega_{\eta}$. We will call this pinning $\eta_{xy}$. See \cref{fig:etaxy}. 

\begin{figure}
    \centering
    \includegraphics[width=0.4\linewidth]{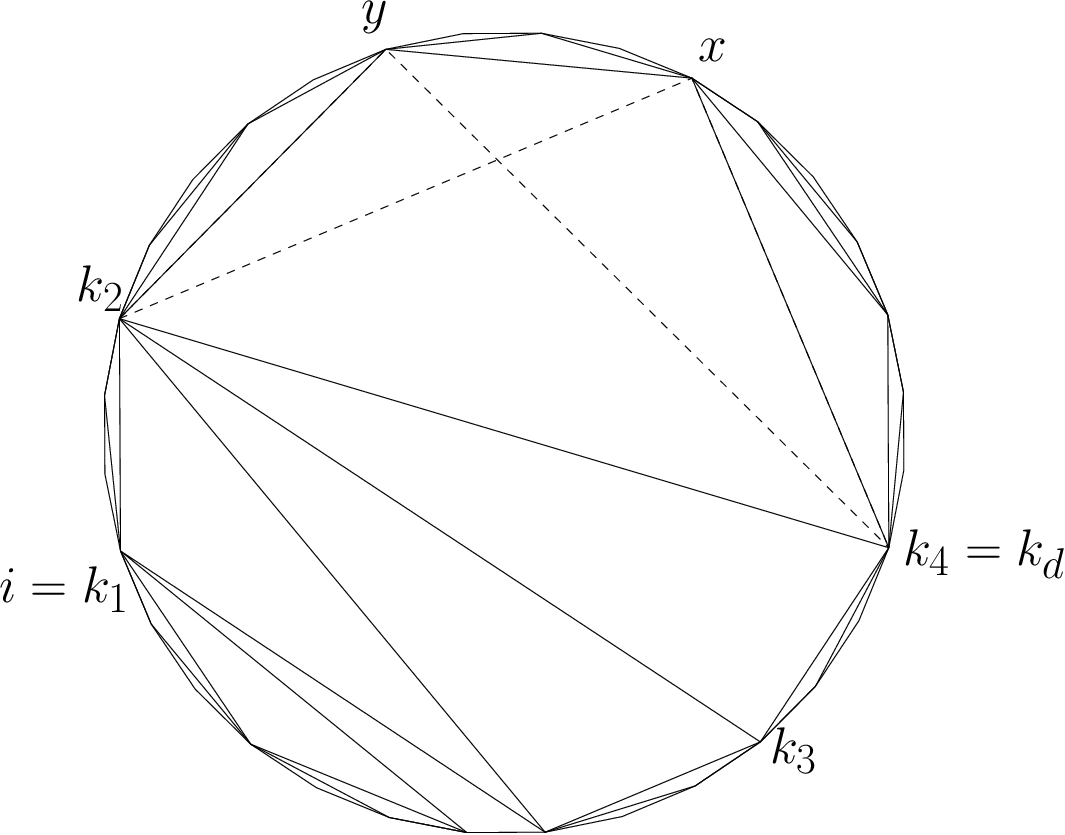}
    \caption{Two triangulations $x,y\in\Omega_{\eta}$, where $\eta = \eta_{xy} = (k_1, k_2, k_3, k_d)$.}
    \label{fig:etaxy}
\end{figure}

Given $x \in \Omega$, each diagonal in $x$ induces a transition to $y \in \Omega$, and thereby induces a pinning $\eta_{xy}$. Furthermore, using the duality between Catalan triangulations and Catalan trees, let $x^*$ be the tree that corresponds to the triangulation~$x$. One can identify each diagonal in~$x$\textemdash and therefore each pinning~$\eta_{xy}$\textemdash with a tree edge in~$x^*$ between a parent node and a child node. We will define $d_x(\eta) = d_x(\eta_{xy})$ to be the depth of the child node of this tree edge.

Our strategy in this section will be as follows: let $i,j \in \Omorproj$ be as in \cref{lem:l1dirweak}. We observe that given $x \in \Omega_i$, and given a pinning $\eta~=~(k_1 = i, \dots, k_d),$ there are (at most) two additional pinnings $\eta^{(l)} = \eta k_{d+1}^{(l)}$ and $\eta^{(r)} = \eta k_{d+1}^{(r)}$ such that $k_{d+1}^{(l)} \in \Omorproj_{\eta l}$, $k_{d+1}^{(r)} \in \Omorproj_{\eta r}$, and $x \in \Omega_{\eta^{(l)}}$ and $x \in \Omega_{\eta^{(r)}}$. See \cref{fig:etaxsplit}. 

Given  a pinning $\eta$ and a vertex $k \in \Omorproj_{\eta l}$ or $k \in \Omorproj_{\eta r}$, let $\eta' = \eta k$. Then given $x \in \Omega_{\eta'}$ define $\phi_{ij,\eta'}^{(x)} = \phi_{ij,\eta k}$\textemdash where $\phi_{ij,\eta k}$ is the notation used in \cref{lem:l1dirweak}. We will show that given $i\in S, j \in T$, for each depth level~$d$ we have 
\[
    \sum_{\eta: x \in \Omega_{\eta}, d(\eta) = d} |\phi_{ij,\eta}^{(x)}| \leq 1\,.\]
That is, the sum of the congestion values incurred in the flow in \cref{lem:l1dirweak} across flips at a given level of (the dual tree of) a given triangulation $x \in \Omega_i$ sums to at most one. Formally:

\begin{lemma}
    \label{lem:rhodecomp}
    With the notation of \cref{lem:l1dirweak}, given $i,j\in \Omorproj$ and given  $\eta = (k_1 = i, k_2, \dots, k_d)$ such that $j \in \eta$, let 
    \begin{align*}
        \phi_{ij,l}^*(\eta) &= \max_{s \in \Omorproj_{\eta,l}}\{|\phi_{ij,\eta s}|\}, \\
        \phi_{ij,r}^*(\eta) &= \max_{s \in \Omorproj_{\eta,r} }\{|\phi_{ij,\eta s} |\}\,.
    \end{align*}
    Then  if $\eta' = (k_1, k_2, \dots, k_d, k_{d+1})$ where $k_{d+1} \in \Omorproj_{\eta,l}$ then $\phi_{ij,l}^*(\eta') + \phi_{ij,r}^*(\eta') \leq \phi_{ij,l}^*(\eta)$, and the symmetric fact holds for $k_{d+1} \in \Omorproj_{\eta, r}$.

\end{lemma}

Then we will use the fact that when $x$ is a (uniformly random) triangulation drawn from $\pi$, the expected depth of $x^*$ is $\widetilde O(\sqrt n)$:

\begin{lemma}[\cite{fgor}]
    The probability that a random Catalan tree on $n$ nodes has depth at least $d$ is
    \[
        O\left(n^{3/2}e^{-d^2/(4n)}\right)\,.
    \]
    The probability that a random Catalan tree on $n$ nodes has depth at most $d$ is 
    \[
        O\left(n^{3/2}e^{-\delta n / d^2}\right)\,,
    \]
    for a constant $\delta > 0$.
    \label{lem:heightwhp}
\end{lemma}

\begin{lemma}[\cite{meirmoon}]
    \label{lem:meirmoon}
    There exist constants $C>0$ and~$c$ such that for sufficiently large $n$, the probability that a uniformly chosen node in a uniformly chosen Catalan tree on $n$ nodes has depth $d$ is proportional to $d$ whenever $d \leq C \sqrt n \log^c n$.
\end{lemma}

A straightforward consequence of \cref{lem:heightwhp,lem:meirmoon} is:

\begin{corollary}
    \label{lem:expdepth}
    The expected depth of a random node in a random Catalan tree is $\widetilde \Theta(\sqrt n)$.
\end{corollary}

We will use \cref{lem:expdepth} to bound the total congestion incurred across edges in $x$ as the sum of congestion over $O(\sqrt n)$ levels, which in turn will be $O(\sqrt n)$ total congestion by \cref{lem:rhodecomp}. Since the transition probability in the flip walk between any adjacent $x,y$ pair is $\Theta\left(\frac{1}{n}\right)$, this in turn will imply that the average congestion over the edges (transitions) incident to $x$ is $O\left(\frac{\sqrt n}{n}\right)$. We state this in \cref{lem:ldirstronger}. Roughly speaking, this will, with some algebraic manipulation, enable the shaving of a $\Theta(\sqrt{n})$ factor from our overall average congestion bound. More precisely, we will use \cref{lem:ldirstronger} to prove \cref{lem:ldirstrong}.

\begin{figure}
    \centering
    \includegraphics[width=0.4\linewidth]{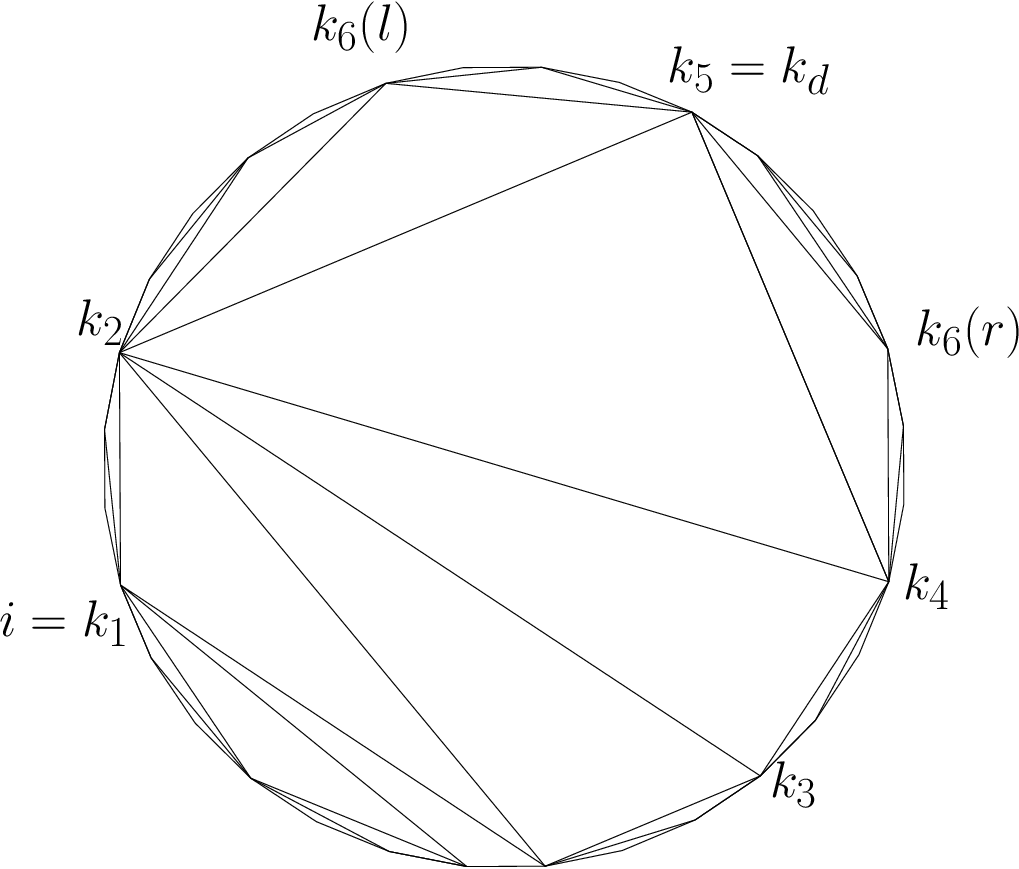}
    \caption{A triangulation $x \in \Omega$, with a pinning $\eta = (k_1, \dots, k_d)$ indicated such that $x \in \Omega_{\eta}$. Here we have $x \in \Omega_{\eta^{(l)}}$ and $x \in \Omega_{\eta^{(r)}}$ where $\eta^{(l)} = \eta k_{d+1}^{(l)} = \eta k_6^{(l)}$ and similarly $x \in \Omega_{\eta^{(r)}}$ where $\eta^{(r)} = \eta k_6^{(r)}$.}
    \label{fig:etaxsplit}
\end{figure}

Let $\mathcal{T}_{ij}$ be the set of all triangles that have one vertex at $j$ and two other vertices on opposite sides of the special edge $(i, n+1)$ (see \cref{fig:Tij}). Given a triangle $t \in \mathcal{T}_{ij}$, let $\Omega_{it} := \{x \in \Omega_i \mid x \textnormal{ includes the triangle } t\}$. The triangle $t$ divides the polygon into three sub-polygons, bounded by the three sides of $t$. Let $l_i(t)$ denote the size of the sub-polygon on the side bounded by the ``base'' of~$t$ (the side closer to the special edge) and the edge $(i,n+1)$. Let $u_i(t)$ denote the size of the larger of the other two sub-polygons.

Let $d(\eta)$ be the depth of (i.e.~the number of triangles in) a given pinning $\eta$. Given $x \in \Omega_i$ let $d_j(x)$ denote the depth of the unique pinning $\eta$ such that $x \in \Omega_{\eta j}$. See \cref{fig:notation_sec_6}.

\begin{figure}[h]
    \centering
    \includegraphics[width=0.4\linewidth]{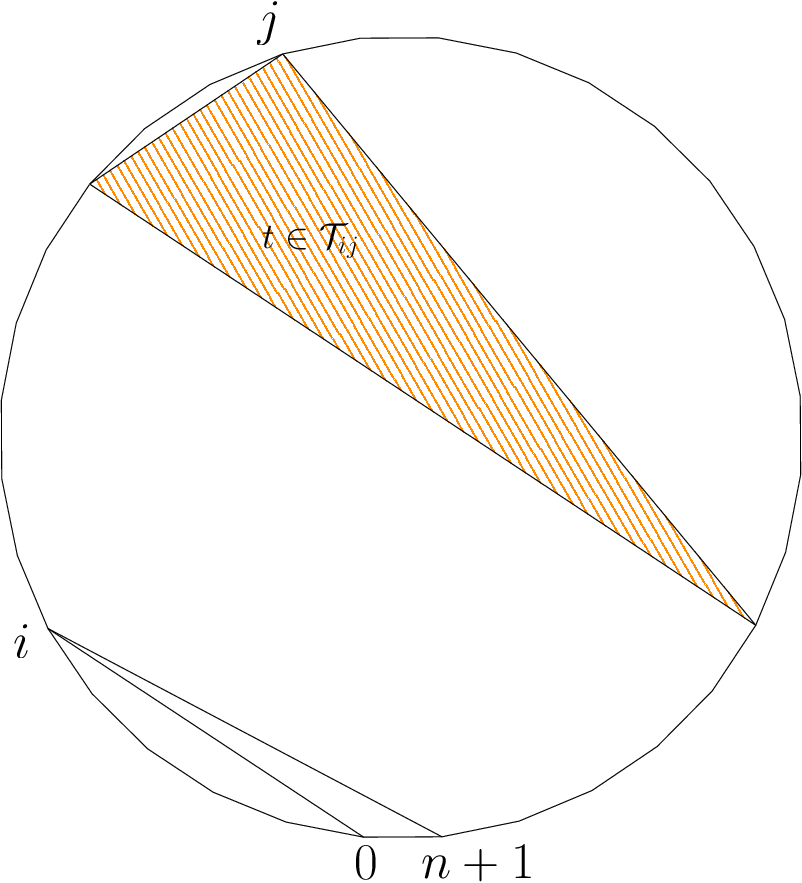} \hspace{2cm}
    \includegraphics[width=0.4\linewidth]{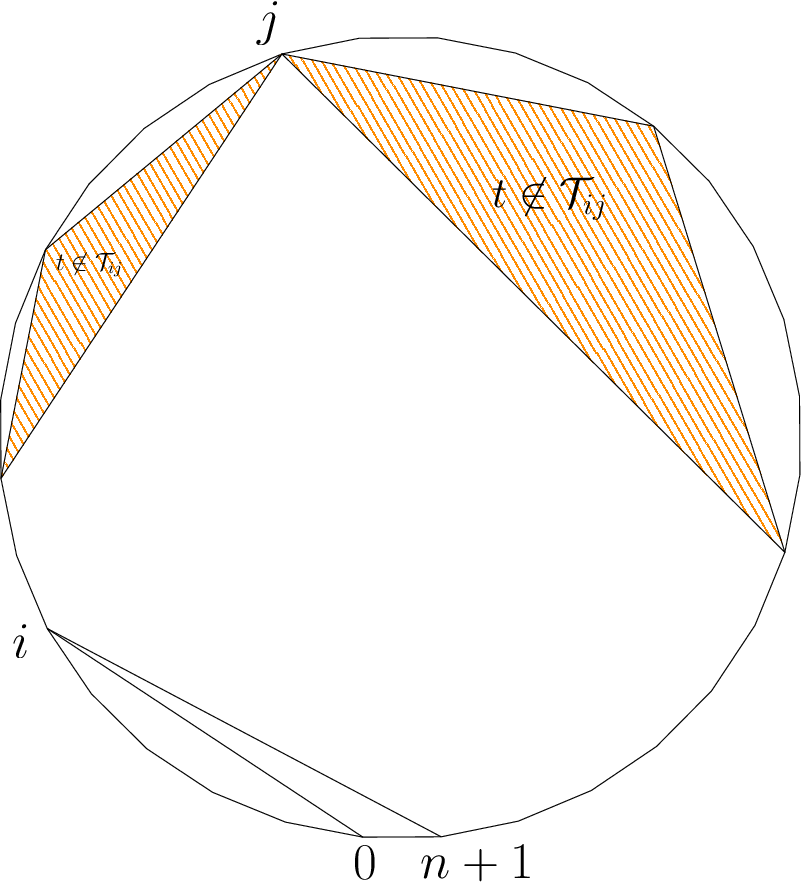}
    \caption{Elements and non-elements of $\mathcal{T}_{ij}$.}
    \label{fig:Tij}
\end{figure}

\begin{figure}
    \centering
    \includegraphics[width=0.5\linewidth]{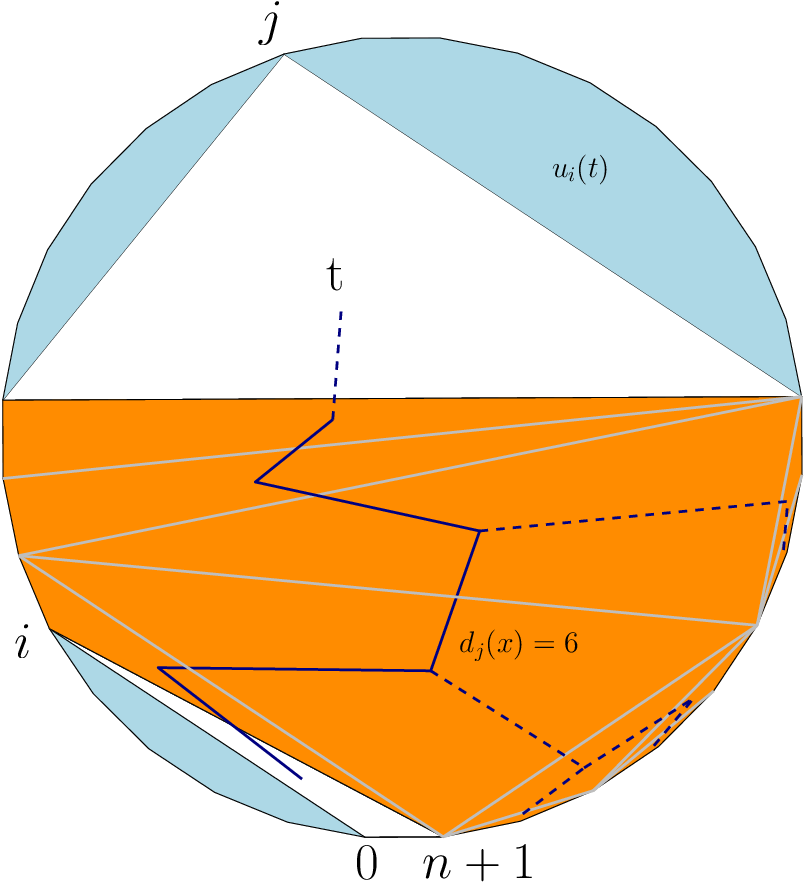}
    \caption{$t$ is a triangle in $\mathcal{T}_{ij}$. Given a triangulation $x\in \Om_i$ containing $t$, we recover the pinning $\eta:x\in\Om_{\eta j}$ by considering the dual graph of the triangulation. The triangles of $\eta$ correspond to the vertices of the (unique) path from $(0,n+1,i)$ to $t$ (excluded). $l_i(t)$ is the size of the orange polygon.}
    \label{fig:notation_sec_6}
\end{figure}

Given a pinning $\eta$ and a vertex $j \notin \eta$, let $j \sim \eta$ if $j \in \Omorproj_{\eta, l}$ or $j \in \Omorproj_{\eta, r}$. Given $t \in \mathcal{T}_{ij}$, let $\eta \sim t$ if $j \sim \eta$ and $\eta$ contains the side of $t$ opposite to $j$.

Then:
\begin{lemma}
    Let $i, j \in \Omorproj$. Let $t \in \mathcal{T}_{ij}$ and $\eta = (i, k_2, k_3, \dots, k_d)$ be a pinning in $\Omorproj_i$ such that $j \sim \eta$ and $\eta \sim t$. Then, there exist constants $C, c > 0$ such that if $n$ is sufficiently large then
    
    \[
        \sum_{x\in \Omega_{\eta j}, y \in \Omega_{\eta}}|\phi_{ij,xy}|\pi(x)P(x, y) \leq \frac{C\log^c n \sqrt{u_i(t)}}{\Delta}\cdot \piorproj(\eta) \piorproj(j)\,.
    \]
    
    \label{lem:ldirstronger}
\end{lemma}
\begin{proof}
    Let $x^*$ denote the dual tree of a given triangulation $x$. Let $x_t^*$ denote the subtree of $x^*$ rooted at the node corresponding to the triangle $t$. Consider a breadth-first partitioning of the non-root nodes in $x_t^*$ into levels $L_1, L_2, \dots, L_d$ where we set $d \leq C'(\log^{c'}u_i(t)) \cdot \sqrt {u_i(t)}$ is either the maximum level in $x_t^*$, or $C'(\log^{c'} u_i(t)) \cdot \sqrt {u_i(t)}$, whichever is smaller. (We defer the choice of $C', c'$ for now.) Identify each node with the neighbor $y$ of $x$ obtained by performing a rotation at the edge connecting the node to its parent. By assumption $j \sim \eta$ (and therefore $j \notin \eta$) let $\hat\phi = \phi_{ij,\eta j}$. By  \cref{lem:rhoj1} and by applying induction to \cref{lem:rhodecomp} we obtain  $\hat\phi = \phi_{ij,\eta j} = \frac{\piorproj(j)}{\piorproj_{\eta}(j)}$ and $\sum_{y \in L_t}|\phi_{ij,xy}| \leq \frac{\piorproj(j)}{\piorproj_{\eta}(j)}$. Therefore letting $L(y)$ denote the level of $y$ we have
    \begin{align}
        \Exp_x \Exp_{y \sim x} |\phi_{ij,xy}| &= \Exp_x \left[\frac{1}{\Delta} \sum_{y \sim x} |\phi_{ij,xy}|\right] \nonumber \\
        &= \Exp_x\left[\frac{1}{\Delta}\sum_{t=1}^d \sum_{y \in L_t}|\phi_{ij,xy}| + \frac{1}{\Delta}\sum_{y: L(y) > d}|\phi_{ij,xy}|\right] \nonumber \\
        &\leq \Exp_x\left[\frac{1}{\Delta}\sum_{t=1}^d\hat\phi\right] + \Exp_x \left[\frac{1}{\Delta}(\textnormal{\# nodes at depth }> d) \cdot\hat\phi\right] \nonumber \\
        &= \frac{d \hat\phi}{\Delta} + \Exp_x \left[\frac{1}{\Delta}(\textnormal{\# nodes at depth }> d) \right]\cdot\hat\phi \nonumber \\
        &\leq \frac{d \hat\phi}{\Delta} + \frac{u_i(t)}{\Delta}\cdot \frac{C'\log^{c'} (u_i(t))}{\sqrt {u_i(t)}}\cdot \hat\phi \label{eqn:depthbound} \\
        &\leq \frac{C\log^c (u_i(t)) \cdot \sqrt {u_i(t)}}{\Delta}\cdot \frac{\piorproj(j)}{\piorproj_{\eta}(j)}\,,\label{eqn:expfinal}
    \end{align}

    where we justify~\eqref{eqn:depthbound} as follows: if $x^*$ is a random tree and $\nu$ is a random node in $x^*$ then by \cref{lem:heightwhp}: 
    \begin{align*}
        \Pr[(\textnormal{depth of }\nu) >d] &\leq \Pr[(\textnormal{depth of }x^*) > d] \\
        &\leq C' n^{3/2}\exp(-d^2/(4u_i(t))) \\
        &\leq C''n^{3/2}\exp(-(C')^2\left(\log^{2c'} u_i(t)\right) /4) \\
        &\leq \frac{1}{\sqrt {u_i(t)}} \,,
    \end{align*}
    whenever $C' \geq \sqrt 2$ and $c' \geq 1/2$. 
    
    Finally, the first claim follows from multiplying \eqref{eqn:expfinal} by $\piorproj(\eta j) = \piorproj(\eta)\piorproj_{\eta}(j)$.
\end{proof}

\begin{lemma}
    \label{lem:numeta}
    There exist constants $C, c > 0$ such that, for sufficiently large $n$ and for given $i,j \in \Omorproj$ and $t \in \mathcal{T}_{ij}$, we have 
    \[
        \sum_{ \eta \sim t} \frac{\piorproj(\eta)}{\pi(\Omega_{it})} \leq C\sqrt {l_i(t)}\log^c n\,.
    \]
\end{lemma}
\begin{proof}

    We compute
    \begin{align*}
          \sum_{\eta \sim t} \frac{\piorproj(\eta)}{\pi(\Omega_{it})} &= \sum_{x\in \Omega_{it}} \sum_{\eta \in \Omorproj_i: j\sim \eta, x \in \Omega_{\eta}}\frac{\pi(x)}{\pi(\Omega_{it})} \nonumber \\
          &= \sum_{x\in \Omega_{it}} \frac{\pi(x)}{\pi(\Omega_{it})}\left(|\{\eta \in \Omorproj_i: j \sim \eta, x \in \Omega_{\eta}\}|\right) \\
          &= \sum_{x \in \Omega_{it}}\frac{\pi(x)}{\pi(\Omega_{it})} d_j(x) \\
          &\leq C \sqrt{l_i(t)}\log^c n\,,
    \end{align*}
    for constants $C,c$.
    The last inequality is by \cref{lem:expdepth} and by the observation that $d_j(x)$ is bounded by the depth of the binary tree corresponding to the triangulation $x$.
\end{proof}

In other words, summing the combined measure of all pinnings $\eta$ such that $\eta j$ is a valid pinning, overcounts many states $x$, as each state belongs to multiple nested pinnings. \cref{lem:numeta} states that the factor of overcounting is $\widetilde O(\sqrt n)$, because the expected depth of the nesting is $\widetilde O(\sqrt n)$.

The following was observed in~\cite{eppstein2022improved, mrs} and follows from straightforward computations using Stirling's approximation and the Catalan formulas:

\begin{lemma}[\cite{eppstein2022improved, mrs}]
    There exist constants $C_1, C_2 > 0$ such that sufficiently large $n$, the probability that a random triangulation of the convex $n+2$-gon contains a given triangle with side lengths $r, s, n+2-r-s$, where $r \leq \frac{n+2}{2}$ and $s \leq \frac{n+2}{2}$, is bounded by
    \[
        \frac{C_1}{(rs)^{3/2}} \leq \sum_{x\in \Omega: x \textnormal{ contains } t}\pi(x) \leq \frac{C_2}{(rs)^{3/2}}\,,
    \]
    recalling that $\pi$ is the uniform measure.
    \label{lem:sidelengths}
\end{lemma}
To simplify notation, we will write that $\sum_{x\in \Omega: x \textnormal{ contains } t}\pi(x) = \Theta\left(\frac{1}{(rs)^{3/2}}\right)$ instead of stating the full inequality in \cref{lem:sidelengths}. We now prove the following:

\begin{lemma}\label{lem:boundlu}
    There exist constants $C,c > 0$ such that, if $n$ is sufficiently large, then given $i,j\in \Omorproj$, 
    \begin{equation*}
        \sum_{t \in \mathcal{T}_{ij}} \pi(\Omega_{it})\sqrt{l_i(t)}\sqrt{u_i(t)} \leq C \sqrt{n}\log^c n \piorproj(i)\,.
    \end{equation*}
\end{lemma}
\begin{proof}
    Each $t \in \mathcal{T}_{ij}$ has three sides: one side has endpoint $j$ and another endpoint $w$ in the interval $[i, j-1]$. Another side has endpoint $j$ and another endpoint $z$ in the interval $[j+1, n+1]$. The third side has endpoints $w$ and $z$, satisfying $i \leq w < j < z \leq n + 1$.

    Assume that $j - i \leq \frac{n+2-i}{2}$. (If $j - i > \frac{n+2-i}{2}$ the proof will be symmetric.) Let $r = j - w$.

    Then:
    \begin{enumerate}[(i)]
        \item if $z - j \leq \frac{n+2-i}{2}$ let $s = z - j$;
        \item if $z - j > \frac{n+2-i}{2}$ let $s = n+2-i - z + w$ and note $s \leq \frac{n+2-i}{2}$.
    \end{enumerate}
    (In case (ii) we are setting $s$ to be the length of the ``bottom'' side of~$t$.) 
    In each case, applying \cref{lem:sidelengths} to the subpolygon bounded by the edge $(i, n+1)$ and containing~$j$, we have 
    
    \[
        \frac{\pi(\Omega_{it})}{\piorproj(i)} = \Theta\left(\frac{1}{(rs)^{3/2}}\right)\,.
    \]

    Furthermore, in the first case we have $u_i(t) = \max\{r,s\}$ and $l_i(t) \leq n + 2 - i$; in the second case we have $u_i(t) \leq n + 2 - i$ and $l_i(t) = s$.

    Therefore
    \begin{align}
        \sum_{t \in \mathcal{T}_{ij}}\frac{\pi(\Omega_{it})}{\piorproj(i)} \cdot \sqrt{l_i(t)}\cdot \sqrt{u_i(t)} 
        & \leq  \sum_{r=1}^{j-i}\Theta\left(\frac{1}{r^{3/2}}\right)\sum_{s=1}^{r}\Theta\left(\frac{1}{s^{3/2}}\right)\sqrt{n+2-i}\sqrt{r} \nonumber \\
        &\qquad+ \sum_{r=1}^{j-i}\Theta\left(\frac{1}{r^{3/2}}\right)\sum_{s=r+1}^{(n+2-i)/2}\Theta\left(\frac{1}{s^{3/2}}\right)\sqrt{n+2-i}\sqrt{s} \nonumber  \\
        &\qquad+  \sum_{r=1}^{j-i} \Theta\left(\frac{1}{r^{3/2}}\right)\sum_{s=j-i-r+1}^{(n+2-i)/2}\Theta\left(\frac{1}{s^{3/2}}\right)\sqrt{s}\sqrt{n+2-i} \nonumber \\
        &= O\left(\sqrt{n}\log n\right)\,. \nonumber
    \end{align}
To see the last asymptotic equality, we note that the first term is bounded by
$$O\left(\sqrt{n}\sum_{r=1}^n\frac{1}{r}\sum_{s=1}^\infty\frac{1}{s^{3/2}}\right)=O(\sqrt{n}\log n),$$
where we used the fact that the harmonic series is bounded by $O(\log n)$ and the series $\sum_s 1/s^{3/2}$ converges. We bound the other terms similarly.

\end{proof}

\begin{lemma}
   In the notation used in \cref{lem:l1dirweak}, given $x,y\in\Omega_i$, if $\phi_{ij,xy} \neq 0$ then there exists a pinning $\eta = (k_1 = i, k_2, \dots, k_d)$ such that $j \sim \eta$ and such that $x \in \Omega_{\eta j}$ or $y \in \Omega_{\eta j}$.
 \label{lem:sumoverj}
\end{lemma}
\begin{proof}
    Recall that we defined $\phi_{ij,xy} = \phi_{ij,\eta s t}$ whenever $x\in \Omega_{\eta s t}$ and $y \in \Omega_{\eta t s}$. From the definition of $\phi_{ij,\eta s t}$ we have that $\phi_{ij, \eta s t} = 0$ (and therefore $\phi_{ij,xy} = 0$) whenever $j \notin \eta$ and $j\neq s$ and $j \neq t$. The claim follows.
\end{proof}

\begin{lemma}
    There exist constants $C,c > 0$ such that if $n$ is sufficiently large, for given $i,j\in \Omorproj$, we have:
    \[
        \sum_{x\in\Omega_i}\sum_{\eta: x \in \Omega_{\eta j}}\sum_{y \in \Omega_i: y \sim x}|\phi_{ij,xy}|\Delta\pi(x)P(x,y) \leq C\sqrt n \log^c n\cdot \piorproj(i)\piorproj(j)\,.
    \]
    \label{lem:numetastrong}
\end{lemma}
\begin{proof} 
    Note that for each $x \in \Omega_i$, there exists a unique $t\in \mathcal{T}_{ij}$ such that $x \in \Omega_{it}$. Also, there exists a unique pinning~$\eta$ such that $x \in \Omega_{\eta j}$. Also if $x \in \Omega_{\eta j}$  then $\eta \sim t$. By \cref{lem:sumoverj} if $x \in \Omega_{\eta j}$ and $y \in \Omega_i$ and $\phi_{ij,xy} \neq 0$ then $y \in \Omega_{\eta}$. Therefore:
    \begin{align}
        \sum_{x \in \Omega_i} \sum_{\eta: x \in \Omega_{\eta j}} \sum_{y \in \Omega_i: y \sim x}|\phi_{ij,xy}|\Delta\pi(x)P(x, y) 
        &= \sum_{t\in\mathcal{T}_{ij}}\sum_{\eta \sim t}\sum_{x \in \Omega_{\eta j}}\sum_{y \in \Omega_{\eta}: y \sim x}|\phi_{ij,xy}|\Delta\pi(x)P(x, y)\nonumber \\
        &\leq C'\log^{c'} n\sum_{t\in\mathcal{T}_{ij}}\sum_{\eta \sim t}\piorproj(\eta)\piorproj(j)\sqrt{u_i(t)} \label{eqn:useldirstronger} \\
        &\leq C''\log^{c''} n\sum_{t\in\mathcal{T}_{ij}}\pi(\Omega_{it})\piorproj(j)\sqrt{l_i(t)}\sqrt{u_i(t)} \label{eqn:usenumeta}\\
        &\leq C\log^c n \cdot \sqrt n \piorproj(i)\piorproj(j)\,,\label{eqn:usenumetastrong} 
    \end{align}
    where \eqref{eqn:useldirstronger} is by \cref{lem:ldirstronger}, \eqref{eqn:usenumeta} is by \cref{lem:numeta}, and \eqref{eqn:usenumetastrong} is by \cref{lem:boundlu}, and where $C',C'',c',c''$ come from \cref{lem:ldirstronger} and \cref{lem:numeta}.
\end{proof}

We are now ready to prove \cref{lem:ldirstrong}:
\begin{proof}[Proof of \cref{lem:ldirstrong}]
We construct a transport flow from $\Omega[S]$ to $\Omega[T]$ where $S, T\subseteq \Omorproj$. We will then bound $\bar \rho_{\Omega[S]}$.

Define the flow as follows: consider the flow function $\phi_{ij}: \left\{(x, y) \in (\Omega_i \cup \Omega_j)^2\mid P(x, y) > 0\right\} \rightarrow \mathbb{R}$ in \cref{lem:l1dirweak}. Construct a function $\phi: \left\{(x, y) \in (\Omega_S \cup \Omega_T)^2 \mid P(x, y) > 0\right\} \rightarrow \mathbb{R}$ as follows:

Given $x, y \in \Omega[S] \cup \Omega[T]$ such that $P(x, y) > 0$, let $\phi_{x, y} = \sum_{i \in S, j \in T}\phi_{ij,xy}$.

(We may assume that $S \cap T = \emptyset$: otherwise, for all $i \in S \cap T$, one may construct a trivial flow $\phi_{ii}$ that sends zero flow across every edge between adjacent pairs of states in $\Omega_i$.)

The functions~$\phi_{S}$ and~$\phi_{T}$ together induce an $\Omega[S]$-$\Omega[T]$ flow (satisfying the axioms in in \cref{sec:mcflow}): this flow simply results from the (edge-wise) sum of all the flows $\Gamma^{i\rightarrow j}$, $i \in S, j \in T$. By \cref{lem:flowtotransportflow} this implies the existence of a transport flow~$\Gamma^{S\rightarrow T}$ from $\Omega[S]$ to $\Omega[T]$.

By \cref{lem:maxcong} the maximum congestion incurred by this flow is at most~$\Delta$.

Let $C, c$ be the constants in \cref{lem:numetastrong}.  We bound $\bar \rho_{\Omega[S]}$ as follows:

\begin{align}
    2\bar\rho_{\Omega[S]} &= \frac{\Delta}{\piorproj(S)}\sum_{x \in \Omega_S, y \in \Omega_S \cup \Omega_T}  |\phi_{xy}| \pi(x)P(x, y) \label{eqn:barrho2}  \\
    &\leq \frac{\Delta}{\piorproj(S)}\sum_{i \in S}\sum_{x \in \Omega_i}\sum_{j \in T}\sum_{y \in \Omega_i \cup \Omega_j}|\phi_{ij,xy}|\pi(x)P(x, y)  \nonumber \\
    &= \frac{\Delta}{\piorproj(S)}\sum_{i \in S}\sum_{j \in T}~\sum_{\eta = (k_1 = i, k_2, \dots, k_d),\atop j \sim \eta}~\sum_{x \in \Omega_{\eta j}, y \in \Omega_{\eta}} |\phi_{ij,xy}| \pi(x)P(x, y) \nonumber \\
    &\quad + \frac{\Delta}{\piorproj(S)}\sum_{i \in S, j \in T}\sum_{x \in \Omega_i, y \in \Omega_j}|\phi_{ij,xy}|\pi(x)P(x, y)  \label{eqn:sumoverj}  \\
    &\leq \frac{C \sqrt n \log^{c} n}{\piorproj(S)}\sum_{i \in S, j \in T}\piorproj(i)\piorproj(j) + \frac{1}{\piorproj(S)}\sum_{i\in S, j \in T}|\phi_{ij,ij}|\piorproj(i)\piorproj_i(j) \label{eqn:subnumetastrong} \\
    &=  \frac{C \sqrt n \log^{c} n}{\piorproj(S)}\sum_{i \in S, j \in T}\piorproj(i)\piorproj(j) + \frac{1}{\piorproj(S)}\sum_{i\in S, j \in T}\piorproj(i)\piorproj(j) \label{eqn:subrhoijij}\\
    &= C \sqrt n \log^c n \piorproj(T)\,,\label{eqn:piT}
\end{align}
where the factor of 2 on the left-hand side of~\eqref{eqn:barrho2} is because by the definition of~$\phi$ we have $\phi_{xy} = -\phi_{yx}$ and also because we have defined $\bar\rho_{\Omega[S]}$ as the average congestion of a transport flow that sends positive flow in at most one direction across an edge $(x, y)$.

We have \eqref{eqn:sumoverj} by \cref{lem:sumoverj}. For~\eqref{eqn:subnumetastrong} we have applied \cref{lem:numetastrong}; we have also applied the definition $\phi_{ij,xy} = \phi_{ij,ij}$, and have used the identity
\[
    \Delta\sum_{x \in \Omega_i, y \in \Omega_j}\pi(x)P(x,y) = \piorproj(i)\piorproj_i(j)\,,
\]
which follows from definitions. We have~\eqref{eqn:subrhoijij} by the definition $\phi_{ij,ij} = \frac{\piorproj(j)}{\piorproj_i(j)}$.  

The claim follows from \eqref{eqn:piT}. 
\end{proof}

We now prove \cref{lem:proj_flow}, the last ingredient in the proof of \cref{thm:lsi} given in \cref{sec:trimixing}: 
\begin{proof}[Proof of \cref{lem:proj_flow}]
It suffices to apply \cref{lem:ldirstrong} with $\Omega^{(t,\eta)}$ substituted for~$\Omega$ and, for each~$k$, with $S$ substituted for $\hat S_k^{(t,\eta)}$, with $T$ substituted for $\Omorproj^{(t,\eta)} \setminus S$, and with $\Omega[S]$ and $\Omega[T]$ substituted for $S$ and $T$ respectively. The result then follows from plugging in the bounds on $\bar \rho_{\Omega[S]}$ and $\bar \rho_{\Omega[T]}$ given by \cref{lem:ldirstrong}.
\end{proof}

\printbibliography

\appendix

\section{Multi-way Single-Commodity Flows (MSFs) and Transport Flows}
\label{sec:msf}
In this section we make formalize some missing details from \cref{sec:transport_flow}.

We also prove \cref{lem:flowreduction}. More precisely, we show that the flow construction of \cite{eppstein2022improved} implies the functional inequalities in \cref{lem:flowreduction}.

First, concerning the definition of transport flows given in \cref{sec:transport_flow}, we justify our assertion that one can assume that a transport flow uses a given transition $(x, y)$ if and only if it does not use $(y, x)$:
\begin{lemma}
\label{lem:onedir}
Given $S, T \subseteq \Omega$, suppose $\Gamma = \Gamma^{S\rightarrow T}$ is a transport flow from $S$ to $T$, and suppose for some $(x, y) \in \Omega^2$ such that $P(x, y) > 0$ there exist distinct paths $\gamma, \gamma'$ such that $(x, y) \in \gamma, (y, x) \in \gamma'$, and $\Gamma(\gamma) > 0$ and $\Gamma(\gamma') > 0$. Suppose without loss of generality that $\Gamma(\gamma) > \Gamma(\gamma')$.

Then $\Gamma$ can be modified to obtain a distribution $\Gamma'$ in which $\Gamma'(\gamma') = 0$, without introducing a new path that uses $(y, x)$. 
\end{lemma}
\begin{proof}
If $(x, y) \in \gamma$ and $(y, x) \in \gamma'$, where $\gamma$ has endpoints $z,w$ and $\gamma'$ has endpoints $z',w'$, then let $\gamma[z, x], \gamma[z', y']$ (respectively $\gamma[y, w], \gamma[x, w']$) denote the segments of the paths $\gamma, \gamma'$ from $z$ to $x$ and $z'$ to $y$ (respectively from $y$ to $w$ and from $x$ to $w'$). Then let $\gamma''$ be the path obtained from concatenating $\gamma[z, x]$ with $\gamma'[x, w']$ and let $\gamma'''$ be obtained from concatenating $\gamma'[z', y]$ with $\gamma[y, w]$. Now define the distribution $\Gamma'$ as follows:
\begin{enumerate}[(i)]
    \item for all paths $\delta \notin \{\gamma, \gamma', \gamma'', \gamma'''\}$ let $\Gamma'(\delta) := \Gamma(\delta)$.
    \item Let $\Gamma'(\gamma') := 0$.
    \item Let $\Gamma'(\gamma) := \Gamma(\gamma) - \Gamma(\gamma')$.
    \item Let $\Gamma'(\gamma'') := \Gamma(\gamma'') + \Gamma(\gamma')$.
    \item Let $\Gamma'(\gamma''') := \Gamma(\gamma''') + \Gamma(\gamma')$
\end{enumerate}
We readily obtain from the definition above that~$\Gamma'$ is a transport flow from $S$ to $T$, that $\gamma'', \gamma'''$ do not use $(y, x)$, and that $\Gamma'(\gamma') = 0$ as desired.

\end{proof}
Applying \cref{lem:onedir} iteratively to a transport flow~$\Gamma$, one obtains a transport flow with the desired property.

Inspired by the Ford-Fulkerson algorithm~\cite{Fordf56}, we also prove \cref{lem:flowtotransportflow}, which allows us to pass from combinatorial flows to transport flows:
\begin{proof}[Proof of \cref{lem:flowtotransportflow}]
For convenience, we modify the construction by adding a special ``source'' vertex (state) $s$ and a special ``sink'' vertex $t$. We add transitions (the probabilities are irrelevant for the present aim) for each $z \in S$ and for each $w \in T$.  

Now we view the $S$-$T$ flow instead as an $s$-$t$ flow, in which we use the additional transitions to route flow rom $s$ to $S$ and $T$ to $t$. Assign capacities to the (directed) edges in the resulting flow network in which the capacity of a (directed) edge $(x, y)$ across which $\phi$ sends positive flow is equal to $\phi(x,y)$. 

We construct the desired transport flow (collection of paths forming the distribution~$\Gamma = \Gamma^{S\rightarrow T}$) in an iterative fashion, by defining a sequence of weight functions $\{w_t \mid t \in [T]\}$ where $T$ is  finite and will be described shortly. For every $z \in S, w \in T$ and for every (simple) path $\gamma$ in $(\Omega, P)$ from $z$ to $w$, let $w_0(\gamma) = 0$.   

We will maintain the invariant that for all $t\geq 0$ and for every edge $(x, y)$,  we will have
\[\sum_{\gamma \ni (x, y)}w_t(\gamma) - \sum_{\gamma \ni (y, x)}w_t(\gamma)\leq \phi(x, y)\]

Given $t \geq 0$, let an edge $(x, y)$ be \emph{tight at iteration} $t$ if $\sum_{\gamma \ni (x, y)}w_t(\gamma) - \sum_{\gamma \ni (y, x)}w_t(\gamma) = \phi(x, y)$. 

The process will stop at time $T$, where  we define $T$ to be the minimum~$t$ such that
for all $z \in S$, 
\[
    \sum_{\gamma \ni (s,z)}w_t(\gamma) = \sum_{z \in S}\sum_{x \sim z}\phi(z, x)
\]
Note that this is equivalent to the condition that for all $w \in T$,
\[
    \sum_{\gamma \ni (w,t)}w_t(\gamma) = \sum_{x \sim w}\phi(x, w)
\]

Now, given $t \geq 0$, if $t = T$ then the weight function~$w_t$ forms a transport flow~$\Gamma$ from $S$ to $T$, by the definition of a transport flow. (Here we drop $s,t$ and their incident edges.)

If $t < T$, then we claim that there exists some $s$-$t$ path $\gamma$ such that no edge in $\gamma$ is tight. Then we can simply increase the weight of some such path until an edge on the path becomes tight, obtaining $w_{t+1}(\gamma) > w_t(\gamma)$.

To see the claim, suppose for a contradiction that every $s$-$t$ path has a tight edge. Then the set of all tight edges in the flow network forms a cut, call it $\mathcal{C}$. By the max-flow-min-cut theorem, the combined weight of the edges in this cut is at least equal to the maximum flow under $\phi$ that $s$ sends to $t$. That is:
\[
    \sum_{(x, y) \in \mathcal{C}}\sum_{\gamma \ni (x, y)}w_t(\gamma) \geq \sum_{z \in S}\sum_{x \sim z}\phi(z, x)
\]
However, we also have
\[
    \sum_{z \in S}\sum_{\gamma \ni (s, z)}w_t(\gamma) = \sum_{\gamma: \gamma \textnormal{is an s-t path}}w_t(\gamma) \geq  \sum_{(x, y) \in \mathcal{C}}\sum_{\gamma \ni (x, y)}w_t(\gamma)
\]
where the first equality is because every $s$-$t$ path $\gamma$ contains at least one edge $(s, z)$ for some $z \in S$. The last inequality is because the right-hand side is a summation over all $s$-$t$ paths containing cut edges, and the left-hand side is a summation over all $s$-$t$ paths.

These two inequalities together imply that 
\[
    \sum_{z \in S}\sum_{\gamma \ni (s, z)}w_t(\gamma) \geq \sum_{z \in S}\sum_{x \sim z}\phi(z, x)
\]
Since by construction the left-hand side never exceeds the right-hand side, this
implies that $t = T$, a contradiction.
\end{proof}

In the remainder of this section, we build the required machinery for proving \cref{lem:flowreduction}.
To this end, we will use the artifice of a \emph{multi-way single-commodity flow (MSF)} introduced and used in~\cite{eppstein2022improved}:

\begin{definition}[\cite{eppstein2022improved}]
Given a graph $G = (V, E)$, define a \emph{multi-way single-commodity flow (MSF) problem} as a tuple
$(S, T, \sigma, \delta)$
where:
\begin{itemize}
    \item $S \subseteq V$ is the set of \emph{source} vertices
    \item $T \subseteq V$ is the set of \emph{sink} vertices
    \item $\sigma: S \rightarrow \mathbb{R}_{\geq 0}$ is the \emph{surplus} function
    \item $\delta: T \rightarrow \mathbb{R}_{\geq 0}$ is the \emph{demand} function
\end{itemize}
We require that 
$\sum_{s \in S} \sigma(s) = \sum_{t \in T} \delta(t)$.

It may be that $\sigma$ (or $\delta$) is a constant function. In this case we will sometimes denote by $\sigma$ (or $\delta$) the value $\sigma(v)$ (or $\delta$).
\label{def:msfproblem}
\end{definition}

\begin{definition}
Given a graph $G = (V, E)$, let $A = \bigcup_{\{u, v\} in E}{(u, v), (v, u)}$ be the set of directed arcs obtained by creating two copies of each (undirected) edge in $E$, one copy in each direction.

Given an MSF problem $\Pi = (S, T, \sigma, \delta)$ where $S, T \subseteq V$, say that a function $\phi: A \rightarrow \mathbb{R}$ \emph{solves} $\Pi$ if $\phi$ satisfies the following conditions:

\begin{itemize}
    \item for all $s \in S \setminus T$, $\sum_{u \in V} \phi(s, u) = \sigma(s)$
    \item for all $t \in T \setminus S$, $\sum_{u \in V} \phi(u, s) = \sigma(T)$
    \item for all $v \in S \cap T$, $\sum_{u \in V} \phi(v, u) = \sigma(v) - \delta(v)$
    \item for all $v \in V \setminus (S \cup T)$, $\sum_{u \in V} \phi(v, u) = 0$
\end{itemize}

Call $\phi$ a \emph{multi-way single-commodity flow (MSF)}, or simply a \emph{flow function}. 
    \label{def:msf}
\end{definition}

\begin{remark}
    To make our proofs more convenient, we have deviated from the definition of an MSF in \cite{eppstein2022improved}, in that we define $\phi(x,y)~=~-\phi(y, x)$, and when $\phi(x, y) \geq 0$ we consider $\phi(x, y)$ to denote the flow routed across the edge $(x, y)$ in the $(x, y)$ direction. We have adjusted the criteria for $\phi$ accordingly in \cref{def:msf}.
    
    \cite{eppstein2022improved} instead considered the flow across the edge to be $\phi(x, y) - \phi(y, x)$ when $\phi(x, y) \geq \phi(y, x)$. 
\end{remark}

\cite{eppstein2022improved} considered the problem of constructing a multicommodity flow in the graph $(\Omega, Q)$ where $\Omega$ is the state space of the triangulation walk, viewed as the vertex set of a graph, and where $Q = \{(x, y) \mid P(x, y) > 0\}$ is the (undirected) edge set obtained by connecting every pair of triangulations $x, y \in \Omega$ for which the transition probability $P(x, y)$ is nonzero. This graph is regular and unweighted (since $\pi$ is uniform and since all nonzero transition probabilities under $P$ are the same), and has degree $\Delta = n - 1$.

\cite{eppstein2022improved} considered a collection of MSF problems $\{\Pi_s \mid s \in \Omproj\}$ where
\[
    \Pi_s = \left(\Omega_s, \Omega, \sigma = 1, \delta = \piproj(s)\right)
\]

\cite{eppstein2022improved} decomposed $\Omega$ recursively according to the projection chain $(\Omproj, \Pproj)$ and the Cartesian product structure of each $\Omega_s$ subspace, as we have done in \cref{sec:trixport}. They showed that to construct a uniform multicommodity flow in $(\Omega, Q)$, it suffices to construct an MSF for every $\Pi_t$ problem, for each projection graph $(\Omproj, \Qproj)$ that occurs at every level of the decomposition.

As our present aim is to prove \cref{lem:flowreduction}, we wish to bound the variance
\begin{align}
    \var_{\piproj} \FF &= \sum_{s \in \Omproj} \piproj(s)(\FF(s) - \Exp_{\piproj} \FF)^2 \nonumber \\
    &= \sum_{s \in \Omproj} \piproj(s)(\FF(s) - \FF(\Omega))^2 \nonumber
\end{align}
where we let $\FF(\Omega) = \Exp_{\piproj} \FF = \Exp_{\pi} f$.

We consider the term $\piproj(s)(\FF(s) - \FF(\Omega)^2$ to represent the problem of distributing $\piproj(s)$ total flow, initially concentrated evenly among the states in $\Omega_s$, throughout all of $\Omega$. 

That is, we associate the expression with the MSF problem $\Pi_s$.

The following lemmas are implicit in~\cite{eppstein2022improved} and immediate from the definitions. These lemmas connect MSFs with transport flows:

\begin{lemma}
    Suppose an MSF problem $\Pi = (S, T, \sigma, \delta)$ ``decomposes'' as a ``serial'' collection of problems
    \[
        \{\Pi_k = (S_k, T_k, \sigma_k, \delta_k) \mid k \in [K]\}
    \]
    for some $K \geq 1$. That is, suppose for $k \in [K - 1]$, $T_k = S_{k+1}$ and $\delta_k = \sigma_{k+1}$. Suppose $S = S_1$ and $T = T_K$, and suppose $\sigma = \sigma_1$ and $\delta = \delta_K$.
 
    Then
    $\sigma \pi(S) = \delta\pi(T) = \sigma_k \pi(S_k) = \delta_k\pi(T_k)$ for all $k$, and 
    \[
        \sigma \pi(S)(f(S) - f(T)) = \sum_k \sigma_k\pi(S_k)(f(S_k) - f(T_k)) = \sum_k \delta_k\pi(T_k)(f(S_k) - f(T_k))
    \]
\end{lemma}

\begin{lemma}
Suppose an MSF problem $(S, T, \sigma, \delta)$ decomposes as a ``parallel,'' disjoint ``composition'' of problems
\[
\{\Pi_k = (S_k, T_k, \sigma_k, \delta_k) \mid k \in [K]\}
\]
for some $K \geq 1$.

That is, suppose $S = \bigcup_{k \in [K]} S_k$ and $T = \bigcup_{k \in [K]} T_k$. Suppose for $k \in [K]$, for all $s \in S_k$, $\sigma_k(s) = \sigma(s)$, and for all $t \in T_k$, $\delta_k(t) = \delta(t)$.

Then 
\[\sigma \pi(S) = \delta\pi(T) = \sum_k \sigma_k\pi(S_k) = \sum_k \delta_k \pi(T_k)\]
\end{lemma}

\begin{corollary}
If an MSF problem decomposes into (parallel, disjoint) sums and (serial) decompositions of problems $\{\Pi_k = (S_k, T_k, \sigma_k, \delta_k)\}$ such that the longest serial ``chain'' in the decomposition has length $L$, then 
\begin{equation}
    \sigma \pi(S)(f(S) - f(T))^2 \leq L\sum_k\sigma_k\pi(S_k)(f(S_k) - f(T_k))^2 \label{eqn:msfdecomp}
\end{equation}
\label{cor:msfdecomp}
\end{corollary}
\begin{proof}
    From the definitions of parallel and serial decompositions, the expression $f(S) - f(T)$ on the left-hand side decomposes into an expectation of chains (or ``paths'') of telescoping differences, where the longest chain has length $L$. We obtain the right-hand side by applying the Cauchy-Schwarz inequality.
\end{proof}

Considering the inequality in \cref{lem:flowreduction}, one can compare the left-hand side of \cref{cor:msfdecomp} with the expression $\piproj(s)(\FF(s) - \FF(\Omega))^2$, letting $\sigma = 1, S = \Omega_s$, and $T = \Omega_t$. Our aim is to show an inequality of the form \eqref{eqn:msfdecomp}, where $S_k = S_k^{(t, \omega)}$ as in \cref{lem:flowreduction}.

For convenience, we will define the additional notion. Consider a collection of MSF problems, $\mathcal{P} = \{P_1, P_2, \dots, P_k\}$. Define the \emph{multicommodity transport flow} (MTF) problem  $\{\mathcal{S}, \mathcal{T}, \{\sigma_i \mid i = 1,\dots, k\}, \{\delta_i \mid i = 1, \dots, k\}\}$ where $\mathcal{S} = \{S_1, \dots, S_k\}, \mathcal{T} = \{T_1, \dots, T_k\}$. We say that a flow function $\phi = \{\phi_i \mid i = 1, \dots, k\}$ solves the MTF $\mathcal{P}$ if $\phi_i$ solves $P_i$ for all $i$.

We can then reason about subproblems (decomposition) of an MTF the way we do about an MSF: suppose MSF problems $P = (S, T, \sigma, \delta), P' = (S', T', \sigma', \delta')$ each split serially into $P_{1} = (S_{1}, T_{1}, \sigma_{1}, \delta_{1}), P_{2} = (S_{2}, T_{2}, \sigma_{2}, \delta_{2}), P'_{1} = (S'_{1}, T'_{1}, \sigma'_{1}, \delta'_{1}), P'_{2} = (S'_{2}, T'_{2}, \sigma'_{2}, \delta'_{2})$ and suppose $T_{1} = S_{2} = T'_{1} = S'_{2}$. Then we consider the MTF problem $\mathcal{P} = \{P, P'\}$ to be \emph{equivalent} to the \emph{collection} of MTF problems $\mathcal{Q} = \{P_1, P'_1\}, \mathcal{R} = \{P_2, P'_2\}$.

That is:
\begin{remark}
    If an MTF problem~$\mathcal{P}$ can be decomposed as in \cref{cor:msfdecomp}, then the inequality in \cref{cor:msfdecomp} holds for~$\mathcal{P}$.  
    \label{rmk:mtfdecomp}
\end{remark}

\begin{lemma}[\cite{eppstein2022improvedfull}]
\label{lem:efmtffull}
Consider an MTF problem $\mathcal{P}$ (over the state space $\Omega$) induced by $|\Omproj|$ MSF problems $\{P_t \mid t \in \Omproj\}$ in which each central-triangle-induced class $t \in \Omproj$ begins with a surplus $\sigma = 1$, i.e.~$\mathcal{P}_t = \{S = \Omega_t, T = \Omega, \sigma = 1, \delta = \piproj(t)\}$. There exists an MTF $\phi$ that solves all of the problems $\{\mathcal{P}_t\}$ in which each edge in $\Pproj$ carries total congestion at most $\widetilde O(\sqrt n)$.

The maximum path length of $\phi$ is $\widetilde O(1)$.

Furthermore, $\mathcal{P}$ is equivalent to a collection of MTF subproblems each of which is of one of two forms:
\begin{enumerate}[(i)]
\item \label{case:across} $\mathcal{P}_{t,t'} = (\Omega_{tt'}, \Omega_{t't}, \sigma_{tt'}, \delta_{tt'})$ where $\sigma_{tt'} = \widetilde O(\sqrt n)$.
\item \label{case:internal} $\mathcal{P}_{t,k} = \{(S_k^{(t, \eta)}, \Omega_t, \sigma_k, \delta_k)\}$ or $\mathcal{P}_{t,k} = \{(\Omega_t, S_k^{(t,\eta)}, \sigma_k, \delta_k)\}$ where $S_k^{(t, \eta)} = \bigcup_{i \in S} \Omega_{ti}^{(\eta)}$ for some $S \subseteq \Omorproj^{(t,\eta)}$, where $\eta \in \Omega_t^{(1)} \times \Omega_t^{(2)}$ or $\eta \in \Omega_t^{(2)} \times \Omega_t^{(3)}$ or $\eta \in \Omega_t^{(1)} \times \Omega_t^{(3)}$ and where $\sigma_k = \widetilde O(\sqrt n)$. 
\end{enumerate}
Each $t$ participates in at most $\widetilde O(1)$ of the subproblems.
\end{lemma}
\begin{proof}
    The series of lemmas in \cite[Appendix B]{eppstein2022improvedfull} describes a collection of MSF problems (essentially an MTF problem) as described in the statement of this lemma, and establishes a flow construction with the congestion claimed. The construction decomposes~$\mathcal{P}$ into subproblems of the two forms claimed: this is immediate from the proofs of \cite[Lemma 32, Lemma 38, Corollary 39, Lemma 40, and Lemma 41]{eppstein2022improvedfull}.
\end{proof}

By \cref{rmk:mtfdecomp}, combining \cref{lem:efmtffull} (in particular the bounds on the congestion and the maximum path length) with \cref{cor:msfdecomp} implies \cref{lem:flowreduction}.

\section{Proofs of Lemmas in \cref{sec:avgcond}}
\label{sec:deferredproofs}
\begin{lemma}
    In the construction in the proof of \cref{lem:l1dirweak}, given any $\eta = (k_1 = i, k_2, \dots, k_{d-2}, k_{d-1}, k_d)$ where $j \in \eta$, and any $s, t \in \Omorproj_{\eta,r}$ we have, if $k_d < s < t$ then either $0 \leq \phi_{ij,\eta s} \leq \phi_{ij,\eta t}$ or $0 \geq \phi_{ij,\eta s} \geq \phi_{ij,\eta t}$.

    Similarly if $s, t \in \Omorproj_{\eta, l}$ and if $t < s < k_d$ then either $0 \leq \phi_{ij,\eta s} \leq \phi_{ij,\eta t}$ or $0 \geq \phi_{ij,\eta s} \geq \phi_{ij,\eta t}$.

    \label{lem:rhomono}
\end{lemma}
\cref{lem:rhomono} states that the congestion is monotonic as one travels around the polygon. See \cref{fig:rhomonost}.
\begin{figure}
    \centering
    \includegraphics[width=0.5\linewidth]{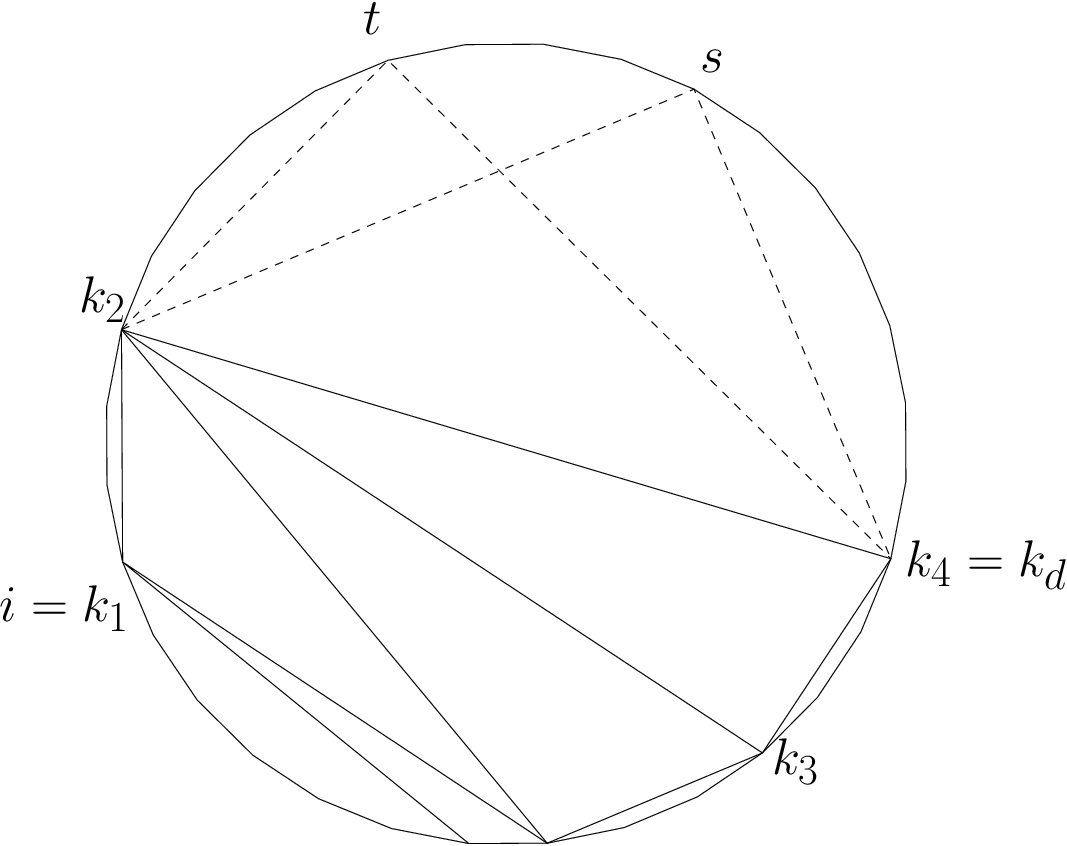}
    \caption{An illustration of \cref{lem:rhomono}: in this example, $\eta = (k_1, k_2, k_3, k_4)$; we have $s, t \in \Omega_{\eta,r}$ and $t < s < k_4$; and therefore either $0 \leq \phi_{ij,\eta s} \leq \phi_{ij,\eta t}$ or $0 \geq \phi_{ij,\eta s} \geq \phi_{ij,\eta t}$.}
    \label{fig:rhomonost}
\end{figure}

\begin{proof}[Proof of \cref{lem:rhomono}]
    We proceed by induction on $d$. Given $j \in T$, for $\eta = (i, j)$ we have 
    
    \[
        \phi_{ij,ijs} =  \frac{\piorproj_i(j)}{\piorproj_{is}(j)}(\phi_{ij,is} - \phi_{ij,ij}) = -\frac{\piorproj_i(j)}{\piorproj_{is}(j)}\phi_{ij,ij} = -\frac{\piorproj_i(s)}{\piorproj_{ij}(s)}\phi_{ij,ij}
    \]
        
     (where we have used an analogue of~\eqref{eqn:rhodistmat} for the last equality) and similarly $\phi_{ij,ijt} = -\frac{\piorproj_i(t)}{\piorproj_{ij}(t)}\phi_{ij,ij}$. The claim then follows for $\phi_{ij,ijs}, \phi_{ij,ijt}$ from \cref{lem:pimono}.

    Also $\phi_{ij,iks} = \phi_{ij,ikt} = 0$ whenever $k \neq j$.

    Now suppose $\eta = (k_1, k_2, \dots, k_{d-1}, k_d)$ and $\eta' = (k_1, k_2, \dots, k_{d-1}) = \eta \setminus \{k_d\}$, and $\eta'' = (k_1, k_2, \dots, k_{d-2}) = \eta' \setminus \{k_{d-1}\}$.

    Then 
    \begin{align}
        \phi_{ij,\eta s} = \frac{\piorproj_{\eta'}(s)}{\piorproj_{\eta}(s)}(\phi_{ij,\eta' s} - \phi_{ij,\eta})
        \label{eqn:rhomonoind}
    \end{align}
    and $\phi_{ij,\eta t}$ is similar. By the inductive hypothesis either 
    
    \begin{align}
        \phi_{ij,\eta' t} \geq \phi_{ij,\eta' s} \geq \phi_{ij,\eta' k_d} = \phi_{ij,\eta} \geq  0 \textnormal{ or } \phi_{ij,\eta} = \phi_{ij,\eta' k_d} \geq \phi_{ij,\eta' s} \geq \phi_{ij,\eta' t}  \geq  0 \label{eqn:userhomono}
    \end{align} (or else the same holds with the terms negated and the inequalities reversed).
    By \cref{lem:pimono}
    \begin{align}
        \frac{\piorproj_{\eta'}(t)}{\piorproj_{\eta}(t)}
        \geq \frac{\piorproj_{\eta'}(s)}{\piorproj_{\eta}(s)}
        \geq 0 \label{eqn:usepimono}
    \end{align}

    Therefore combining \eqref{eqn:usepimono} and \eqref{eqn:userhomono}, and using the definition of $\phi_{ij,\eta s}$ and $\phi_{ij,\eta t}$ (i.e.~\eqref{eqn:rhomonoind}) we obtain $\phi_{ij,\eta t} \geq \phi_{ij,\eta s} \geq 0$ or else $\phi_{ij,\eta t} \leq \pi_{ij,\eta s} \leq 0$.
\end{proof}

\begin{proof}[Proof of \cref{lem:rhodecomp}]
    Let $\Omorproj_{\eta', l} = \{i_l, i_l + 1, \dots, k_{d+1} - 1\}$ and $\Omorproj_{\eta', r} = \{k_{d+1} + 1, \dots, k_{d} - 1\}$ where $i_l$ is the leftmost element of $\Omorproj_{\eta', l}$ and is determined by $\eta$.

    Then we have
    \begin{equation}
        \phi_{ij,l}^*(\eta) = |\phi_{ij,\eta i_l}| \label{eqn:rholeta}
    \end{equation}
    which can be seen to follow from \cref{lem:pimono} and \cref{lem:rhomono}.

    For the same reason
    \begin{align}
        \phi_{ij,l}^*(\eta') &= |\phi_{ij, \eta' i_l}| \label{eqn:rholetaprime}\\
        \phi_{ij,r}^*(\eta') &= |\phi_{ij, \eta' k_{d}-1}| \label{eqn:rhoretaprime}
    \end{align}

    Furthermore, by the definition of $\phi_{\cdot, \cdot}$:
    \begin{align}
        \phi_{ij,\eta' i_l} &= \frac{\piorproj_{\eta}(i_l)}{\piorproj_{\eta k_{d+1}}(i_l)}(\phi_{ij,\eta i_l} - \phi_{ij, \eta k_{d+1}}) \label{eqn:rhodiff1}\\
        \phi_{ij,\eta' k_d-1} &= \frac{\piorproj_{\eta}(k_d-1)}{\piorproj_{\eta k_{d+1}}(k_d-1)}(\phi_{ij,\eta k_d-1} - \phi_{ij, \eta k_{d+1}}) \label{eqn:rhodiff2}
    \end{align}

    Then by \eqref{eqn:rhodiff1} and \eqref{eqn:rhodiff2} and another application of \cref{lem:pimono} and \cref{lem:rhomono}:
    \begin{align}
        |\phi_{ij,\eta' i_l}| + |\phi_{ij,\eta' k_d - 1}| &= |\phi_{ij,\eta' i_l} - \phi_{ij, \eta' k_d - 1}| \label{eqn:sumabs}\\
        &=   \left|\frac{\piorproj_{\eta}(i_l)}{\piorproj_{\eta k_{d+1}}(i_l)}(\phi_{ij,\eta i_l} - \phi_{ij, \eta k_{d+1}}) + \frac{\piorproj_{\eta}(k_d-1)}{\piorproj_{\eta k_{d+1}}(k_d-1)}(\phi_{ij, \eta k_{d+1}} - \phi_{ij,\eta k_d-1})\right| \nonumber \\
        &\leq  \left|(\phi_{ij,\eta i_l} - \phi_{ij, \eta k_{d+1}}) + (\phi_{ij, \eta k_{d+1}} - \phi_{ij,\eta k_d-1})\right| \label{eqn:removepi}\\
        &\leq |\phi_{ij, \eta i_l}| \label{eqn:sumrho}
    \end{align}
    where \eqref{eqn:sumabs} is because, by \cref{lem:rhomono}, $\phi_{ij,\eta' i_l}$ and $\phi_{ij,\eta' k_d-1}$ have opposite signs, and where \eqref{eqn:removepi} is by \cref{lem:pigood}.

    The claim follows by combining \eqref{eqn:sumrho} with \eqref{eqn:rholeta}, \eqref{eqn:rholetaprime}, and \eqref{eqn:rhoretaprime}.
\end{proof}
\begin{proof}[Proof of \cref{lem:maxcong}]
This analysis will closely follow \cite{eppstein2022improved}.

We will prove by induction on $|\eta|$ that for all $\eta,$ $|\phi_{\eta}| \leq 1$, implying the result\textemdash where we define $\phi_{\eta}:= \sum_{i\in S, j\in T}\phi_{ij,\eta}$.

It will be helpful to observe that $\phi_{\cdot,\cdot}$ satisfies the axioms of a flow function. We wish to show that at most one unit of flow is sent across any $x,y$ edge.

Observe first that for all $\eta = i \in \Omorproj$, and for all $s \in \Omorproj_{i,l}$ (similarly $s \in \Omorproj_{i,r}$) each triangulation  in $\Omega_{\eta s} = \Omega_{is}$ can be viewed as having at most one unit of congestion that it may have received from a neighboring state in some $\Omega_{si}$ (if $s \in T$), and which it must distribute throughout $\Omega_{\eta}$, i.e.
\[
    \phi_{\eta s} = \sum_{j \in T}\phi_{ij,\eta s} = \phi_{s,\eta s} 
\]
and therefore $|\phi_{\eta s}| \leq 1$.

Then for all $s, t \sim \eta$ we have
\begin{align}
    \phi_{\eta s t} &= \sum_{j \in T}\phi_{ij, \eta s t} \\
    &= \sum_{j \in T} \frac{\piorproj_{\eta}(t)}{\piorproj_{\eta s}(t)}(\phi_{ij,\eta t} - \phi_{ij, \eta s})  \\
    &= \frac{\piorproj_{\eta}(t)}{\piorproj_{\eta s}(t)}(\sum_{j \in T}\phi_{ij,\eta t} - \sum_{j \in T} \phi_{ij, \eta s}) \\
    &= \frac{\piorproj_{\eta}(t)}{\piorproj_{\eta s}(t)}(\phi_{\eta t} - \phi_{\eta s})
\end{align}
and applying \cref{lem:pigood} and \cref{lem:rhomono} implies $|\phi_{\eta s t}| \leq 1$.
\end{proof}
\end{document}